\documentclass[12pt]{amsart}

\usepackage{stmaryrd}
\usepackage{amsmath}
\usepackage{amscd}
\usepackage{amssymb}
\usepackage{enumerate}
\usepackage{amsfonts}
\usepackage{graphicx}
\usepackage[all]{xy}
\usepackage{mathrsfs}
\usepackage[hypertex]{hyperref}

\usepackage{bbm}

\newcommand{\limto}{{\displaystyle\lim_{\longrightarrow}}}
\newcommand{\rightlim}{\mathop{\limto}}

\newcommand{\leftlim}{\mathop{\displaystyle\lim_{\longleftarrow}}}
\newcommand{\limfromn}{\leftlim\limits_{\raise3pt\hbox{$n$}}}
\newcommand{\limton}{\rightlim\limits_{\raise3pt\hbox{$n$}}}

\newcommand{\rightlimit}[1]{\mathop{\lim\limits_{\longrightarrow}}\limits%
                   _{\raise3pt\hbox{$\scriptstyle #1$}}}
\newcommand{\leftlimit}[1]{\mathop{\lim\limits_{\longleftarrow}}\limits%
                   _{\raise3pt\hbox{$\scriptstyle #1$}}}

\numberwithin{equation}{section}

\newcommand{\rar}[1]{\stackrel{#1}{\longrightarrow}}

\newcommand{\isom}{\rar{\simeq}}

\newcommand{\into}{\hookrightarrow}

\newcommand{\al}{\alpha}
\newcommand{\be}{\beta}
\newcommand{\ga}{\gamma}
\newcommand{\Ga}{\Gamma}

\newcommand{\De}{\Delta}
\newcommand{\la}{\lambda}

\newcommand{\eps}{\epsilon}
\newcommand{\sg}{\sigma}

\newcommand{\te}{\theta}

\newcommand{\vp}{\varphi}

\newcommand{\bA}{{\mathbb A}}

\newcommand{\bF}{{\mathbb F}}
\newcommand{\bG}{{\mathbb G}}

\newcommand{\bN}{{\mathbb N}}

\newcommand{\bQ}{{\mathbb Q}}

\newcommand{\bZ}{{\mathbb Z}}

\newcommand{\cE}{{\mathcal E}}
\newcommand{\cF}{{\mathcal F}}

\newcommand{\cH}{{\mathcal H}}

\newcommand{\cJ}{{\mathcal J}}

\newcommand{\cL}{{\mathcal L}}
\newcommand{\cM}{{\mathcal M}}

\newcommand{\cO}{{\mathcal O}}

\newcommand{\cR}{{\mathcal R}}

\newcommand{\cW}{{\mathcal W}}

\newcommand{\sM}{{\mathscr M}}
\newcommand{\sN}{{\mathscr N}}

\newcommand{\fP}{{\mathfrak P}}

\newcommand{\fp}{{\mathfrak p}}

\newcommand{\abs}[1]{\lvert #1\rvert}

\newcommand{\End}{\operatorname{End}}
\newcommand{\Hom}{\operatorname{Hom}}

\newcommand{\Aut}{\operatorname{Aut}}
\newcommand{\Spec}{\operatorname{Spec}}

\newcommand{\id}{\operatorname{id}}

\newcommand{\pr}{\mathrm{pr}}
\newcommand{\ad}{\operatorname{ad}}

\newcommand{\Ind}{\operatorname{Ind}}

\newcommand{\tr}{\operatorname{tr}}
\newcommand{\Tr}{\operatorname{Tr}}
\newcommand{\Nm}{\operatorname{Nm}}
\newcommand{\Nrd}{\operatorname{Nrd}}

\newcommand{\tens}{\otimes}

\newcommand{\st}{\,\big\vert\,}

\newcommand{\sbr}{\smallbreak}
\newcommand{\mbr}{\medbreak}

\newcommand{\cJt}{\tilde{\cJ}}
\newcommand{\Jt}{\tilde{J}}

\newcommand{\FF}{\mathbb{F}}

\newtheorem{thm}{Theorem}[section]
\newtheorem{cor}[thm]{Corollary}
\newtheorem{lem}[thm]{Lemma}

\newtheorem{prop}[thm]{Proposition}

\newtheorem*{thma}{Theorem A}
\newtheorem*{thmb}{Theorem B}
\newtheorem*{thmc}{Theorem C}

\theoremstyle{remark}
\newtheorem{rem}[thm]{Remark}
\newtheorem{rems}[thm]{Remarks}
\newtheorem{example}[thm]{Example}

\newcommand{\Fr}{\operatorname{Fr}}

\newcommand{\Gal}{\operatorname{Gal}}

\newcommand{\Lie}{\operatorname{Lie}}

\newcommand{\Nilp}{\operatorname{Nilp}}

\newcommand{\diag}{\operatorname{diag}}

\newcommand{\ql}{\overline{\bQ}_\ell}
\newcommand{\qls}{\overline{\bQ}_\ell^\times}

\newcommand{\Mod}{\!-\!\operatorname{mod}}

\newcommand{\bfq}{\overline{\bF}_q}
\newcommand{\fqn}{\bF_{q^n}}

\newcommand{\rec}{\operatorname{rec}}

\newcommand{\floor}[1]{\lfloor #1\rfloor}
\newcommand{\ceil}[1]{\lceil #1\rceil}

\newcommand{\QQ}{{\overline{\mathbb{Q}}}}
\newcommand{\MM}{{\mathcal M}}
\newcommand{\from}{\colon}
\newcommand{\C}{C}
\newcommand{\powerseries}[1]{\llbracket #1 \rrbracket}
\newtheorem{defn}[thm]{Definition}
\newtheorem{rmk}[thm]{Remark}
\newcommand{\Adic}{\operatorname{Adic}}
\newcommand{\Sets}{\operatorname{Sets}}
\newcommand{\Spf}{\operatorname{Spf}}
\newcommand{\Nil}{\operatorname{Nil}}
\newcommand{\ExtRig}{\operatorname{ExtRig}}
\newcommand{\QuasiLog}{\operatorname{QuasiLog}}
\newcommand{\qlog}{\operatorname{qlog}}
\newcommand{\Spa}{\operatorname{Spa}}
\newcommand{\CAff}{\operatorname{CAff}}
\newcommand{\perf}{{\operatorname{perf}}}
\newcommand{\Cont}{\operatorname{Cont}}

\newcommand{\ab}{{ab}}
\newcommand{\nr}{{nr}}

\newcommand{\gp}{\fp}
\newcommand{\gP}{\fP}

\newcommand{\U}{\mathbf{U}}
\newcommand{\Z}{\mathbb{Z}}
\newcommand{\Q}{\mathbb{Q}}

\newcommand{\injects}{\hookrightarrow}

\newcommand{\set}[1]{\left\{ #1 \right\}}
\newcommand{\tatealgebra}[1]{\left< #1 \right>}

\begin{document}

\title[Maximal varieties and the LLC]{Maximal varieties and the local Langlands correspondence for $GL(n)$}

\author[M.~Boyarchenko and J.~Weinstein]{Mitya Boyarchenko and Jared Weinstein}

\begin{abstract} The cohomology of the Lubin-Tate tower is known to realize the local Langlands correspondence for $GL(n)$ over a nonarchimedean local field.  In this article we make progress towards a purely local proof of this fact.  To wit, we find a family of open affinoid subsets of Lubin-Tate space at infinite level, whose cohomology realizes the local Langlands correspondence for a broad class of supercuspidals (those whose Weil parameters are induced from an unramified degree $n$ extension).  A key role is played by a certain variety $X$, defined over a finite field, which is ``maximal" in the sense that the number of rational points of $X$ is the largest possible among varieties with the same Betti numbers as $X$.  The variety $X$ is derived from a certain unipotent algebraic group, in an analogous manner as Deligne-Lusztig varieties are derived from reductive algebraic groups.
\end{abstract}

\maketitle

\setcounter{tocdepth}{1}

\tableofcontents


\section*{Introduction}

Let $K$ be a nonarchimedean local field with ring of integers $\cO_K$ and residue field $\bF_q$, and let $n\geq 1$ be an integer.  The local Langlands correspondence and the Jacquet-Langlands correspondence for $GL_n(K)$ are both realized in the $\ell$-adic cohomology of the Lubin-Tate tower $\MM_{H_0,\infty}=\varprojlim\MM_{H_0,m}$, where $\MM_{H_0,m}$ is the rigid analytic space parameterizing deformations of a fixed one-dimensional formal $\cO_K$-module $H_0$ of height $n$ over $\overline{k}$ together with a Drinfeld level $m$ structure.  (For the precise statement, see the introduction to~\cite{HarrisTaylor:LLC}.)   At present, this fact can only be proved using global methods.  A program initiated by the second author in~\cite{WeinsteinGoodReduction} and ~\cite{WeinsteinSemistableModels} aims to obtain a purely local proof by first constructing a sufficiently nice model of $\MM_{H_0,m}$, and then computing the cohomology of $\MM_{H_0,\infty}$ using the nearby cycles complex on the special fiber of this model.

\mbr

This idea has roots in the work of T. Yoshida~\cite{yoshida}, who found an open affinoid in $\MM_{H_0,1}$ whose reduction turned out to be a certain Deligne-Lusztig variety for the group $GL_n$ over $k$.  Using this affinoid, Yoshida showed by purely local methods that the local Langlands correspondence for {\em depth zero} supercuspidal representations of $GL_n(K)$ is realized in the cohomology of $\MM_{H_0,1}$.   Our work is concerned with a large class of supercuspidals of positive depth.  Instead of working with any particular layer $\MM_{H_0,m}$ of the tower, we work directly with $\MM_{H_0,\infty}$, which carries the structure of a perfectoid space (see \cite{WeinsteinSemistableModels}, or \cite{ScholzeWeinstein} for the case of general Rapoport-Zink spaces).  The following theorem does not require any global methods.

\begin{thma}  Let $C$ be the completion of an algebraic closure of $K$.  Let $m\geq 1$ be an integer.  There exists an open affinoid subset $V\subset \MM_{H_0,\infty,C}$, depending on $m$, which is invariant under the action of $GL_n(K)\times D^\times\times\cW_K$ (here $D/K$ is the central division algebra of invariant $1/n$ and $\cW_K$ is the Weil group of $K$), having the following property.  The reduction $\overline{V}$ of $V$ is a scheme over $\overline{\bF}_q$ which also admits an action of $GL_n(K)\times D^\times\times \cW_K$.  For every irreducible admissible representation $\pi$ of $GL_n(K)$ with $\overline{\bQ}_\ell$ coefficients, the following are equivalent:
\begin{enumerate}
\item $\Hom_{GL_n(K)}\left(\pi, H^{n-1}_c(\overline{V},\overline{\bQ}_\ell)\right) \neq 0$.
\item Up to twisting by a one-dimensional character, the Weil parameter of $\pi$ takes the form $\Ind_{L/K} \theta$, where $L/K$ is the unramified extension of degree $n$, and $\theta\from L^\times\to\overline{\bQ}_\ell^\times$ is a character of conductor $m+1$, whose conductor cannot be lowered through twisting by a character of the form $\chi\circ N_{L/L'}$, where $L'\subset L$ is a proper subextension and $N_{L/L'}\from L^\times\to (L')^{\times}$ is the norm map.  (This condition implies that $\pi$ is supercuspidal.)
\end{enumerate}
If these conditions hold then we have a $D^\times\times \cW_K$-linear isomorphism
\[\Hom_{GL_n(K)}\left(\pi,H^{n-1}_c(\overline{V},\overline{\bQ}_\ell)\right)\isom\check{\pi}'\otimes\sigma^{\sharp}(\pi),\]
where $\pi'$ corresponds to $\pi$ under the local Jacquet-Langlands correspondence, $\check{\pi}'$ is the contragredient of $\pi'$ and $\pi\mapsto \sigma^\sharp(\pi)$ is a certain normalization (see Theorem C in \S\ref{ss:Theorem-C}) of the local Langlands correspondence.
\end{thma}

In other words, the degree $n-1$ cohomology of the scheme $\overline{V}$ manifests the Jacquet-Langlands and local Langlands correspondences in its middle cohomology for all supercuspidals of the type described in the theorem.  $\overline{V}$ is not of finite type, but it is closely related to a smooth affine variety $X$ over $\bF_{q^n}$ of dimension $n-1$.  To wit, $\overline{V}$ is the inverse limit of a tower of schemes, each of which is isomorphic to a disjoint union of copies of the perfection of $X\otimes\overline{\bF}_q$.  (The perfection of a scheme in characteristic $p$ is the inverse limit of the scheme under the absolute Frobenius endomorphism.)

\mbr

The variety $X$ is rather interesting in its own right.  It is derived from a certain unipotent group $\U$ over $\bF_{q^n}$ in a manner which resembles certain constructions in the Deligne-Lusztig theory for reductive groups over finite fields (see, e.g., \cite[Def.~1.17(ii)]{deligne-lusztig}).   In fact $X$ is the preimage under the Lang map $x\mapsto \Fr_{q^n}(x)x^{-1}$ of a certain subvariety $Y\subset \U$.  (Here $\Fr_{q^n}$ is the $q^n$th power Frobenius map.  See \S\ref{theunipotentgroup} for the definitions of $\U$ and $Y$.)  Then $X$ admits an action of $\U(\bF_{q^n})$ by right multiplication.  In the course of proving Theorem A we give a complete description of the $\ell$-adic cohomology of $X$.  The theorem below gives a summary of our results.

\begin{thmb}  As a representation of $\U(\bF_{q^n})$, the space $\bigoplus_{i\geq 0} H^i_c(X\otimes\overline{\bF}_q,\overline{\bQ}_\ell)$ decomposes into a direct sum of irreducible representations, each occurring with multiplicity one.  Furthermore, for each $i$, $\Fr_{q^n}$ acts on $H^i_c(X\otimes\overline{\bF}_q,\overline{\bQ}_\ell)$ as the scalar $(-1)^{i-1}q^{ni/2}$.
\end{thmb}

We remark that in the context of the theorem, if $q^{ni/2}$ is not an integer then $H^i_c(X\otimes\overline{\bF}_q,\overline{\bQ}_\ell)=0$. A more precise version of Theorem B is Theorem \ref{t:cohomology-of-X}, whose proof occupies Part 2 of the article.

\mbr

The scalar $(-1)^{i-1}q^{ni/2}$ is significant because it implies that $X$ is a {\em maximal variety} in the following sense.  Let $S$ be any scheme of finite type over a finite field $\bF_Q$.  It follows from~\cite{De80}, Thm. 3.3.1, that for each $i$ and every eigenvalue $\alpha$ of $\Fr_Q$ acting on $H^i_c(S,\QQ_\ell)$, there exists an integer $m\leq i$ such that all complex conjugates of $\alpha$ have absolute value $Q^{m/2}$.  So the Grothendieck-Lefschetz trace formula
\[ \#S(\bF_Q) = \sum_{i\in\bZ} (-1)^i\tr\left(\Fr_Q,\; H_c^i(S,\QQ_\ell)\right) \]
implies the following bound on the number of rational points of $S$:
\[ \#S(\bF_Q) \leq \sum_{i\in\bZ} Q^{i/2}\dim H_c^i(S,\QQ_\ell). \]
This bound is achieved if and only if $\Fr_Q$ acts on $H^i_c(S,\QQ_\ell)$ via the scalar $(-1)^iQ^{i/2}$ for each $i$, in which case the scheme $S$ is called {\em maximal}.  (There are plenty of references in the literature to ``maximal curves" over finite fields:  these are smooth projective curves which attain the Hasse-Weil bound on the number of rational points.  As far as we know, our definition of maximality for arbitrary schemes over a finite field is new.)  Theorem B implies that $X$ is a maximal variety over $\bF_Q$.  We remark that if $n=2$ then $X$ is a disjoint union of $q$ copies of the ``Hermitian curve" $y^q+y=x^{q+1}$, which has long been known to be maximal over $\bF_{q^2}$.

\subsection*{Acknowledgements} We are grateful to Vladimir Drinfeld for suggesting an idea that allowed us to significantly clarify our proof of Proposition \ref{p:reduced-norm-key}, and to Guy Henniart for teaching us about the methods used in his article \cite{Henniart-JLC-I}, which we adopted for our proof of Theorem A.

\section{Outline of the paper}
Part 1 investigates the geometry of $\mathcal{M}_{H_0,\infty}$, the Lubin-Tate space at infinite level.  While this is too large to be a rigid space, it is possible to formulate $\mathcal{M}_{H_0,\infty}$ as a moduli problem on the category of adic spaces (Defn. \ref{MH0definition}) which turns out to be representable.  In \S\ref{s:LT-tower} we review a result from \cite{WeinsteinSemistableModels} which furnishes a linear-algebra description of $\mathcal{M}_{H_0,\infty}$.  We begin with the formal $\cO_K$-module $H_0/\overline{\FF}_q$ of dimension 1 and height $n$.  Let $H$ be any lift of $H_0$ to $\cO_{\breve{K}}$, where $\breve{K}$ is the completion of the maximal unramified extension of $K$.  We consider the ``universal cover" $\tilde{H}=\varprojlim H$ (inverse limit with respect to multiplication by a uniformizer of $K$) as a $K$-vector space object in the category of formal schemes over $\cO_{\breve{K}}$.  Then $\tilde{H}$ does not depend on the choice of lift $H$, and as a formal scheme we have $\tilde{H}\isom \Spf\cO_{\breve{K}}\powerseries{T^{1/q^\infty}}$.  Passing to the generic fiber, one has $\tilde{H}_{\eta}$, whose underlying space is a ``perfectoid open ball" over $\breve{K}$.  The level structure on the universal deformation of $H_0$ over $\cM_{H_0,\infty}$ induces a morphism $\cM_{H_0,\infty}\to \tilde{H}^n_{\eta}$.  Then Thm. \ref{LTdiagram} shows that $\cM_{H_0,\infty}$ is equal to the locally closed subspace of $\tilde{H}^n_{\eta}$ cut out by a certain explicit determinant condition.  If $C$ is the completion of an algebraic closure of $K$, then the base change $\cM_{H_0,\infty,C}$ is a perfectoid space admitting an action of the triple product group $GL_n(K)\times D^\times\times \cW_K$.

In \S\ref{s:affinoid-LT-tower}, we construct a special open affinoid subset of $\cM_{H_0,\infty,C}$ which plays a role in Theorem A.  First we study the CM points of $\cM_{H_0,\infty}$ in \S\ref{CMpoints}.  Let $x\in\cM_{H_0,\infty}(C)$ be a CM point which corresponds to a deformation of $H_0$ with endomorphisms by $L$ (where $L$ is as in Theorem A).  In \S\ref{affinoiddefinition} we identify a descending sequence of open affinoid neighborhoods $\mathcal{Z}_{x,1}\supset \mathcal{Z}_{x,2}\supset\cdots$ of $x$ in $\cM_{H_0,\infty,C}$.  Fix $x$ and $m$, and set $\mathcal{Z}=\mathcal{Z}_{x,m}$.  The main theorem of Part 1 is Thm. \ref{existenceofaffinoid}, which shows that the special fiber $\overline{\mathcal{Z}}$ of $\mathcal{Z}$ is a union of copies of the perfection of the variety $X$ described in the introduction.  We also compute the stabilizer $\mathcal{J}$ of $\mathcal{Z}$ in $GL_n(K)\times D^\times\times \cW_K$, along with its action on $\overline{\mathcal{Z}}$;  this action is induced from an action of $\mathcal{J}$ on the variety $X$ itself.

This being done, we let $V\subset\cM_{H_0,\infty,C}$ be the union of the translates of $\mathcal{Z}$ under $GL_n(K)\times D^\times\times \cW_K$.  We
find that the representation of $GL_n(K)\times D^\times\times \cW_K$ on $H^{n-1}_c(\overline{V},\ql)$ is (modulo a small issue involving twists) isomorphic to the induced representation
$\Ind_{\cJ}^{GL_n(K)\times D^\times\times\cW_K}H^{n-1}_c(X\tens\bfq,\ql)$.  To prove Theorem A, we must therefore calculate the cohomology of $X$, and also show that this induced representation realizes the local Langlands correspondences.  These are the aims of Parts 2 and 3, respectively.

Part 2 is devoted to the proof of a more precise version of Theorem B, which we state as Theorem \ref{t:cohomology-of-X}. Namely, we obtain an explicit description of the irreducible representations of the finite group $U=\U(\fqn)$ that appear in $H^\bullet_c(X\tens\bfq,\ql)$. In particular, we show that for every character $\psi$ of the center of $U$, there exists a unique irreducible representation of $U$ that is both a summand of $H^\bullet_c(X\tens\bfq,\ql)$ and has central character $\psi$. \S\ref{s:DL-formulation} is devoted to setting up the notation for the statement and the proof of Theorem \ref{t:cohomology-of-X}. \S\ref{s:reduced-norm} establishes some properties of the ``reduced norm map" $\U\to\mathbb{G}_a$, which is a geometric version of the reduced norm map $D^\times\to K^\times$ and of the usual determinant $GL_n(K)\to K^\times$.  The heart of the proof of the main theorem is in \S\ref{s:proof-Thm-B}.  A step-by-step outline of the argument can be found in \S\ref{ss:outline-new}. Since the proof of Theorem \ref{t:cohomology-of-X} is somewhat long and complicated, let us summarize the key underlying ideas. First, using the definition of $X$, it is not hard to show that for every representation $\rho$ of $U$ over $\ql$, the ``$\rho$-isotypic component'' of $H^\bullet_c(X\tens\bfq,\ql)$ can be naturally identified with $H^\bullet_c(Y\tens\bfq,\cE_{\rho})$, where $\cE_\rho$ is the local system on $\U$ associated to $\rho$ (see Corollary \ref{c:hom-rep-into-cohomology}). One then reduces the proof of Theorem \ref{t:cohomology-of-X} to the calculation of $H^\bullet_c(Y\tens\bfq,\cE_{\rho})$ for certain representations $\rho$ that can be induced from $1$-dimensional representations of groups of $\bF_{q^n}$-points of connected subgroups of $\U$. Using the methods of \cite[\S2]{DLtheory}, one identifies $H^\bullet_c(Y\tens\bfq,\cE_{\rho})$ with the cohomology of certain rank $1$ local systems on affine spaces. Finally, the latter turn out to be amenable to an inductive calculation using certain linear fibrations of affine spaces $\bA^d\rar{}\bA^{d-1}$, the proper base change theorem and the projection formula.

Finally, Part 3 connects the first two parts by using Theorem B to prove Theorem A. We note that Theorem A is concerned with the Lubin-Tate tower of $K$ and certain special cases of the local Langlands and Jacquet-Langlands correspondences, while Theorem B is only concerned with the action of the finite group $U=\U(\fqn)$ on the cohomology of the variety $X$ over a finite field, which on the surface is unrelated to the local Langlands correspondence. A bridge between the two results is provided by Theorem C (\S\ref{ss:Theorem-C}), in which we prove that the induced representation $\Ind_{\cJ}^{GL_n(K)\times D^\times\times\cW_K}H^{n-1}_c(X\tens\bfq,\ql)$ realizes certain special cases of the local Langlands and Jacquet-Langlands correspondences. Theorem C is proved in \S\ref{s:proof-Thm-C}; we make heavy use of the methods developed by Henniart in \cite{Henniart-MathNachr1992,Henniart-JLC-I}, along with Theorem B. The article concludes with \S\ref{s:proof-Thm-A}, where we prove Theorem A by combining Theorem C with the main results of Part 1.

\part{Affinoids in the Lubin-Tate tower}


\section{The Lubin-Tate tower at infinite level}\label{s:LT-tower}
As before, let $K$ be a non-archimedean local field with uniformizer $\varpi$ and residue field $\bF_q$.  Let $n\geq 1$, and let $H_0/\overline{\bF}_q$ be a one-dimensional formal $\cO_K$-module of height $n$.  Then $H_0$ is unique up to isomorphism, and $D=\End H_0\otimes K$ is the central division algebra over $K$ of invariant $1/n$.

The Lubin-Tate tower is a projective system of analytic spaces $\cM_{H_0,m}$ ($m\geq 0$) which parametrize deformations of $H_0$ with level $\varpi^m$ structure.   In \S\ref{infinitelevel} we review the construction of an analytic space $\cM_{H_0,\infty}$, which is (in a sense) the inverse limit of the $\cM_{H_0,m}$.  For now let us only describe the points of $\cM_{H_0,\infty}$.  For a complete valued extension field $E/K$, $\cM_{H_0,\infty}(E)$ is the set of isogeny classes of triples $(H,\rho,\phi)$, where $H/\cO_{E}$ is a formal $\cO_K$-module, $\rho\from H_0\otimes\cO_{E}/\varpi\to H\otimes\cO_{E}/\varpi$ is a quasi-isogeny and $\phi\from K^n\tilde{\to} V(H)(E)$ is a basis for the rational Tate module of $H$.  The pair $(H,\rho)$ is a {\em deformation} of $H_0$ to $\cO_{E}$, whereas $\phi$ is a {\em level structure}.  The group $GL_n(K)\times D^\times$ acts on $\cM_{H_0,\infty}$, by operating on $\phi$ and $\rho$, respectively.

The goal of this section is to give a precise description of $\cM_{H_0,\infty}$ as an adic space, as in \cite{WeinsteinSemistableModels}.   On the one hand, $\cM_{H_0,\infty}$ represents a moduli problem for adic spaces which generalizes the one given in the previous paragraph.  On the other hand, it turns out that  $\cM_{H_0,\infty}$ has a convenient linear algebra description, which goes as follows.  Let $M(H_0)$ be the Dieudonn\'e module of $H_0$.  Then the top exterior power $\wedge^n M(H_0)$ is the Diedonn\'e module of a formal $\cO_K$-module of height 1 and dimension 1, which we call $\wedge H_0$.   Recall that $\tilde{H}$ means the universal cover of any lift of $H_0$ to $\cO_{\breve{K}}$, and similarly for $\widetilde{\wedge H}$.   Thm. \ref{logcommutes} furnishes a natural $K$-alternating map $\delta\from \tilde{H}^n\to \widetilde{\wedge H}$ (which even has an explicit presentation, cf. the proof of Thm. \ref{logcommutes}).  By classical Lubin-Tate theory, $\wedge H_0$ has a unique lift $\wedge H$, which shows that $\cM_{\wedge H_0,\infty}$ may be identified with the locally closed subspace of $\widetilde{\wedge H}_{\eta}$ which parametrizes nonzero sequences of torsion elements.  Finally, Thm. \ref{LTdiagram} shows that $\cM_{H_0,\infty}$ is the preimage of $\cM_{\wedge H_0,\infty}$ under $\delta\from \tilde{H}^n_{\eta}\to\widetilde{\wedge H}_{\eta}$.

\subsection{Formal $\cO_K$-modules:  definitions}
We rely heavily on the notions of formal $\cO_K$-modules and formal $\cO_K$-module laws.  These are reviewed below.


Let $A$ be an $\cO_K$-algebra.  A $1$-dimensional {\em formal $\cO_K$-module law} over $A$ is a collection of power series $H(X,Y)\in A\powerseries{X,Y}$ and $[a]_H(X)\in A\powerseries{X}$ ($a\in\cO_K$) satisfying the usual constraints.  The addition law $H(X,Y)$ will be written $X+_HY$, and the entire package will simply be called $H$.  A homomorphism $f\from H\to H'$ between $1$-dimensional formal $\cO_K$-module laws over $A$ is a power series $f(X)\in A\powerseries{X}$ without constant term for which $f(X+_HY)=f(X)+_{H'}f(Y)$ and $f([a]_H(X))=[a]_{H'}(f(X))$.

For a 1-dimensional formal $\cO_K$-module law $H/A$, the Lie algebra $\Lie H$ is the free $A$-module spanned by the symbol $d/dX$.  A homomorphism $f\from H\to H'$ induces a homomorphism of $A$-modules $\Lie H\to \Lie H'$ which is simply multiplication by $f'(0)$.

Formal $\cO_K$-module laws of higher dimension are defined similarly;  if $H$ has dimension $d$ then $\Lie H$ is a free $A$-module of rank $d$.

It will be useful to present a more functorial description of formal $\cO_K$-module laws.  For this we need the notion of an adic ring.

\begin{defn} A topological ring $A$ is {\em adic} if there exists an ideal $I\subset R$ such that $A$ is separated and complete for the $I$-adic topology.   Such an $I$ is called an {\em ideal of definition} for $A$.  If $A$ is an adic ring, an {\em adic $A$-algebra} $R$ is an adic ring together with a continuous homomorphism $A\to R$.
\end{defn}

For an adic ring $A$, let $\Adic_A$ be the category of adic $R$-algebras.
We often consider covariant functors $\mathcal{F}\from\Adic_A\to \Sets$.
If $\mathcal{F}$ is representable by an adic $A$-algebra $R$, we will often confuse $\mathcal{F}$ with the affine formal scheme $\Spf R$.

A basic example is the functor $A\mapsto \Nil(A)$ which assigns to an adic $R$-algebra the set of topologically nilpotent elements of $A$.  Then $\Nil=\Spf A\powerseries{T}$.


Let $A$ be an adic $\cO_K$-algebra, and let $H$ be a $d$-dimensional formal $\cO_K$-module law over $A$.  Let $\cO_K\Mod$ be the category of $\cO_K$-modules.  Then $H$ determines a functor $\Adic_A\to\cO_K\Mod$, which will also be called $H$.  For an object $R$ of $\Adic_A$, $H(R)$ is the set $\Nil(R)$ with $\cO_K$-module structure determined by the operations in $H$.  The composition of $H\from\Adic_A\to\cO_K\Mod$ with the forgetful functor $\cO_K\Mod\to\Sets$ is representable by $\Spf A\powerseries{X_1,\dots,X_d}$.

By a {\em formal $\cO_K$-module} over $A$ we will mean a functor $H\from \Adic_A\to\cO_K\Mod$ which is isomorphic to the functor induced by some formal $\cO_K$-module law.  Then $\Lie H$ may be defined as the kernel of $H(A[X]/X^2)\to H(A)$.  An isomorphism between $H$ and the functor induced by a formal $\cO_K$-module law will be called a coordinate on $H$.  A choice of coordinate on $H$ gives a basis for the free $A$-module $\Lie H$.




\subsection{Formal $\cO_K$-modules over $\overline{\bF}_q$}   These are easily classified.  A 1-dimensional formal $\cO_K$-module over a perfect field $k$ containing $\bF_q$ is either isomorphic to $\hat{\mathbb{G}}_a$, or else there exists a maximal integer $n\geq 1$ for which (with respect to some choice of coordinate on $H$) $[\varpi]_H(X)$ is a power series in $X^{q^n}$.  In the latter case, $H$ is {\em $\varpi$-divisible}, and $n$ is the {\em height} of $H$.  If $k$ is assumed to be algebraically closed (for instance if $k=\overline{\bF}_q$), then there even exists a coordinate on $H$ for which $[\varpi]_H(X)=X^{q^n}$.

If $A$ is a local $\cO_K$-algebra whose residue field $k$ is perfect, we will refer to the height of a formal $\cO_K$-module $H/A$ as the height of $H\otimes k$ (if this exists).

\subsection{Logarithms}
Let $A$ be an adic  $\cO_K$-algebra.  For a 1-dimensional formal $\cO_K$-module law $H$ over $A$, we have the corresponding logarithm series $\log_H(T)\in (A\otimes K)\powerseries{T}$.  This is the unique power series of the form $T+c_2T^2+\dots$ which furnishes an isomorphism between $H\otimes (A\otimes K)$ and the formal additive group $\hat{\mathbb{G}}_a$.  If $A$ is $\cO_K$-flat, then $\log_H$ determines $H$.

For each $n\geq 1$, there is a particularly convenient formal $\cO_K$-module law $H$ for which
\[ \log_H(T) = \sum_{i=0}^\infty \frac{T^{q^{in}}}{\varpi^i}. \]
We call this $H$ the {\em standard formal $\cO_K$-module} of height $n$.  It is obtained by setting $v_n=1$ and $v_j=0$ (for all $j\neq n$) in Hazewinkel's universal $p$-typical formal $\cO_K$-module over $\cO_K[v_1,v_2,\dots]$ (see \cite{GrossHopkins}, \S13).
Although $H$ has a model over $\cO_K$, we will take its base ring to be $\cO_{\breve{K}}$, where $\breve{K}$ is the completion of the maximal unramified extension of $K$.  It is easy to check that $H\otimes \overline{\bF}_q$ has height $n$ in the sense of the previous subsection.

If $H$ is a general formal $\cO_K$-module over $A$, then the logarithm $\log_H$ is an isomorphism between $H\otimes A[1/\varpi]$ and the additive formal $\cO_K$-module $\Lie H\otimes \hat{\mathbb{G}}_a$.

\subsection{Additive extensions and the Dieudonn\'e module}
Let $H$ be a formal $\cO_K$-module of height $n$ over a local $\cO_K$-algebra $A$.

\begin{defn} A {\em {rigidified additive extension}} of $H$ is an exact sequence $0\to\hat{\mathbb{G}}_a\to E\to H\to 0$ of formal $\cO_K$-modules equipped with a splitting $\Lie H\to \Lie E$ of Lie algebras.  The group of isomorphism classes of rigidified additive extensions of $H$ is denoted $\ExtRig(H,\hat{\mathbb{G}}_a)$.  
\end{defn}
To give a splitting $\Lie H\to \Lie E$ is equivalent to giving an invariant differential $\omega_E$ on $E$ whose pull-back under $\hat{\mathbb{G}}_a\to E$ is the canonical differential $dT$ on $\hat{\mathbb{G}}_a$.

There is a {\em universal additive extension} (see \cite{GrossHopkins}, \S 11)
\[ 0\to V\to E\to H\to 0,\]
with $V$ isomorphic to $(\hat{\mathbb{G}}_a)^{n-1}$.  Then $\ExtRig(H,\hat{\mathbb{G}}_a)$ is dual to $\Lie E$.

Rigidified additive extensions of $H$ can be constructed using special power series called {\em quasilogarithms}.  For a power series $g(T)\in (A\otimes K)\powerseries{T}$, we let
\begin{eqnarray*}
\Delta g(X,Y)=g(X+_HY)-g(X)-g(Y)\\
\delta_a g(X)= g([a]_H(X))-ag(X),\text{ $a\in \cO_K$}.
\end{eqnarray*}
Let $\delta g$ denote the collection of power series $\set{\Delta g,\delta_a g}$.  We say $\delta g$ is integral if $\Delta g$ and $\delta_a g$ lie in $A\powerseries{T}$.

\begin{defn}
A {\em quasilogarithm} for $H$ is a power series $g(T)\in (A\otimes K)\powerseries{T}$ without constant coefficient for which $\delta g$ and $dg$ (the derivative) are both integral.  Define the module of quasilogarithms as the $A$-module
\[ \QuasiLog(H)=\frac{\set{g(T)\in (A\otimes K)\powerseries{T}\biggm\vert g(0)=0,\delta g\text{ and } dg \text{ integral}}}{\set{g(T)\in A\powerseries{T}\biggm\vert g(0)=0}}\]
\end{defn}

If $g(T)$ is a quasilogarithm for $H$, we may define a two-dimensional formal $\cO_K$-module law $E$ by
\begin{eqnarray*}
(X,X')+_E(Y,Y')&=&(X+Y+\Delta g(X',Y'), X'+_H Y') \\
\left[a\right]_E(X,X')&=&(aX+\delta_ag(X'), X').
\end{eqnarray*}
Then $0\to \hat{\mathbb{G}}_a\to E\to H\to 0$ is an additive extension of $H$.  We define a differential $\omega_E$ on $E$ by the formula
\[ \omega_E = dX'+dg(X). \]
Then the pull-back of $\omega_E$ to $\hat{\mathbb{G}}_a$ is $dT$, so that $E$ and $\omega_E$ define a rigidified additive extension of $H$.  If $g(T)$ lies in $A\powerseries{T}$, then $E$ is isomorphic to the trivial extension $\hat{\mathbb{G}}_a\oplus H$ via the isomorphism $(X+g(X'),X')$.
We therefore have a map $\QuasiLog(H)\to\ExtRig(H,\hat{\mathbb{G}}_a)$.

\begin{prop}[\cite{GrossHopkins}, Prop. 8.5]  The map $\QuasiLog(H)\to \ExtRig(H,\hat{\mathbb{G}}_a)$ is an isomorphism of $A$-modules.
\end{prop}

For the standard formal $\cO_K$-module, the quasilogarithms can be written down explicitly.

\begin{lem}[\cite{GrossHopkins}, Prop. 13.8] \label{Qlogbasis} Let $H$ be the standard formal $\cO_K$-module law.  A basis for $\QuasiLog(H)$ is given by
\[\log_H(T),\frac{1}{\varpi}\log_H(T^q),\dots,\frac{1}{\varpi}\log_H(T^{q^{n-1}}).\]
\end{lem}


Let $H_0=H\otimes\overline{\bF}_q$.
We will write $M(H_0)=\ExtRig(H,\hat{\mathbb{G}}_a)$:
this is the {\em Dieudonn\'e module} of $H_0$.
$M(H_0)$ does not depend on the choice of lift $H$.  In general, if $A\to A'$ is a surjection of local $\cO_K$-algebras whose kernel has $\cO_K$-divided powers, and $H/A$ is a formal $\cO_K$-module, then $\ExtRig(H,\hat{\mathbb{G}}_a)$ only depends on $H\otimes A'$ in a functorial sense.

Lemma \ref{Qlogbasis} gives a privileged basis for the rational Dieudonn\'e module $M(H_0)\otimes K$, corresponding to the quasilogarithms \[\log_H(T),\log_H(T^q),\dots,\log_H(T^{q^{n-1}}).\]  We call this the standard basis of $M(H_0)\otimes K$.

\subsection{The universal cover} Let $A$ be an adic $\cO_K$-algebra, and let $H$ be a $\varpi$-divisible formal $\cO_K$-module over $A$.  We define the {\em universal cover} $\tilde{H}$ as the functor from $\Adic_A$ to $K$-vector spaces, defined by
\[ \tilde{H}(R) = \varprojlim H(R), \]
where the inverse limit is taken with respect to multiplication by $\varpi$.

\begin{lem}\label{H0tilde}
Let $H_0$ be a 1-dimensional $\varpi$-divisible formal $\cO_K$-module over $\overline{\bF}_q$.   Then $\tilde{H}_0$ is isomorphic to $\Spf \overline{\bF}_q\powerseries{T^{1/q^\infty}}$, where $\overline{\bF}_q\powerseries{T^{1/q^\infty}}$ is defined as the $T$-adic completion of the ring $\overline{\bF}_q[T^{1/q^\infty}]$.
\end{lem}

\begin{proof}  Since $H_0$ is a $\varpi$-divisible formal $\cO_K$-module over an algebraically closed field containing $\bF_q$, we may choose the coordinate on $H_0$ in such a way that $[\varpi]_{H_0}(T)=T^{q^n}$, where $n$ is the height of $H_0$.  Then for an adic $\overline{\bF}_q$-algebra $R$ we have $\tilde{H}_0(R)=\varprojlim \Nil(R)$, where the limit is taken with respect to the maps $x\mapsto x^{q^n}$.  Thus $\tilde{H}_0$ is the inverse limit of the affine formal schemes $\varprojlim \Spf\overline{\bF}_q\powerseries{T}$ with respect to the maps $T\mapsto T^{q^n}$, and this is exactly $\Spf\overline{\bF}_q\powerseries{T^{1/q^\infty}}$.
\end{proof}

For an adic $\cO_K$-algebra $A$, we will use the notation $\Nil^\flat$ for the functor:
\[ \Nil^{\flat}(R)=\varprojlim_{x\mapsto x^q} \Nil(R). \]
This functor is representable by the formal scheme $\Spf A\powerseries{T^{1/q^\infty}}$, where $A\powerseries{T^{1/q^\infty}}$ is the completion of $A[T^{1/q^\infty}]$ with respect to the ideal generated by $I$ and $T$ (where $I$ is an ideal of definition for $A$).

\begin{lem}
\label{Nilcrystalline}
Let $A$ be an adic $\bZ_p$-algebra with ideal of definition $I$.  For an adic $A$-algebra $R$, the reduction map $\Nil^{\flat}(R)\to\Nil^{\flat}(R/I)$ is an isomorphism.
\end{lem}

\begin{proof} The inverse map is as follows.  We may assume that $p\in I$. Let $(x_0,x_1,\dots)$ be an element of $\Nil^{\flat}(R/I)$.  For $i=0,1,\dots$, let $y_i\in R$ be any lift of $x_i$.  Put
\[ z_i = \lim_{n\to\infty} y_{n+i}^{q^n}. \]
Then $(z_0,z_1,\dots)$ lies in $\Nil^{\flat}(R/I)$ and lifts $(x_0,x_1,\dots)$
\end{proof}


\begin{lem}  \label{crystalline}  Let $A$ be an adic $\cO_K$-algebra admitting an ideal of definition $I$ for which $A/I=\overline{\bF}_q$.  Let $H$ and $H'$ be two 1-dimensional $\varpi$-divisible formal $\cO_K$-modules over $A$, and let $H_0$, $H_0'$ be their reductions modulo $I$.
\begin{enumerate}
\item For every object $R$ of $\Adic_A$, the natural reduction map $\tilde{H}(R)\to\tilde{H}_0(R/I)$ is an isomorphism.
\item There is an isomorphism of functors $\tilde{H}\isom\Nil^{\flat}$ (after forgetting the $K$-vector space structure on $\tilde{H}$).  Thus $\tilde{H}$ is representable by $\Spf A\powerseries{T^{1/q^\infty}}$.
\item Morphisms $H_0\to H_0'$ of $\varpi$-divisible formal $\cO_K$-modules over $A/I$ lift naturally to morphisms $\tilde{H}\to\tilde{H}'$ over $A$.
\end{enumerate}
\end{lem}

\begin{rmk}  The restriction to the case of 1-dimensional formal modules is simply for ease of notation.
\end{rmk}

\begin{proof} Part (1) is similar to Lemma \ref{Nilcrystalline}:  if $(x_0,x_1,\dots)\in \tilde{H}_0(R/I)$, let $y_i$ be an arbitrary lift of $x_i$ for $i\geq 0$, and then let
\[ z_i = \lim_{n\to\infty} [\varpi^n]_H(y_{n+i}). \]
Then $(z_0,z_1,\dots)$ is the unique lift of $(x_0,x_1,\dots)$ to $\tilde{H}(R)$.

Part (2) follows from Lemma \ref{Nilcrystalline}:  For an adic $A$-algebra $R$, we have \[\tilde{H}(R)\isom\tilde{H}_0(R/I)\isom\Nil^{\flat}(R/I)\isom \Nil^{\flat}(R). \]

For part (3), let $f_0\from H_0\to H_0'$ be a morphism;  this induces a morphism $\tilde{f}_0\from \tilde{H}_0\to \tilde{H}_0'$.  The required morphism $\tilde{f}\from \tilde{H}\to\tilde{H}'$ is the composite map
\[
\xymatrix{
 \tilde{H}(R)\ar[r]^{\sim}& \tilde{H}_0(R/I)\ar[r]^{\tilde{f}_0} &
 \tilde{H}_0'(R/I) \ar[r]^{\sim} & \tilde{H}'(R).
 }
\]
\end{proof}

\subsection{Some calculations in the universal cover of the standard formal $\cO_K$-module.}
\label{calculationsHtilde}
It will be useful to make the isomorphism $\tilde{H}\isom \Nil^{\flat}$ explicit in the case that $H$ is the standard formal $\cO_K$-module.  It is
\begin{eqnarray*}
\tilde{H}(R)&\to& \Nil^{\flat}(R) \\
(x_1,x_2,\dots)&\mapsto& (y,y^{1/q},\dots),
\end{eqnarray*}
where
\[ y^{1/q^i} = \lim_{m\to\infty} x_m^{q^{mn-i}} \]
for $i=0,1,\dots$.  We will write $\lambda\from \tilde{H}\to \Nil^{\flat}$ to refer to this isomorphism, and $\lambda_{i}\from \tilde{H}\to \Nil$ for its projection onto the $i$th component.

$\End H=\cO_L$ is the ring of integers in the unramified extension $L/K$ of degree $n$.  Indeed, if $\alpha$ is a root of $T^{q^n}-T$ in $\cO_L$, then there is a corresponding endomorphism of $H$ given on the level of coordinates by $[\alpha]_H(T)=\alpha T$. (Note that $\log_H\alpha T=\alpha\log_H T$, so this does actually define an endomorphism.)  On the other hand, if $H_0$ is the reduction of $H$ modulo $\varpi$, then $\End H_0=\cO_D$ is generated over $\cO_K$ by $\cO_L$ and the Frobenius endomorphism $\Pi$ (which sends $T$ to $T^q$).  By Lemma \ref{crystalline}, $\Pi$ lifts to an automorphism of the universal cover $\tilde{H}$.  We have
\[ \lambda_i(\Pi x)=\lambda_i(x)^q \]
for $x$ any section of $\tilde{H}$ and any $i=0,1,\dots$.

\begin{lem}\label{doubletailedsum} Let $H$ be the standard formal $\cO_K$-module, and let $R$ be a $\varpi$-torsion-free  $\cO_{\breve{K}}$-algebra which is complete for the $\varpi$-adic topology.  We have a commutative diagram
\[
\xymatrix{
(x_0,x_1,\dots)\in \ar[d] & \tilde{H}(R) \ar[rr]^{\lambda} \ar[dr]_{\log_H} && \Nil^{\flat}(R) \ar[dl] & \ni (y,y^{1/q},\dots) \ar[d] \\
\sum_{i=0}^\infty \frac{x_0^{q^{ni}}}{\varpi^i} & & R[1/p] & & \sum_{i=-\infty}^{\infty} \frac{y^{q^{ni}}}{\varpi^i}
}
\]
\end{lem}
\begin{proof} If the sequence $(y,y^{1/q},\dots)\in\varprojlim \Nil(R)$, corresponds to $(x_0,x_1,\dots)\in \tilde{H}(R)$, then
\[
x_0 = \lim_{m\to\infty} [\varpi^{m}]_{H}(y^{1/q^{mn}}).
\]
Taking logarithms, we get
\begin{eqnarray*}
\log_H(x_0)
&=&\lim_{m\to\infty} \varpi^{m}\log_H(y^{1/q^{mn}}) \\
&=&\lim_{m\to\infty}\sum_{i=0}^{\infty} \frac{y^{q^{n(i-m)}}}{\varpi^{i-m}} \\
&=&\sum_{i=-\infty}^{\infty} \frac{y^{q^{ni}}}{\varpi^i}
\end{eqnarray*}
as required.
\end{proof}

This calculation appears in \cite{FarguesFontaine}, \S7.

\subsection{The quasilogarithm map} \label{quasilogmap} Let $H_0$ be a 1-dimensional $\varpi$-divisible formal $\cO_K$-module over $\overline{\bF}_q$.  Let $H$ be any lift of $H_0$ to $\cO_{\breve{K}}$.  We will describe a functorial map of $K$-vector spaces $\qlog_H\from \tilde{H}(A)\to M(H_0)\otimes (A\otimes K)$, where $A$ is any adic $\cO_{\breve{K}}$-algebra.  The universal cover $\tilde{H}$ does not depend on the choice of lift $H$, and neither will $\qlog_H$.

Let \[ 0\to V\to E\to H\to 0\]
be the universal additive extension of $H$, so that $M(H_0)=\Lie E$.  Let $x\in\tilde{H}(A)$ be represented by the sequence $(x_1,x_2,\dots)$.  Lift this arbitrarily to a sequence $(y_1,y_2,\dots)$ of elements of $E(A)$, and then define $y\in E(A)$ by
\[ y=\lim_{n\to\infty} \varpi^ny_n. \]
Then $y$ does not depend on the choices made, and we may define $\qlog_H(x)=\log_{E}(y)\in (\Lie E)\otimes (A\otimes K) $.

It will be useful to make this map explicit when $H$ is the standard formal $\cO_K$-module.  The following lemma follows from combining the above construction with Lemma \ref{doubletailedsum}.

\begin{lem} \label{quasiloglemma} Let $H$ be the standard formal $\cO_K$-module.  Let $A$ be a $\varpi$-torsion-free adic $\cO_{\breve{K}}$-algebra.  Let $x\in\tilde{H}(A)$.
\begin{enumerate}
\item With respect to the standard basis of $M(H_0)\otimes K$, the coordinates of $\qlog_H(x)$ are given by
\[ \qlog_H(x)=(\log_H(x),\log_H(\Pi x),\dots,\log_H(\Pi^{n-1} x)) \]
\item Suppose that $x\in\tilde{H}(A)$ corresponds to $(y,y^{1/q},\dots)\in \varprojlim \Nil(A)$.  Then
\[ \log_H(\Pi^j x) = \sum_{i\in \bZ}\frac{y^{q^{ni+j}}}{\varpi^i}. \]
\end{enumerate}
\end{lem}



\subsection{Formal schemes and adic spaces} This section is a review of \cite{ScholzeWeinstein}, \S2.2.  Let $L$ be a complete nonarchimedean local field ring of integers $\cO$ and residue field $\kappa$.  Let $\varpi\in \cO$ be any element with $\abs{\varpi}<1$.   There is a {\em generic fiber functor} $M\mapsto M_{\eta}$ from sufficiently nice formal schemes $M$ over $\cO$ to rigid-analytic spaces over $L$;  see for instance the discussion in \cite{deJongCrystallineDieudonneModuleTheory}, \S7.1.  Here ``sufficiently nice" means that $M$ is covered by affine formal schemes $\Spf A$, where $A$ is a noetherian adic $\cO$-algebra for which $A\otimes \kappa$ is finitely generated over $\kappa$.  A typical example is $M=\Spf \cO\powerseries{T}$, in which case $M_{\eta}$ is the 1-dimensional open unit ball.

In the sequel we will need to work with formal schemes which are not locally noetherian, and so we need a more flexible generic fiber functor.  This is provided by Huber's theory of adic spaces, cf. \cite{HuberAdicSpaces}.  If $A$ is a (not necessarily Noetherian) topological $\cO$-algebra admitting a finitely generated ideal of definition, then we have the topological space $\Spa(A,A)$, which comes equipped with a presheaf of topological rings.  As a set, $\Spa(A,A)$ consists of those continuous valuations on $A$ which are bounded by 1.  In the theory developed in \cite{HuberAdicSpaces}, $\Spa(A,A)$ is not considered an adic space unless its structure presheaf is a sheaf, but in \cite{ScholzeWeinstein} there is a Yoneda-style generalization of the notion of adic space which does not require this condition.

The association $\Spf A\mapsto \Spa(A,A)$ extends to a fully faithful functor $M\mapsto M^{\ad}$ from formal schemes over $\Spf(\cO)$ which locally admit a finitely generated ideal of definition to adic spaces over $\Spa(\cO,\cO)$.  It then makes sense to define the {\em adic generic fiber}
\[M_{\eta}=M^{\ad}\otimes_{\Spa(\cO,\cO)} \Spa(L,\cO), \]
an adic space over $L$.  (This is denoted $M_{\eta}^{\ad}$ in \cite{ScholzeWeinstein}).

Let $\Nilp_{\cO}$ denote the category of $\cO$-algebras in which $\varpi$ is nilpotent.  A formal scheme $M$ represents a functor on $\Nilp_{\cO}$.  If $M$ locally admits a finitely generated ideal of definition, then $M_{\eta}$ has a functorial interpretation as well, and it will be useful to relate the functorial interpretations of $M$ and $M_{\eta}$.  Let $\CAff_{L,\cO}$ denote the category of affinoid $(L,\cO)$-algebras $(R,R^+)$ (in the sense of Huber) for which $R^+$ is $\varpi$-adically complete.  $M_{\eta}$ is determined by its functor of points $M_{\eta}\from \CAff_{L,\cO}\to \Sets$.

\begin{prop}[\cite{ScholzeWeinstein}, Prop. 2.2.2]  The functor $M_{\eta}\from \CAff_{L,\cO}\to\Sets$ is the sheafification of
\[ (R,R^+)\mapsto \varinjlim_{R_0\subset R^+} M(R_0)=\varinjlim_{R_0\subset R^+}\varprojlim_m M(R_0/\varpi^m), \]
where the injective limit is over open and bounded $\cO$-subalgebras $R_0\subset R^+$.
\end{prop}

\subsection{The Lubin-Tate space without level structure}
\label{nolevel}   Let $H_0/\overline{\bF}_q$ be the (unique) formal $\cO_K$-module of dimension 1 and height $n$.  We consider the functor \[M_{H_0}\from \Nilp_{\cO_{\breve{K}}}\to \Sets\] be the functor which assigns to $A$ the set of isomorphism classes of pairs $(H,\rho)$, where $H$ is a formal $\cO_K$-module over $A$ and
\[ \rho\from H_0\otimes_{\overline{\bF}_q} A/\varpi\to H\otimes_{A}A/\varpi \]
is a quasi-isogeny.  Such a pair $(H,\rho)$ will be called a {\em deformation of $H_0$ to $A$}.

Then $M_{H_0}$ is representable by a formal scheme, which is isomorphic to a disjoint union of $\bZ$ copies of $\Spf \cO_{\breve{L}}\powerseries{u_1,\dots,u_{n-1}}$, parametrized by the height of the quasi-isogeny $\rho$.

Accordingly, the generic fiber $\cM_{H_0}=M_{H_0,\eta}$ is a disjoint union of open balls of dimension $n-1$.  By the above characterization of adic generic fibers, we have the following moduli interpretation of $\cM_{H_0}$.  Let $(R,R^+)$ be an object of $\CAff_{\breve{K},\cO_{\breve{K}}}$, and let $X=\Spa(R,R^+)$.  Then an element of $\cM_{H_0}(R,R^+)$ corresponds to a cover of $X$ by open subsets $U_i=\Spa(R_i,R_i^+)$, open and bounded $\cO_{\breve{L}}$-subalgebras $R_{i,0}\subset R_i^+$, and pairs $(H_i,\rho_i)$ over $R_{i,0}$, satisfying the obvious compatibility condition.  Such a family of $(H_i,\rho_i)$ will be collectively referred to as a deformation of $H_0$ over $(R,R^+)$.   In a slight abuse of notation we will write such family simply as $(H,\rho)$.

\subsection{The Lubin-Tate space at infinite level}
\label{infinitelevel}
We review some recent results from \cite{WeinsteinSemistableModels} and \cite{ScholzeWeinstein} concerning moduli of $p$-divisible groups with infinite level structures.  Suppose as usual that $H_0$ is a one-dimensional formal $\cO_K$-module over $\overline{\bF}_q$ of height $n$.  If $(R,R^+)$ is an object of $\CAff_{\breve{K},\cO_{\breve{K}}}$, and $(H,\rho)$ is a deformation of $H_0$ to $(R,R^+)$, then (as above) $(H,\rho)$ corresponds to a family of pairs $(H_i,\rho_i)$ defined over a covering of $\Spa(R,R^+)$.  For each $i$, we have the Tate module $T(H_i)=\varprojlim H_i[\varpi^m]$, an affine group scheme.  Passing to adic generic fibers, we have for each $i$ an adic space $T(H_i)_{\eta}$ over $\breve{K}$;  these glue together to form an adic space $T(H)_{\eta}$ which carries the structure of an abelian group.

If $x=\Spa(L,L^+)$ is a point of $\Spa(R,R^+)$ with $L$ algebraically closed, then $T(H)^{\ad}(L,L^+)$ is a free $\cO_K$-module of rank $n$.

\begin{defn}\label{MH0definition} Let $\cM_{H_0,\infty}$ be the functor on $\CAff_{\breve{K}}$-algebras which assigns to $(R,R^+)$ the set of triples $(H,\rho,\phi)$, where $(H,\rho)$ is a deformation of $H_0$ to $(R,R^+)$, and $\phi\from \cO_K^n\to T(H)^{\ad}_{\eta}(R,R^+)$ is a morphism of $\bZ_p$-modules which is an isomorphism at every point $x=\Spa(L,L^+)\in\Spa(R,R^+)$.
\end{defn}



Let $H$ be any lift of $H_0$ to $\cO_{\breve{K}}$.  The main result of \cite{ScholzeWeinstein}, \S6.3 is that there is an alternate linear-alegbra description of $\cM_{H_0,\infty}$ has nothing to do with deformations of $H_0$.  Recall from \S\ref{quasilogmap} that we have for every adic $\cO_{\breve{K}}$-algebra $A$ a map $\qlog_H\from \tilde{H}(A)\to M(H_0)\otimes (A\otimes K)$ which does not depend on the choice of lift $H$.   From this we get a morphism of adic spaces $\qlog_H\from \tilde{H}_{\eta}\to M(H_0)\otimes \mathbb{G}_a$ (where $\mathbb{G}_a$ is to be interpreted as the adic space version of the additive group).

\begin{thm}  Let $(\cM_{H_0,\infty})'$ be the functor on $\CAff_{\breve{K}}$ which assigns to $(R,R^+)$ the set of $n$-tuples $(s_1,\dots,s_n)\in\tilde{H}^{\ad}_{\eta}(R,R^+)$ such that the following conditions are satisfied:
\begin{enumerate}
\item The matrix $(\qlog(s_1),\dots,\qlog(s_n))\in (M(H_0)\otimes R)^n$ is of rank exactly $n-1$.
\item For all points $x$ of $\Spa(R,R^+)$, the vectors $s_1(x)$,$\dots$,$s_n(x)$ are $K$-linearly independent.
\end{enumerate}
Then $\cM_{H_0,\infty}$ and $(\cM_{H_0,\infty})'$ are isomorphic.  (In particular $(\cM_{H_0,\infty})'$ does not depend on the choice of $H$.)
\end{thm}

Let $\wedge H$ be the one-dimensional formal $\cO_K$-module over $\cO_{\breve{K}}$ of height 1 whose Dieudonn\'e module is the top exterior power $\wedge^nM(H_0)$.  This is the formal $\cO_K$-module whose logarithm is
\[ \log_{\wedge H}(T)=\sum_{i=0}^{\infty} (-1)^{(n-1)i}\frac{T^{q^i}}{\varpi^i}. \]
Passing to the universal cover, if $(x_0,x_1,\dots)\in\widetilde{\wedge H}(R)$ corresponds to $(y,y^{1/q},\dots)\in \Nil^{\flat}(R)$, then
\begin{equation}
\label{logwedgeH}
\log_{\wedge H}x_0 = \sum_{i=-\infty}^{\infty} (-1)^{(n-1)i}\frac{y^{q^i}}{\varpi^i}.
\end{equation}

\begin{thm}
\label{logcommutes}
There exists a $K$-alternating map $\delta\from \tilde{H}^n\to\widetilde{\wedge H}$, such that the diagram
\[
\xymatrix{
\tilde{H}^n_{\eta} \ar[r]^{\delta} \ar[d]_{\qlog_H\times\cdots\times\qlog_H} & \widetilde{\wedge H}_{\eta} \ar[d]^{\log_{\wedge H}} \\
M(H_0)^n\otimes\mathbb{G}_a \ar[r]_{\det} & M(\wedge H_0)\otimes\mathbb{G}_a
}
\]
commutes.
\end{thm}

\begin{proof} The morphism $\delta$ is constructed using in \cite{ScholzeWeinstein}, \S6.4, at least in the case of $K=\bQ_p$, using an interpretation of $\tilde{H}$ in terms of $p$-adic Hodge theory.  The general case adds no real additional complication.

Alternatively, we can give an explicit description of $\delta$, and in any case we will need such a description for the calculations that follow.  Assume that $H$ is the standard formal $\cO_K$-module of height $n$.  This entails no loss of generality,
 since neither $\tilde{H}$ nor the quasilogarithm map depends on the choice of lift.  First define a morphism of formal schemes $\delta_0\from \tilde{H}^n\to\wedge H$ as follows.  Suppose a section $(s_1,\dots,s_n)$ of $\tilde{H}^n$ is given, for which the corresponding section of $(\Nil^\flat)^n$ is $(x_1,\dots,x_n)$ (this means that the $x_i$ are topologically nilpotent elements with distinguished $q$th power roots).  We set
\[ \delta_0(s_1,\dots,s_n)=(\wedge H)\sum_{(a_1,\dots,a_n)} \varepsilon(a_1,\dots,a_n)x_1^{q^{a_1}}x_2^{q^{a_2}}\cdots x_n^{q^{a_n}},\]
where
\begin{itemize}
\item the sum ranges over $n$-tuples $(a_1,\dots,a_n)$ such that $a_1+\dots+a_n=n(n-1)/2$, and such that each $a_i$ occupies a distict residue class modulo $n$,
\item $\varepsilon(a_1,\dots,a_n)$ is the sign of the permutation $i\mapsto a_{i+1}\pmod{n}$ of the set $\set{0,1,\dots,n-1}$, and
\item the symbol $(\wedge H)\sum$ means that the sum is carried out using the operation $+_{\wedge H}$.
\end{itemize}
Then we have
\begin{eqnarray*}
\log_{\wedge H}(\delta_0(s_1,\dots,s_n))
&=&\sum_{(a_1,\dots,a_n)} \varepsilon(\underline{a})\log_{\wedge H}(x_1^{q^{a_1}}\cdots x_n^{q^{a_n}})\\
&=&\sum_{(a_1,\dots,a_n)} \varepsilon(\underline{a}) \sum_{m\in\bZ}(-1)^{(n-1)m}\frac{x_1^{q^{a_1+m}}\cdots x_n^{q^{a_n+m}}}{\varpi^m},
\end{eqnarray*}
and it is not difficult to see that this is the same as
\[\det\left(\sum_{m\in\bZ} \frac{x_i^{q^{mn+j}}}{\varpi^m}\right)_{1\leq i\leq n,\;0\leq j\leq n-1},\]
which in turn equals $\det\qlog_H(s_1,\dots,s_n)$ by Lemma \ref{quasiloglemma}.  Thus we have shown that the diagram of adic spaces
\[
\xymatrix{
\tilde{H}^n_{\eta} \ar[r]^{\delta_0} \ar[d]_{\qlog_H} & \wedge H_{\eta} \ar[d]^{\log_{\wedge H}} \\
M(H_0)^n\otimes \mathbb{G}_a \ar[r]_{\det} & M(\wedge H_0) \otimes \mathbb{G}_a.
}
\]
commutes.  We claim that $\delta_0$ is $\cO_K$-multilinear and alternating.  This will follow from the same property of $\det\from M(H_0)^n\to M(\wedge H_0)$.  For instance, if $s_1,s_1',s_2,\dots,s_n$ are sections of $\tilde{H}$ over an affinoid algebra $(R,R^+)$, define an element \[\partial=\partial(s_1,s_1',s_2,\dots,s_n)\in (\wedge H)(R^+)\] by
\[\partial=\delta_0(s_1+s_1',s_2,\dots,s_n)-\delta_0(s_1,s_2,\dots,s_n)-\delta_0(s_1',s_2,\dots,s_n). \]
(here the operations are taking place in $\wedge H$).  Then the commutativity of the above diagram shows that $\log_{\wedge H}(\partial)=0$.  The kernel of $\log_{\wedge H}$ is the torsion $H[\varpi^\infty]$.  Thus the morphism of adic spaces $\partial_{\eta}\from \tilde{H}^{n+1}\to(\wedge H)_{\eta}$ factors through $H[\varpi^\infty]_{\eta}$.  But $\tilde{H}^{n+1}_{\eta}$ is connected and
$H[\varpi^\infty]_{\eta}$ is discrete, so $\partial_{\eta}$ must be constant.  Since obviously $\partial(0,\dots,0)=0$, we have $\partial_{\eta}=0$ identially.  This implies that $\partial=0$, since $\partial$ can be recovered from $\partial_{\eta}$ by looking at the induced morphism on integral global sections.  A similar argument can be applied to show that $\delta_0(s_1,\dots,s_n)$ is $\cO_K$-multilinear and alternating.

We may then define morphisms $\delta_i\from\tilde{H}^n\to\wedge H$ by (for instance) $\delta_i(s_1,\dots,s_n)=\delta_i(\varpi^{-i}s_1,\dots,s_n)$.  Then $\delta=(\delta_0,\delta_1,\dots)$ defines the required morphism $\delta\from \tilde{H}^n \to\widetilde{\wedge H}$.
\end{proof}

The morphism $\delta\from \tilde{H}^n\to\widetilde{\wedge H}$ corresponds to a morphism $\Delta\from (\Nil^{\flat})^n\to \Nil^{\flat}$, in such a way that the diagram
\[
\xymatrix{
\tilde{H}^n \ar[r]^{\delta} \ar[d] & \widetilde{\wedge H} \ar[d] \\
(\Nil^{\flat})^n \ar[r]_{\Delta} & \Nil^{\flat}
}
\]
commutes.  The morphism $\Delta$ corresponds to an element $\Delta(X_1,\dots,X_n)$ of \\ $\cO_{\breve{K}}\powerseries{X_1^{1/q^\infty},\dots,X_n^{1/q^\infty}}$, which comes equipped with a family of $q$th power roots.  It will be helpful to have a first-order approximation of $\Delta$.
\begin{lem}\label{Deltaestimate} We have
\[
\Delta(X_1,\dots,X_n)\equiv \det(X_i^{q^j})_{1\leq i\leq n,\;0\leq j\leq n-1}
\]
modulo terms of higher degree in $\cO_{\breve{K}}\powerseries{X_1^{1/q^\infty},\dots,X_n^{1/q^\infty}}$.
\end{lem}

\begin{proof} This follows from the explicit description of $\Delta$ given in the proof of Thm. \ref{logcommutes}.
\end{proof}

The following theorem gives a complete description of the space $\mathcal{M}_{H_0,\infty}$ in terms of the morphism $\delta$.

\begin{thm}[\cite{ScholzeWeinstein}, Thm. 6.4.1] \label{LTdiagram} There is a cartesian diagram
\[
\xymatrix{
\mathcal{M}_{H_0,\infty} \ar[r]^{\delta}\ar[d]  & \mathcal{M}_{\wedge H_0,\infty} \ar[d] \\
\tilde{H}_{\eta}^n \ar[r]_{\delta} & \widetilde{\wedge H}_{\eta} .
}
\]
\end{thm}

We remark that $\mathcal{M}_{\wedge H_0,\infty}$ is the disjoint union of $\bZ$ copies of the one-point space $\Spa(\hat{K}^{ab},\cO_{\hat{K}^{ab}})$, where $\hat{K}^{ab}$ is the completion of the maximal abelian extension of $K$.  We have that $\mathcal{M}_{\wedge H_0,\infty}$ is a locally closed subspace of $\widetilde{\wedge H}_{\eta}$, and therefore $\mathcal{M}_{H_0,\infty}$ is a localy closed subspace of $\tilde{H}^n_{\eta}$.

\subsection{The action of $GL_n(K)\times D^\times\times \cW_K$}
\label{groupactions} Let $C$ be the completion of an algebraic closure of $K$, and let $\mathcal{M}_{H_0,\infty,C}$ be the base change of $\mathcal{M}_{H_0,\infty}$ to $C$.  We remark that $\mathcal{M}_{H_0,\infty,C}$ is a perfectoid space in the sense of \cite{ScholzePerfectoidSpaces}.  We define a right action of $GL_n(K)\times D^\times\times \cW_K$ on $\mathcal{M}_{H_0,\infty,C}$ (which becomes a left action on cohomology).

The action of $GL_n(K)\times D^\times$ is easy enough to define in terms of the moduli problem represented by $\mathcal{M}_{H_0,\infty}$ (and doesn't require base changing to $C$).  We describe this action in terms of the description of $\mathcal{M}_{H_0,\infty}$ in Thm. \ref{LTdiagram}.  If $x=(x_1,\dots,x_n)$ is a section of $\tilde{H}^n$, and if $(g,b)\in GL_n(K)\times D^\times$, then we set $x^{(g,b)}=g^T(b^{-1}x_1,\dots,b^{-1}x_n)$, where the $T$ denotes transpose.  This action preserves $\mathcal{M}_{H_0,\infty}$.

We now turn to the action of $\cW_K$.  First we define an action of $\cW_K$ on $\tilde{H}^n_{\eta,C}$.  Suppose $w\in \cW_K$.  Let us write $\Phi$ for the Frobenius automorphism of $\breve{K}$ which induces the $q$th power map on the residue field.  Then $w$ is an automorphism of $C$ which induces $\Phi^m$ on $\breve{K}$ for some $m\in\Z$.  We get a morphism of formal schemes
\[ 1\otimes w\from \tilde{H}\hat{\otimes}_{\cO_{\breve{K}}}\cO_C\to\tilde{H}\hat{\otimes}_{\cO_{\breve{K}},\Phi^m}\cO_C\]
which induces a morphism of adic spaces over $C$:
\[ 1\otimes w\from \tilde{H}_{\eta,C}\to \tilde{H}^{(q^m)}_{\eta,C}:=(\tilde{H}\hat{\otimes}_{\cO_{\breve{K}},\Phi^m}\cO_C)_{\eta} \]
On the other hand, we have the absolute Frobenius morphism of formal schemes $\varphi\from H_0\to H_0^{(q^m)}=H_0\otimes_{\overline{\mathbb{F}}_q,\Fr_q^m} \overline{\mathbb{F}}_q$,
which induces an isomorphism $\varphi\from \tilde{H}_0\to\tilde{H}_0^{(q^m)}$,
which in turn induces an isomorphism of adic spaces $\varphi\from \tilde{H}_{\eta,C}\to\tilde{H}_{\eta,C}^{(q^m)}$.  We define an automorphism $x\mapsto x^{w^{-1}}$ of $\tilde{H}_{\eta,C}$ as the composition of $1\otimes w$ with $\Phi$.  This induces an automorphism of $\tilde{H}_{\eta,C}^n$ which preserves $\mathcal{M}_{H_0,\infty,C}$.  Note that the action of $\cW_K$ on $\mathcal{M}_{H_0,\infty,C}$ is $C$-semilinear, and that it commutes with the action of $GL_n(K)\times D^\times$.

\section{A special affinoid in the Lubin-Tate tower at infinite level}\label{s:affinoid-LT-tower}

\subsection{CM points}
\label{CMpoints}
Let $L/K$ be an extension of degree $n$, with uniformizer $\varpi_L$ and residue field $\bF_{q^n}$.  Let $C$ be the completion of a separable closure of $L$.   A deformation $H$ of $H_0$ over $\cO_C$ has {\em CM by $\cO_L$} if there exists a $K$-linear isomorphism $L\to \End G\otimes K$.  Equivalently, $H$ has CM by $L$ if it is isogenous to a formal $\cO_L$-module (necessarily of height 1).

If $H/\cO_C$ is has CM by $\cO_L$, and $\phi$ is a level structure on $H$, we get a triple $(H,\rho,\phi)$ defining an $C$-point of $\cM_{H_0,\infty}$.   Points of $\mathcal{M}_{H_0,\infty}$ constructed in this matter will be called {\em CM points} (or points with CM by $L$).

Let $x$ be a point of $\cM_{H_0,\infty}$ with CM by $L$ which corresponds to the triple $(H,\iota,\phi)$.  Then $x$ induces embeddings $i_1\from L\to M_n(K)$ and $i_2\from L\to D$, characterized by the commutativity of the diagrams
\[
\xymatrix{
K^n \ar[d]_{i_1(\alpha)} \ar[r]^{\phi} & V(H) \ar[d]^{\alpha} \\
K^n \ar[r]_{\phi} & V(H)
}
\]
and
\[
\xymatrix{
H_0 \ar[d]_{i_2(\alpha)} \ar[r]^{\iota} & H\otimes \bF_{q^n} \ar[d]^{\alpha} \\
H_0 \ar[r]_{\iota} & H\otimes\bF_{q^n}
}
\]
for $\alpha\in K$.  At the risk of minor confusion, from this point forward we will usually suppress $i_1$, $i_2$ from the notation and instead think of $L$ as a subfield of $M_n(K)$ and $D$.  Let $\Delta\from L\to M_n(K)\times D$ be the diagonal embedding.  The group $GL_n(K)\times D^\times$ acts transitively on the set of CM points;  the stabilizer of our fixed CM point $x$ is $\Delta(L^\times)$.

By Lubin-Tate theory, points of $\cM_{H_0,\infty}$ with CM by $L$ are defined over the completion of the maximal abelian extension of $L$.  That is, these points are fixed by the commutator $[\cW_L,\cW_L]$.  Recall that the relative Weil group $\cW_{L/K}$ is the quotient of $\cW_K$ by the closure of $[\cW_L,\cW_L]$.  If $x$ has CM by $L$, and $w\in \cW_K$, then $x^{w^{-1}}$ also has CM by $L$, and therefore there exists a pair $(g,b)\in GL_n(K)\times D^\times$ for which $x^{(g,b)}=x^{w^{-1}}$.  Then $w\mapsto L^\times (g,b)$ is a well-defined injective map $j\from \cW_{L/K}\to L^\times\backslash(GL_n(K)\times D^\times)$.

Recall also that there is an exact sequence
\[
\xymatrix{
1\ar[r] & L^\times \ar[r]^{\rec_L} & \cW_{L/K} \ar[r] & \Gal(L/K) \ar[r] & 1
}
\]
corresponding to the canonical class in $H^2(\Gal(L/K),L^\times)$ (cf. \cite{TateNumberTheoreticBackground}).

\begin{lem}  \label{LubinTateReciprocity} For all $\alpha\in L^\times$, we have $j(\rec_L\alpha)=L^\times(1,\alpha)$.
\end{lem}

\begin{proof}  This is tantamount to the statement that $x^{\rec_L\alpha}=x^{(1,\alpha^{-1})}$, and will follow from classical Lubin-Tate theory.

By replacing $x$ with a translate, we may assume that $H/\cO_L$ admits endomorphisms by all of $\cO_K$, and that $\phi$ maps $\cO_K^n$ isomorphically onto $T(H)$.

In~\cite{LubinTate}, the main theorem shows that the maximal abelian extension $L^{\ab}/L$ is the compositum of $L^{\nr}$, the maximal unramified extension, with $L_\infty$, the field obtained by adjoining the roots of all iterates $[\varpi_L^n]_{H}$ to $L$.   Write $\alpha=\varpi_L^mu$, with $u\in\cO_L^\times$.
In the notation of~\cite{LubinTate}, the Artin symbol $(\alpha,L^{\ab}/L)$ restricts to the $m$th power of the (arithmetic) ${q^n}$th power Frobenius $\Fr_{q^n}$ on $L^{\nr}$ and sends a root $\xi$ of $[\varpi^n_L]_{H}$ to $[u^{-1}]_{H}(\xi)$.
But $\rec_L$ sends a uniformizer to a geometric Frobenius, so $\rec_L(\alpha)=(\alpha^{-1},L^{\ab}/L)$ as elements of $\cW_L^{\ab}$.
Thus for a unit $u\in \cO_L^\times$, $x^{\rec_L(u)}$ is represented by the triple $(H,\iota,\phi)^{\rec_L(u)}=(H,\iota,\phi\circ u)=x^{(u,1)}=x^{(1,u^{-1})}$ as claimed.
Finally, since $\rec_L(\varpi_L)$ acts as geometric Frobenius on $L^{\nr}$ and as the identity on $L_\infty$, we have  $x^{\rec_L(\varpi_L)}=(H,\iota,\phi)^{\rec_L(\varpi_L)}=(H,\iota\circ\Fr_{q^n}^{-1},\phi)$.
Since $[\varpi_L]_{G}$ reduces to $\Fr_{q^n}$ modulo $\varpi_L$, we have $x^{\rec_L(\varpi_L)}=x^{(1,\varpi^{-1}_L)}$.  This completes the proof of the claim.
\end{proof}

Let $\mathcal{N}_1$ and $\mathcal{N}_2$ be the normalizers of $L^\times$ in $GL_n(K)$ and $D^\times$, respectively.  Then both $\mathcal{N}_1$ and $\mathcal{N}_2$ are extensions of $\Gal(L/K)$ by $L^\times$.  Let $\mathcal{N}\subset GL_n(K)\times D^\times$ be the pullback in the diagram
\[
\xymatrix{
\mathcal{N} \ar[r] \ar[d]& \mathcal{N}_2 \ar[d] \\
\mathcal{N}_1 \ar[r]& \Gal(L/K).
}
\]
Then $\mathcal{N}$ is also an extension of $\Gal(L/K)$ by $L^\times$.  A pair $(g,b)\in GL_n(K)\times D^\times$ belongs to $\mathcal{N}$ if and only if there exists $\sigma\in \Gal(L/K)$ such that for all $\alpha\in L^\times$ we have $g^{-1}\alpha g = b^{-1}\alpha b = \alpha^{\sigma}$.

\begin{prop}\label{jisomorphism} The map $j\from \cW_K\to L^\times\backslash (GL_n(K)\times D^\times)$ factors through an isomorphism of groups $\cW_{L/K}\to L^\times\backslash \mathcal{N}$.
\end{prop}

\begin{proof} Let $w\in \cW_K$, and let $j(w)=L^\times(g,b)$, so that $x^w=x^{(g,b)^{-1}}$.  We first claim that $(g,b)\in\mathcal{N}$.

Let $\sigma$ be the image of $w$ in $\Gal(L/K)$, and let $\alpha\in L^\times$ be arbitrary.  Repeatedly using the fact that the actions of $GL_n(K)\times D^\times$ and $\cW_K$ on $\mathcal{M}_{H_0,\infty}(C)$ commute, we have by Lemma \ref{LubinTateReciprocity}
\begin{eqnarray*}
x^{(1,b^{-1}\alpha b)}
&=&x^{(g,b)^{-1}(1,\alpha)(g,b)}\\
&=& x^{w(1,\alpha)(g,b)} \\
&=& x^{(1,\alpha)w(g,b)}\\
&=& x^{\rec_L(\alpha)w(g,b)}\\
&=& x^{(g,b)\rec_L(\alpha)w}\\
&=& x^{w^{-1}\rec_L(\alpha)w }\\
&=& x^{\rec_L(\alpha^\sigma)}\\
&=& x^{(1,\alpha^{\sigma})},
\end{eqnarray*}
so that $b^{-1}\alpha b=\alpha^{\sigma}$.  A similar calculation shows that $g^{-1}\alpha g=\alpha^{\sigma}$.   Thus $(g,b)\in \mathcal{N}$.  One sees from the calculation
\[
x^{j(w)j(w')}=x^{w^{-1}j(w')}=x^{j(w')w^{-1}}=x^{(w')^{-1}w^{-1}}=x^{(ww')^{-1}}=x^{j(ww')}
\]
that $j$ factors through a group homomorphism $\cW_K\to L^\times\backslash\mathcal{N}$.  The restriction of this homomorphism to $\cW_L$ factors through $\cW_L^{\ab}$, so in fact $j$ factors through a homomorphism $\cW_{L/K}\to L^\times\backslash\mathcal{N}$, which we temporarily call $\tilde{j}$.  From the diagram
\[
\xymatrix{
1 \ar[r] & \cW_L^{\ab} \ar[d]_{\rec_L^{-1}} \ar[r] & \cW_{L/K}  \ar[r] \ar[d]^{\tilde{j}} & \Gal(L/K) \ar[d]^{=} \ar[r] & 1 \\
1 \ar[r] &  L^\times \ar[r]^{(1,\alpha)} &  \mathcal{N}/L^{\times} \ar[r] & \Gal(L/K) \ar[r] & 1
}
\]
we see that $\tilde{j}\from \cW_{L/K}\to L^\times\backslash\mathcal{N}$ is an isomorphism.
\end{proof}

Henceforth we will use the letter $j$ for the homomorphism $\cW_K\to L^\times\backslash\mathcal{N}$ which induces the isomorphism $\cW_{L/K}\to L^\times\backslash\mathcal{N}$ of Prop. \ref{jisomorphism}.  Then $j$ is characterized by the property that
\begin{equation}
\label{jproperty}
x^{j(w)}=x^{w^{-1}}
\end{equation}
for all $w\in \cW_K$.


\begin{prop}
\label{stabx}
Let $\mathcal{S}$ be the stabilizer of $x$ in $GL_n(K)\times D^\times\times \cW_K$.  Then $\mathcal{S}$ is the set of triples $(g,b,w)$, where $(g,b)\in\mathcal{N}$ is a lift of $j(w)\in L^\times\backslash\mathcal{N}$.
\end{prop}

\begin{proof} We have already seen that these elements fix $x$.  Suppose $(g,b,w)$ fixes $x$, so that $x^{(g,b)}=x^{w^{-1}}=x^{j(w)}$.  Then $(g,b)j(w)^{-1}\in GL_n(K)\times D^\times$ fixes $x$, and so must lie in the diagonally embedded $L^\times$.
\end{proof}

\subsection{Linking orders}
We continue to assume that $L/K$ is a separable extension of degree $n$ and that $x\in\mathcal{M}_{H_0,\infty}(C)$ is a point with CM by $L$.  Then $x$ determines
$K$-linear embeddings $L\injects M_n(K)$ and $L\injects D$.  As before, let $\Delta\from L\injects M_n(K)\times D$ be the diagonal embedding.  Finally, let $m\geq 0$ be an integer.  In this section we define a $\Delta(\cO_L)$-order $\cL=\cL_{x,m}$ inside of $M_n(K)\times D$ which plays an important role in our analysis.  Much of the material in this section is taken from \cite{WeinsteinSemistableModels}, \S3.3.

The CM point $x$ determines a deformation $H_0$ of $H$ to $\cO_\C$, and a basis for the free $\cO_K$-module $TH=\varprojlim H[\varpi^n](\cO_\C)$.  We may then identify $M_n(K)$ with the algebra of $K$-linear endomorphisms of $VH=TH\otimes K$.  Let $\mathfrak{A}\subset M_n(K)$ be the $\cO_L$-subalgebra of elements which send $\gp_L^iTH$ into $\gp_L^iTH$ for each $i\in\bZ$.  Let $\gP\subset \mathfrak{A}$ be the ideal of elements which send $\gp_L^iTH$ into $\gp_L^{i+1}TH$ for each $m\in\bZ$;  then $\gP$ is the double-sided ideal generated by $\varpi_L$.

We have a $K$-linear pairing $M_n(K)\times M_n(K)\to K$ given by $(a,b)\mapsto \tr(ab)$.
With respect to this pairing we may write $M_n(K)=L\oplus C_1$, where $C_1$ is a left and right $L$-vector space of dimension $n-1$.
Let $\pr_1\from M_n(K)\to L$ be the projection onto the first factor.  Similarly, we can write $D=L\oplus C_2$;  let $\pr_2\from D\to L$ be the projection onto the first factor.

Let $r$ be the largest integer such that $\pr_1(\gP^r)\subset\gp_L^m$, and let
\[ P_{1,m}=\set{a\in\gP^r\biggm\vert \pr_1(a)\in \gp_L^m}, \]
an $\cO_L$-submodule of $\mathfrak{A}$.  By definition of $r$, we have $P^2_{1,m}\subset P_{1,m}$, so that $1+P_{1,m}$ is an open subgroup of $\mathfrak{A}^\times$ containing $1+\gp_L^m$.

We define similar structures for the division algebra $D$.  Let $\gP_D\subset\cO_D$ be the maximal double-sided ideal.  Let $r'$ be the largest integer such that $\pr_2(\gP^{2r'}_D)\subset \gp_L^m$, and let
\[ P_{2,m}=\set{b\in\gP_D^{r'}\biggm\vert \pr_2(b)\in \gP_L^m}.\]
Then $P_{2,m}^2\subset P_{2,m}$, and $1+P_{2,m}$ is an open subgroup of $\cO_D^\times$ which contains $1+\gp_L^m$.

Now we consider structures within the product $M_n(K)\times D$.
Let
\[
\cL = \Delta(\cO_L)+(P_{1,m}\times P_{2,m}).\]
Then $\cL$ is a $\Delta(\cO_L)$-order in $M_n(K)\times D$.  We also define a two-sided ideal $\gP\subset\cL$ by
\[ \gP=\Delta(\gP_L)+(P_{1,m+1}\times P_{2,m+1}). \]
Then $\mathcal{R}=\cL/\gP$ is a finite-dimensional algebra over the residue field $\cO_L/\gp_L$.

\subsection{Description of the linking order in the case of $L/K$ unramified}\label{ss:linking-orders-unramified-case} 
From now on we impose the assumption that $L/K$ is unramified.  It will be helpful to have a completely explicit description of $\cL$ in the case.  We will assume that the CM point $x$ corresponds to the standard formal $\cO_K$-module $H$, which has CM by the full ring of integers in $L$.  This ensures that $M_n(\cO_K)\cap L=\cO_L$ and $\cO_D\cap L=\cO_L$.  Let $C_1^\circ=C_1\cap M_n(\cO_K)$ and $C_2^\circ=C_2\cap \cO_D$.

Then the linking order is
\[ \cL=\Delta(\cO_L)+(\gp_L^m\times\gp_L^m)+\left(\gp_L^{\ceil{m/2}}C_1^\circ\times\gp_L^{\ceil{(m-1)/2}}C_2^\circ\right),\]
and its ideal $\gP$ is
\[ \gP=\Delta(\gp_L)+\left(\gp_L^{m+1}\times\gp_L^{m+1}\right)+\left(\gp_L^{\ceil{(m+1)/2}}C_1^\circ\times \gp_L^{\ceil{m/2}}C_2^\circ\right).\]

\begin{lem}
\label{Sring}
The quotient $S=\mathcal{L}/\gP$ is an $\bF_{q^n}$-algebra of dimension $n+1$.  It admits a basis $1,e_1,\dots,e_n$.  Multiplication in $S$ is determined by the following rules.
\begin{itemize}
\item $e_i\cdot a = a^{q^i}\cdot e_i$, $a\in \bF_{q^n}$
\item If $m=1$, then
\[
e_ie_j=
\begin{cases}
e_{i+j}, & i+j\leq n\\
0,& i+j>n
\end{cases}
\]
\item If $m\geq 2$, then
\[
e_ie_j=
\begin{cases}
e_n, & i+j=n\\
0,& i+j\neq n.
\end{cases}
\]
\end{itemize}
\end{lem}

\begin{proof}  This is a simple calculation.  We will only explain how to construct the elements $e_1,\dots,e_n$.  An interesting feature is that the roles of $M_n(K)$ and $D$ alternate based on the parity of $m$.

Let $s\in GL_n(\cO_K)$ be an element of order $n$ in the normalizer of $L^\times$ such that conjugation by $s$ effects the Frobenius automorphism of $L/K$.  Then the $\cO_L$-module $C_1^\circ$ is spanned by $s,s^2,\dots,s^{n-1}$.  Similarly, we have an element $\Pi\in\cO_D$ coming from the Frobenius endomorphism of $H_0$;  we have that $\Pi^n=\varpi$, conjugation by $\Pi$ effects the Frobenius automorphism on $L/K$, and $C_2^\circ$ is spanned over $\cO_L$ by $\Pi,\Pi^2,\dots,\Pi^{n-1}$.

If $m$ is even, then $e_i$ is the image of $(\varpi^{m/2}s^i,0)$ for $i=1,\dots,n-1$, and $e_n$ is the image of $(\varpi^m,0)$.  If $m$ is odd, then $e_i$ is the image of $(0,\varpi^{(m-1)/2}\Pi^i)$ for $i=1,\dots,n-1$, and $e_n$ is the image of $(0,\varpi^m)$.
\end{proof}

\subsection{The unipotent group $\U$, and the variety $X$}\label{theunipotentgroup}

Let $\U_0$ be the affine group variety over $\bF_q$ whose points over an $\bF_q$-algebra $R$ are formal expressions $1+\alpha_1e_1+\dots+\alpha_ne_n$, with $\alpha_i\in R$.  The group operation is determined by similar rules as in Lemma~\ref{Sring}.  (Thus $\U_0$ depends on $q$, $n$, and $m$, although the dependence on $m$ is determined by whether or not $m=1$.)  Lemma \ref{Sring} shows there is an isomorphism between $\U_0(\bF_{q^n})$ and the $p$-Sylow subgroup of $S^\times$.

\mbr

Let $Y_0\subset \U_0$ be defined by the equation $a_n=0$, so that $Y_0\cong \mathbb{A}^{n-1}_{\bF_q}$.
Write $L_{q^n}\from \U_0\to\U_0$ for the Lang map $g\mapsto \Fr_{q^n}(g)\cdot g^{-1}$, where $\Fr_{q^n}$ is the $q^n$-power Frobenius map.  Put $X_0=L_{q^n}^{-1}(Y_0)$.
Let $\U=\U_0\tens_{\bF_q}\fqn$, $X=X_0\otimes_{\bF_q}\bF_{q^n}$ and $Y=Y_0\otimes_{\bF_q}\bF_{q^n}$.  The group $U=\U(\bF_{q^n})$ acts on $X$ by right multiplication and the map $X\to Y$ induced by $L_{q^n}$ makes $X$ an \'etale $\U(\bF_{q^n})$-torsor over $Y$.
In particular, we obtain an action of $U$ on $H_c^{\bullet}(X_{\bfq},\overline{\bQ}_\ell):=\bigoplus_{i\in\bZ} H^i_c(X_{\bfq},\overline{\bQ}_\ell)$, where $X_{\bfq}\overset{\text{def}}{=}X\tens_{\fqn}\bfq=X_0\tens_{\bF_q}\bfq$.

We can give explicit formulas for the variety $X$.  If $m=1$, then $X/\bF_{q^n}$ is the $(n-1)$-dimensional hypersurface in the variables $a_1,\dots,a_n$ with equation
\[
\label{hyp}
\det\left(\begin{matrix}
a_1^{q^n}-a_1 & a_2^{q^n}-a_2 & a_3^{q^n}-a_3 & \cdots & a_{n-1}^{q^n} - a_{n-1} & a_n^{q^n}-a_n
\\
1 & a_1^q & a_2^q & \cdots & a_{n-2}^q & a_{n-1}^q
\\
0 & 1 & a_1^{q^2} & \cdots & a_{n-3}^{q^2} & a_{n-2}^{q^2}
\\
\vdots & & & \ddots & & \vdots
\\
0 & 0 & 0 & \cdots& 1 & a_1^{q^{n-1}} \end{matrix}\right)=0.
\]
If $m\geq 2$, then the equation is
\[
\label{hyp2}
\det\left(\begin{matrix}

a_1^{q^n}-a_1 & a_2^{q^n}-a_2 & a_3^{q^n}-a_3 & \cdots & a_{n-1}^{q^n} - a_{n-1} & a_n^{q^n}-a_n
\\
1 & 0 & 0 & \cdots & 0 & a_{n-1}^q
\\
0 & 1 & 0 & \cdots & 0 & a_{n-2}^{q^2}
\\
\vdots & & & \ddots & & \vdots
\\
0 & 0 & 0 & \cdots& 1 & a_1^{q^{n-1}} \end{matrix}\right)=0.
\]

If instead we use parameters $b_1,\dots,b_n$ on $\mathbf{U}$ defined by $g^{-1}=1+b_1e_1+\dots+b_ne_n$, then $X$ is defined by
\begin{equation}
\label{hyp3}
\det
\begin{pmatrix}
b_n^{q^n}-b_n & b_{n-1}^{q^n}-b_{n-1} & \cdots & b_2^{q^n}-b_2 &
b_1^{q^n} - b_1 \\
b_1^q & 1 & \cdots & 0 & 0 \\
b_2^{q^2} & b_1^{q^2} & \cdots & 0 & 0 \\
\vdots & &  \ddots &  & \vdots\\
b_{n-1}^{q^{n-1}} & b_{n-2}^{q^{n-1}} & \cdots & b_1^{q^{n-1}} & 1
\end{pmatrix}
=0
\end{equation}
if $m=1$, and
\begin{equation}
\label{hyp4}
\det
\begin{pmatrix}
b_n^{q^n}-b_n & b_{n-1}^{q^n}-b_{n-1} & \cdots & b_2^{q^n}-b_2 &
b_1^{q^n} - b_1 \\
b_1^q & 1 & \cdots & 0 & 0 \\
b_2^{q^2} & 0 & \cdots & 0 & 0 \\
\vdots & &  \ddots &  & \vdots\\
b_{n-1}^{q^{n-1}} & 0 & \cdots & 0 & 1
\end{pmatrix}
=0
\end{equation}
if $m\geq 2$.


\subsection{The norm morphism}\label{ss:reduced-norm-first-definition}
Consider the map $\sN:GL_n(K)\times D^\times\to K^\times$ given by $(g,b)\mapsto \det(g)\Nrd_{D/K}(b)^{-1}$, where $\Nrd_{D/K}$ is the reduced norm homomorphism. If $\cL$ is the linking order from \S\ref{ss:linking-orders-unramified-case}, then $\sN$ takes the subgroup $\cL^\times\subset GL_n(K)\times D^\times$ into $1+\gp_K^m\subset K^\times$ and induces a homomorphism $S^\times\to (1+\gp_K^m)/(1+\gp_K^{m+1})$, where $S=\mathcal{L}/\gP$ is the quotient appearing in Lemma \ref{Sring}. In particular, we obtain a homomorphism $N:U\to\bF_q\subset\fqn$, which is invariant under the conjugation action of $S^\times$ on $U$ and restricts to the trace map $\Tr_{\fqn/\bF_q}:\fqn\to\bF_q$ on the subgroup $\{1+a_n e_n\st a_n\in\fqn\} \cong \fqn$ of $U$.

\mbr

In \S\ref{s:reduced-norm} we prove that the map $N$ can be extended to a morphism of $\fqn$-varieties
\begin{equation}
\label{NU}
N\from \U\to \mathbb{G}_a,
\end{equation}
which is not a homomorphism of algebraic groups, but which has the properties
\begin{itemize}
\item $N(gh)=N(g)+N(h)$ for $g\in\U$ and $h\in U=\U(\bF_{q^n})$.
\item Let $\pr_n\from\U\to\mathbb{G}_a$ be the projection onto the final coordinate $e_n$;  then $\pr_n(L_{q^n}(g))=N(g)^q-N(g)$.
\end{itemize}

The variety $X$ constructed in \S\ref{theunipotentgroup} is not connected. Since $X$ is defined by $\pr_n(L_{q^n}(g))=0$, the second property of $N$ shows that $X$ is the disjoint union of closed subvarieties with equations $N(g)=a$, as $a$ runs through $\bF_q$.  In other words, we have a cartesian diagram of affine varieties over $\fqn$:
\[
\xymatrix{
X \ar[r] \ar[d] & \bF_q \ar[d] \\
\mathbf{U} \ar[r]_N & \mathbb{A}^1
}
\]
(Here $\bF_q$ is to be interpreted as a disjoint union of $q$ points.)  In Remark \ref{r:connected-components-new} below we will see that the fibers of $X$ over the points of $\bF_q$ are geometrically connected.


\subsection{The main result of the section}
\label{main-result-part-1}
Recall from Prop. \ref{stabx} that the stabilizer $x$ in $GL_n(K)\times D^\times\times \cW_K$ is the group $\mathcal{S}$ of triples $(g,b,w)$, where $(g,b)\in\mathcal{N}$ is a lift of $j(w)\in L^\times\backslash \mathcal{N}$.  Let $\mathcal{J}\subset GL_n(K)\times D^\times\times \cW_K$ be the subgroup generated by $\mathcal{L}^\times\times\set{1}$ and $\mathcal{S}$.

The main result of the section concerns an affinoid subset of $\mathcal{M}_{H_0,\infty,C}$ which happens to be $\mathcal{J}$-invariant.  To state it precisely, we need to pin down a certain Frobenius element.  As in the proof of Lemma \ref{Sring}, we have an element $s\in GL_n(K)$ of order $n$ and a uniformizer $\Pi\in\cO_D$, such that conjugation by the pair $(s,\Pi)$ effects the Frobenius automorphism on
$\Delta(L^\times)$.  Thus $(s,\Pi)\in \mathcal{N}$.  By Prop. \ref{jisomorphism}, there exists an element $\Phi\in \cW_K$ for which $j(\Phi)=L^\times(s,\Pi)$.  Note that $\Phi$ is an arithmetic Frobenius element, and $(s,\Pi,\Phi)\in\mathcal{S}$. Then $\mathcal{J}$ is generated by the following subgroups and elements:
\begin{enumerate}
\item The pro-$p$-Sylow subgroup of $\mathcal{L}^\times$, this being the preimage of the group $U=\mathbf{U}(\fqn)$ under the reduction map  $\mathcal{L}^\times\to (\mathcal{L}/\mathfrak{P})^\times$,
\item $(\alpha,\alpha,1)$, where $\alpha\in \cO_L^\times$,
\item $(1,\alpha,\rec_L(\alpha))$, where $\alpha\in L^\times$, and
\item $(s,\Pi,\Phi)$.
\end{enumerate}

Consider the map
\begin{eqnarray*}
\chi\from GL_n(K)\times D^{\times}\times \cW_K &\to& K^\times \\
(g,b,w)&\mapsto & (\det g)(\Nrd_{D/K}(b))^{-1}\rec_K^{-1}(w)^{-1}.
\end{eqnarray*}
We have $\chi(\mathcal{J})=1+\gp_K^m$.

The remainder of the section is devoted to the following theorem.

\begin{thm} \label{existenceofaffinoid}
There exists a $\mathcal{J}$-invariant affinoid subset $\mathcal{Z}\subset \cM_{H_0,\infty,\C}$ whose reduction is characterized by the following cartesian diagram of affine schemes over $\overline{\bF}_q$:
\[
\xymatrix{
\overline{\mathcal{Z}} \ar[r] \ar[d] & 1+\gp_K^m \ar[d] \\
X^{\perf}_{\overline{\FF}_q} \ar[r]_N & \bF_q
}
\]
Here $1+\gp_K^m$ is to be interpreted as the affine scheme $\Spec \Cont(1+\gp_K^m,\overline{\bF}_q)$, and similarly with $\bF_q\isom (1+\gp_K^m)/(1+\gp_K^{m+1})$.  The group $\mathcal{J}$ acts on all objects in this diagram.  We describe first the action of $\mathcal{J}$ on $X^{\perf}_{\overline{\FF}_q}$.   The action of the pro-$p$-Sylow subgroup of $\mathcal{L}^\times\times\set{1}$ factors through the right multiplication action of $U$ on $X$.  If $\alpha\in \cO_L^\times$,
then $(\alpha,\alpha,1)$ acts through
\[(a_1,\dots,a_n)\mapsto (\overline{\alpha}^{q-1}a_1,\dots,\overline{\alpha}^{q^{n-1}-1}a_{n-1},a_n), \]
where $\overline{\alpha}$ is the image of $\alpha$ in $\cO_L/(\varpi)=\bF_{q^n}$.
Elements of $\mathcal{S}$ of the form $(1,\alpha,\rec_L(\alpha))$ with $\alpha\in L^\times$ act trivially.  Finally, $(s,\Pi,\varphi)$ acts on $X^{\perf}_{\overline{\FF}_q}$ as the inverse of the arithmetic Frobenius map (the variables are fixed but scalars get raised to the $1/q$th power).

The action of $\mathcal{J}$ on $1+\gp_K^m$ is through $\chi\from \mathcal{J}\to 1+\gp_K^m$, and similarly for the quotient $\bF_q\approx (1+\gp_K^m)/(1+\gp_K^{m+1})$.

All arrows in the diagram are equivariant for the action of $\mathcal{J}$.
\end{thm}

The following corollary reduces the study of the cohomology of $\overline{\mathcal{Z}}$ to the study of the cohomology of the (finite-type) variety $X$:

\begin{cor}  \label{from-X-to-Z} As a representation of $\mathcal{J}$ we have
\[ H^\bullet_c(\overline{\mathcal{Z}},\overline{\bQ}_\ell) = \bigoplus_\psi H^\bullet_c(X_{\overline{\bF}_q},\overline{\bQ}_\ell) \otimes (\psi\circ\chi),\]
where $\psi$ runs over characters of $1+\gp_K^m$.
\end{cor}

\begin{proof}[Proof of Cor. \ref{from-X-to-Z}]  The description of $\overline{\mathcal{Z}}$ in Thm. \ref{existenceofaffinoid} shows that it is an inverse limit of schemes of finite type:
\[ \overline{\mathcal{Z}}=\varprojlim_{r} X^{\perf}_{\overline{\FF}_q}\times_{(1+\gp_K^m)/(1+\gp_K^{m+1})} (1+\gp_K^m)/(1+\gp_K^{m+r}). \]
The transition maps in this inverse system are all affine, as are the schemes themselves.  Formation of cohomology therefore commutes with the limit (see \cite{SGA4}, VII.5.8), and we find
\begin{eqnarray*}
H^{*}_c(\overline{\mathcal{Z}},\overline{\bQ}_\ell)
&=& \lim_{r\to\infty}H^\bullet_c\left(X^{\perf}_{\overline{\FF}_q}\times_{(1+\gp_K^m)/(1+\gp_K^{m+1})} (1+\gp_K^m)/(1+\gp_K^{m+r}),\overline{\bQ}_\ell\right) \\
&=& \lim_{r\to\infty} H^{*}_c(X^{\perf}_{\overline{\FF}_q},\overline{\bQ}_\ell) \otimes_{\overline{\bQ}_\ell[(1+\gp_K^m)/(1+\gp_K^{m+1})]} \overline{\bQ}_\ell[(1+\gp_K^m)/(1+\gp_K^{m+r})]
\end{eqnarray*}
This is an isomorphism of representations of $\mathcal{J}$;  recall that the action of $\mathcal{J}$ on $1+\gp_K^m$ is through the norm map $\chi$.  The representation $\overline{\bQ}_\ell[(1+\gp_K)^m/(1+\gp_K^{m+r})]$ breaks up as the direct sum of characters $\psi$ of $1+\gp_K^m$ of conductor $m+r$.  In the direct limit we get
\[ H^\bullet_c(\overline{\mathcal{Z}},\overline{\bQ}_\ell)=\bigoplus_{\psi} H^\bullet_c(X^{\perf}_{\overline{\FF}_q},\overline{\bQ}_\ell)\otimes (\psi\circ\chi),\]
where $\psi$ runs over characters of $1+\gp_K^m$.

Finally, we observe that $X^{\perf}_{\overline{\FF}_q}$ is itself an inverse limit of schemes, namely $X^{\perf}_{\overline{\FF}_q}=\varprojlim X_{\overline{\FF}_q}$ along the $q$th power Frobenius maps.  Since the Frobenius map induces an isomorphism on cohomology we find a natural isomorphism $H^\bullet_c(X^{\perf}_{\overline{\FF}_q},\overline{\bQ}_\ell)=H^\bullet_c(X,\overline{\bQ}_\ell)$.
\end{proof}

\subsection{Definition of the special affinoid} \label{affinoiddefinition} Let $H$ be the standard formal $\cO_K$-module over $\cO_{\breve{K}}$.  Let $y$ be a primitive element of the Tate module $TH(\cO_\C)$.  Let $(\xi,\xi^{1/q},\dots)$ be the corresponding element of $\Nil^{\flat}(\cO_\C)$.  Then $\xi$ is characterized by the properties
\[ \sum_{i=-\infty}^{\infty} \frac{\xi^{q^{in}}}{\varpi^i} = 0,\; \abs{\xi}^{q^n-1}=\abs{\varpi}. \]
Let $L/K$ be the unramified extension of degree $n$.  Recall that $\End H=\cO_L$.   Let $\alpha_1,\dots,\alpha_n$ be a basis for $\cO_L/\cO_K$.
Then $(\alpha_1y,\dots,\alpha_ny)$ is an $n$-tuple of elements of $\tilde{H}(\cO_\C)$ representing a point $x\in\cM_{H_0,\infty}(\C)$ with CM by $L$.

 Let $t=\delta(\alpha_1y,\dots,\alpha_ny)\in\widetilde{\wedge H}(\cO_\C)$, so that $t$ represents a point of $\cM_{\wedge H_0,\infty}(\C)$.
 Suppose $(\tau,\tau^{1/q},\dots)\in\Nil^{\flat}(\cO_C)$ corresponds to $t$, so that
 \[ \sum_{i=-\infty}^{\infty} (-1)^{i(n-1)}\frac{\tau^{q^i}}{\varpi^i}=0,\; \abs{\tau}^{q-1}=\abs{\varpi}.\]
\begin{lem}
\label{tauformula}
We have $\tau=\det(\alpha_i^{q^j})\xi^{1+q+\dots+q^{n-1}}$ plus smaller terms.
\end{lem}
\begin{proof} Follows from Lemma \ref{Deltaestimate}.
\end{proof}

For a row vector $\beta=(\beta_1,\dots,\beta_n)\in D^n$, there is the linear functional $\tilde{H}^n\to\tilde{H}$ which sends $y=(y_1,\dots,y_n)$ to $\sum_i \beta_i y_i$;  we write this as $\beta\cdot y$.  If $a\in M_n(K)$ we can write $\beta a$ for the matrix product $(\beta_1,\dots,\beta_n)a$.  Similarly if $b\in D$ we can write $\beta b$ for $(\beta_1 b,\dots,\beta_n b)$.  Finally if $\zeta=(a,b)\in M_n(K)\times D$, we write $\beta\zeta=\beta a - \beta b$.

\begin{lem}[\cite{WeinsteinSemistableModels}, Lemma 4.2.1]
\label{betavector}
There is a nonzero vector $\beta=(\beta_1,\dots,\beta_n)\in L^n$ such that for all $\alpha\in L$ we have $\beta\Delta(\alpha)=0$.
This $\beta$ is unique up to scaling by an element of $L^\times$.  For all $\zeta\in M_n(K)\times D$ we have
\[ (\beta\zeta)\cdot x=(\pr(\zeta)\beta)\cdot x,\;\text{all $\zeta\in M_n(K)\times D$.}\]
\end{lem}


Let us adopt the following convention regarding continuous valuations $\abs{\;}$ on the ring $R$:  for a section $v\in \tilde{H}(R)$, $\abs{v}$ shall mean $\abs{\lambda_0(v)}$, where we recall from \S\ref{calculationsHtilde} that $\lambda_0(v)\in R$ is the $0$th coordinate of $\lambda(v)\in \Nil^{\flat}(R)$.

Let $R$ be the adic $\cO_\C$-algebra which represents $\tilde{H}^n_{\cO_\C}$.  Thus we have $n$ universal elements $X_1,\dots,X_n\in \tilde{H}(R)$.  Let $X=(X_1,\dots,X_n)\in\tilde{H}^n(R)$.   Choose a nonzero vector $\beta$ as in Lemma \ref{betavector}.    The following definition appears in \cite{WeinsteinSemistableModels}, \S4.2.
Define a subset $\mathcal{Y}\subset \tilde{H}^n_{\eta,C}$ by the inequalities
\begin{equation}
\label{defY1}
\abs{\beta\cdot(X-x)} \leq \abs{\varpi_L^m\beta\cdot x}
\end{equation}
and
\begin{equation}
\label{defY2}
\abs{(\beta \zeta)\cdot X}\leq \abs{\varpi_L^m\beta\cdot x},\text{ all $\zeta\in \cL$}.
\end{equation}

Then we define $\mathcal{Z}=\mathcal{Y}\cap \mathcal{M}_{H_0,\infty,\C}$.  This is the special affinoid of Thm. \ref{existenceofaffinoid}.  The next two sections are devoted to the proof of that theorem.

\subsection{The special affinoid:  Case of $m\geq 2$ even}
We have $x=(\alpha_1 y,\alpha_2 y,\dots,\alpha_n y)$ for a basis $\alpha_1,\dots,\alpha_n$ of $\cO_L/\cO_K$. Let $\varphi$ be the Frobenius automorphism of $L/K$.  If $\alpha=(\alpha_1,\dots,\alpha_n)^T$ is the column vector whose components are the $\alpha_i$, then there exists a unique matrix $s\in GL_n(\cO_K)$ for which $\alpha^{\varphi}=s\alpha$.  Then $s$ has order $n$, and conjugation by $s$ induces $\varphi$ on the embedded subfield $L$ of $M_n(K)$.

Let $\beta_1,\dots,\beta_n\in \cO_L$ be elements such that
\[ \sum_i \beta_i\alpha_i^{\varphi^j} = \delta_{j0} \]
for $j=0,1,\dots,n-1$.  Then $\beta=(\beta_1,\dots,\beta_n)$ is an element of the sort described in Lemma \ref{betavector}, and $\beta\cdot x = y$.  We remark that $\alpha$ and $\beta$ are related by the identity
\[
\begin{pmatrix}
\alpha_1 & \alpha_1^\varphi & \cdots & \alpha_1^{\varphi^{n-1}} \\
\alpha_2 & \alpha_2^\varphi & \cdots & \alpha_2^{\varphi^{n-1}} \\
\vdots   &            & \ddots & \cdots     \\
\alpha_n & \alpha_n^\varphi & \cdots & \alpha_n^{\varphi^{n-1}}
\end{pmatrix}^{-1}
=
\begin{pmatrix}
\beta_1 & \beta_2 & \cdots & \beta_n \\
\beta_1^{\varphi} & \beta_2^{\varphi} & \cdots & \beta_n^{\varphi} \\
\vdots &  & \ddots & \vdots\\
\beta_1^{\varphi^{n-1}} & \beta_2^{\varphi^{n-1}} & \cdots & \beta_n^{\varphi^{n-1}}.
\end{pmatrix}.
\]
We also remark that
\begin{equation}
\label{betas}
\beta  = \beta^{\phi^r}s^r,\; r\in\Z
\end{equation}

Recall that $C_1$ is the complement of $L$ in $M_n(K)$ under the trace pairing, and $C_1^\circ=C_1\cap M_n(\cO_K)$.  Then $C_1^\circ$ is the $\cO_L$-submodule of $M_n(\cO_K)$ spanned by $s,s^2,\dots,s^{n-1}$.
Define elements $A_0,\dots,A_{n-1}\in\tilde{H}(R)$ by
\[ A_i=\beta s^{-i}\cdot X,\;i=0,\dots,n-1.\]
Then the $X$s and $A$s are related by
\begin{equation}
\label{equation-X-A}
\begin{pmatrix}
X_1 \\
X_2 \\
\vdots \\
X_n
\end{pmatrix}
=
\begin{pmatrix}
\alpha_1 & \alpha_1^\varphi & \cdots & \alpha_1^{\varphi^{n-1}} \\
\alpha_2 & \alpha_2^\varphi & \cdots & \alpha_2^{\varphi^{n-1}} \\
\vdots   &            & \ddots & \cdots     \\
\alpha_n & \alpha_n^\varphi & \cdots & \alpha_n^{\varphi^{n-1}}
\end{pmatrix}
\begin{pmatrix}
A_0 \\
A_1 \\
\vdots \\
A_{n-1}
\end{pmatrix}
\end{equation}

If we define elements $Y_1,\dots,Y_n\in \tilde{H}(R)$ by setting
\begin{eqnarray*}
Y_i&=&\varpi^{-m/2}\beta s^{-r}\cdot X,\; i=1,\dots,n-1\\
Y_n&=&\varpi^{-m}\beta \cdot (X-x),
\end{eqnarray*}
then each $Y_i$ satisfies $\abs{Y_i}\leq \abs{y}$ for all $\abs{\;}\in \mathcal{Y}$.  The relationship between the $Ys$ and the $As$ is
\begin{eqnarray*}
A_0 &=& x+\varpi^m Y_n\\
A_1 &=& \varpi^{m/2} Y_1\\
&\vdots& \\
A_{n-1} &=& \varpi^{m/2} Y_{n-1}.
\end{eqnarray*}

Write $\lambda_0(Y_i)=\xi Z_i$, so that $Z_i\in R\otimes C$.  We have $\mathcal{Y}=\Spa(S,S^+)$, where
\[ S^+=\cO_{\C}\tatealgebra{Z_1^{1/q^\infty},\dots,Z_n^{1/q^\infty}}. \]
The reduction of $\mathcal{Y}$ is $\Spec\overline{\bF}_q[Z_1^{1/q^\infty},\dots,Z_n^{1/q^\infty}]=\mathbf{A}^{n,\perf}_{\overline{\bF}_q}$, this being the perfection of affine $n$-space over $\overline{\bF}_q$.

\begin{thm}  \label{approximationfordelta} Let $\wedge\mathcal{Y}\subset \wedge \tilde{H}_{\eta,C}$ be the image of $\mathcal{Y}$ under $\delta$.  There is a diagram
\[
\xymatrix{
\overline{\mathcal{Y}} \ar[r] \ar[d]_{\sim} & \overline{\wedge \mathcal{Y}} \ar[d]^{\sim} \\
\mathbf{U}^{\perf}_{\overline{\bF}_q} \ar[r]_{N} & \mathbf{A}^{1,\perf}_{\overline{\bF}_q}
}
\]
in which the vertical arrows are isomorphisms, and the lower horizontal map is induced from the morphism $N\from \mathbf{U}\to\mathbb{A}^1$ of Eq. \eqref{NU}.  This diagram commutes up to sign.
\end{thm}

\begin{proof}
By Prop. \ref{logcommutes}, the morphism $\delta\from \tilde{H}^n\to\widetilde{\wedge H}$ is characterized by the property that $\log_{\wedge H}\delta(X_1,\dots,X_n)=\det\qlog_H(X_1,\dots,X_n)$.  To prove the theorem we will undertake an analysis of $\qlog_H(X_1,\dots,X_n)$, in terms of the integral coordinates $Z_1,\dots,Z_n$ on $\mathcal{Y}$.

In the ring $S^+$ we have the congruences
\begin{eqnarray*}
\log_H A_0 &\equiv& \varpi^m \xi(Z_n-Z_n^{q^n}) \\
\log_H\varpi A_0 &\equiv& \xi^q \\
&\vdots&\\
\log_H\varpi^{n-1} A_0 &\equiv &  \xi^{q^{n-1}},
\end{eqnarray*}
and, for $i=1,\dots,n-1$,
\begin{eqnarray*}
\log_H A_i &\equiv& \xi(Z_i^{q^n}-Z_i) \\
\log_H \varpi A_i &\equiv& \xi^q Z_i^q \\
&\vdots&\\
\log_H \varpi^{n-1}A_i&\equiv& \varpi^{m/2}\xi^{q^{n-1}}Z_i^{q^{n-1}}.
\end{eqnarray*}
Here $a\equiv b$ is taken to mean ``moduli smaller terms."  We have used the fact that $\xi^{q^n}\equiv \varpi\xi$.

By Lemma \ref{quasiloglemma}, the coordinates of $\qlog_H(X_i)$ with respect to the standard basis of $M(H_0)\otimes K$ are given by
\[ \qlog_H(X_i)=\left(\log_H(X_i),\log_H(\varpi X_i),\dots,\log_H(\varpi^{n-1} X_i)\right).\]
This gives us an expression for the matrix $\qlog_H(X_i)_{i=1,\dots,n}\in M(H_0)^n\otimes S$ in terms of the variables $A_0,\dots,A_{n-1}$:
\[
(\qlog(X_i))_i=\left(\alpha_i^{q^j}\right)
\begin{pmatrix}
\log_HA_0 &  \log_H \varpi A_{n-1} & \cdots & \log_H\varpi^{n-1}A_1 \\
\log_HA_1 & \log_H\varpi A_0     & \cdots & \varpi \log_H\varpi^{n-1} A_2 \\
\vdots & \vdots & \ddots & \vdots \\
\log_HA_{n-1} & \log_H\varpi A_{n-2} & \cdots & \log_H\varpi^{n-1} A_0
\end{pmatrix}
\]
Now take determinants.  We apply the preceding congruences for $\log_H\varpi^iA_j$ together with Lemma \ref{tauformula} to find that $\det\qlog_H(X_i)$ equals
\[
\varpi^m\tau\det\begin{pmatrix}
Z_n-Z_n^{q^n} & Z_{n-1}^q & Z_{n-2}^{q^2} & \cdots & Z_2^{q^{n-2}} & Z_1^{q^n} \\
Z_1-Z_1^{q^n} & 1         & 0             & \cdots & 0             & 0         \\
Z_2-Z_2^{q^n} & 0         & 1             & \cdots & 0             & 0         \\
\vdots        &           &               & \ddots &               & \vdots    \\
Z_{n-1}-Z_{n-1}^{q^n} & 0 & 0             & \cdots & 0             & 1
\end{pmatrix}
\]
plus smaller terms in $S^+$.  Let $f(Z_1,\dots,Z_n)$ denote the determinant appearing in the above equation.  Remarkably, the equation $f=0$ cuts out the variety $X$ defined in \S\ref{theunipotentgroup}.   Since $\log_{\wedge H}\delta(X_1,\dots,X_n)=\det\qlog_H(X_1,\dots,X_n)$, we have
\[ \log_{\wedge H}\delta(X_1,\dots,X_n)\equiv \varpi^m\tau f(Z_1,\dots,Z_n) \]
modulo smaller terms in $S^+$.

Let $\delta=(\delta_0,\delta_1,\dots)=\delta(X_1,\dots,X_n)\in\widetilde{\wedge H}(S^+)$;  we intend to use the above congruence to give an approximation to $\delta$ in terms of $Z_1,\dots,Z_n$.  Let $\exp_{\wedge H}(T)\in K\powerseries{T}$ be the exponential series of $\wedge H$;  this series belongs to $\tau\cO_{\breve{K}}\tatealgebra{T/\tau}$.  Since $\log_{\wedge H}\delta\in \varpi^m\tau S^+$, we have that $\log_{\wedge H}\delta_m\in\tau S^+$, and therefore $\exp_{\wedge H}\log_{\wedge H}\delta_m$ converges to an element $\delta'_m\in H(S^+)$, and we have $
\log_{\wedge H}(\delta_m-\delta'_m)=0$.  Thus $\delta_m-\delta'_m$ belongs to $T(\wedge H)(S^+)=T(\wedge H)(\cO_C)$ (since $\Spec S^+$ is connected).  The homomorphism $S^+\to \cO_C$ carrying each $Z_i$ to $0$ takes $\delta_m-\delta'_m$ to $t_m$, so that in fact $\delta_m=t_m+\delta'_m$.  Let $\delta'=\delta-t\in\tilde{H}(S^+)$;  then
\[\delta_0=\varpi^m\delta_m'. \]
We have $\delta_m'=\tau g(Z_1,\dots,Z_n)$ for some $g\in S^+=\cO_{\C}\powerseries{Z_1^{1/q^\infty},\dots,Z_n^{1/q^\infty}}$ without constant terms.


Taking logarithms gives
\[\log_{\wedge H}\delta \equiv \varpi^m\tau (g - g^q) \equiv \varpi^m\tau f \]
modulo smaller terms in $S^+$.
We can conclude from this that $g-g^q=f$ as elements of $\overline{\bF}_q[Z_1^{1/q^\infty},\dots,Z_n^{1/q^\infty}]$.  This means that (in this ring) $g=-N$.
\end{proof}

With Thm. \ref{approximationfordelta}, we can return to the proof of Thm. \ref{existenceofaffinoid}.  By Thm. \ref{LTdiagram}, there is a cartesian diagram
\[
\xymatrix{
\mathcal{M}_{H_0,\infty,\C} \ar[r]^{\delta}\ar[d]  & \mathcal{M}_{\wedge H_0,\infty,\C} \ar[d] \\
\tilde{H}_{\eta,C}^n \ar[r]_{\delta} & \widetilde{\wedge H}_{\eta,C}
}.
\]
Let $\mathcal{Z}$ be the preimage of $\mathcal{Y}$ under $\mathcal{M}_{H_0,\infty,\C}\to\tilde{H}^n_{\eta,C}$.  Every object in the above diagram has a corresponding affinoid subset in the following diagram:
\[
\xymatrix{
\mathcal{Z} \ar[r]^{\delta} \ar[d] & \wedge\mathcal{Z} \ar[d] \\
\mathcal{Y} \ar[r]_{\delta}        & \wedge\mathcal{Y}
}
\]
Here $\wedge\mathcal{Z}\subset \mathcal{M}_{\wedge H,\infty,\C}$ is (set-theoretically) the set of translates of $t\in T(\wedge H)(\cO_C)$ by $1+\gp_K^m$.  Now take reductions of all objects to get a diagram of affine schemes over $\overline{\bF}_q$:
\[
\xymatrix{
\overline{\mathcal{Z}} \ar[r]^{\delta} \ar[d] & \overline{\wedge\mathcal{Z}} \ar[d] \\
\overline{\mathcal{Y}} \ar[r]_{\delta}        & \overline{\wedge\mathcal{Y}}
}
\]
(This diagram is also cartesian, because the fiber product of $\overline{\mathcal{Y}}$ and $\overline{\wedge\mathcal{Z}}$ over $\overline{\wedge\mathcal{Y}}$ is reduced.)  The scheme $\overline{\wedge\mathcal{Z}}$ is isomorphic to the ``constant scheme" $1+\gp_K^m$ (meaning $\Spec\Cont(1+\gp_K^m,\overline{\bF}_p)$).
Now apply Thm. \ref{approximationfordelta}:  the bottom arrow is isomorphic to $N\from \mathbf{U}^{\perf}\to\mathbb{A}^{1,\perf}$.  Thus we have a cartesian diagram of schemes over $\overline{\bF}_q$:
\[
\xymatrix{
\overline{\mathcal{Z}} \ar[r] \ar[d] & 1+\gp_K^m \ar[d]  \\
\mathbf{U}^{\perf}\ar[r]_N & \mathbb{A}^{1,\perf}
}
\]
On the other hand, the variety $X$ appears in a cartesian diagram
\[
\xymatrix{
X \ar[r] \ar[d] & \bF_q \ar[d] \\
\mathbf{U} \ar[r]_N & \mathbb{A}^1
}
\]
Taking the perfection of second diagram and combining it with the first gives a cartesian diagram
\[
\xymatrix{
\overline{\mathcal{Z}} \ar[r] \ar[d] & 1+\gp_K^m \ar[d] \\
X^{\perf}_{\overline{\FF}_q} \ar[r] & \bF_q,
}
\]
thus establishing the first part of Thm. \ref{existenceofaffinoid}.

We now establish that the group $\mathcal{J}\subset GL_n(K)\times D^\times\times \cW_K$ stabilizes $\mathcal{Z}$ and induces the action on $\overline{\mathcal{Z}}$ described in Thm. \ref{existenceofaffinoid}.    Since $\mathcal{Z}=\mathcal{Y}\cap \mathcal{M}_{H_0,\infty,C}$, we can just show that $\mathcal{J}$ preserves $\mathcal{Y}$, and that it induces the correct action on $\overline{\mathcal{Y}}\isom \mathbb{A}_{\overline{\mathbb{F}}_q}^{n,\text{perf}}$.

\begin{prop} The action of $\mathcal{J}$ on $\tilde{H}^n_{\eta,C}$ preserves $\mathcal{Y}$.
\end{prop}

\begin{proof} Recall that $\mathcal{J}$ is the subgroup of $GL_n(K)\times D^\times \times \cW_K$ generated by $\mathcal{L}^\times\times\set{1}$ and $\mathcal{S}$.  The statement that $\Delta(L)^\times\mathcal{L}^\times\times\set{1}$ stabilizes $\mathcal{Y}$ is Lemma 4.2.2 of \cite{WeinsteinSemistableModels}.

It remains to show that $\mathcal{Y}$ is stabilized by $\mathcal{S}$. Before doing so it will be helfpul to work through the action of the Frobenius element $\Phi$ on $H_{\cO_C}^n$, following \S\ref{groupactions}.  Recall that $R$ is the adic $\cO_C$-algebra which represents $H^n_{\cO_C}$, so that we have a universal element $X=(X_1,\dots,X_n)\in\tilde{H}^n(R)$.  The right action of $\cW_K$ on $H^n_{\cO_C}$ induces a left action on $R$ and therefore on $\tilde{H}^n(R)$.  Tracing through the definitions, we see that the Frobenius element $\Phi$ carries each $X_i$ onto $\Pi^{-1}X_i$.  Generally, if $w\in \cW_K$ induces the $q^r$th power Frobenius on the residue field, then $w$ carries $X_i$ onto $\Pi^{-r}X_i$.

To show that $\mathcal{S}$ stabilizes $\mathcal{Y}$, it is enough to consider triples of the form $(s^r,\Pi^ru,w)$,
where $u\in \cO_D^\times$ and $w\in \cW_K$ induces the $q^r$th power Frobenius on the residue field.
Suppose the
inequalities in Eqs. \eqref{defY1} and \eqref{defY2} hold at a point $\abs{\;}$ of $\tilde{H}^n_{\eta,C}$.  We claim they hold at the translate $\abs{\;}^{(s^r,\Pi^ru,w)}$ as well.
Start with Eq. \eqref{defY1}.  Since $(s^r,\Pi^ru,w)$ fixes $x$, it follows that $x^w=s^ru^{-1}x$.
We have
\begin{eqnarray*}
\abs{\beta\cdot (X-x)}^{(s^r,\Pi^ru,w)}&=&
\abs{\beta^{\phi^r}\cdot (s^r)^Tu^{-1}(X-x)}\\
&=&\abs{\beta^{\phi^r}s^ru^{-1} \cdot (X-x)}\\
&=&\abs{\beta u^{-1}\cdot (X-x)} \; \text{by Eq. \eqref{betas}}\\
&=&\abs{\beta \cdot (X-x)}.
\end{eqnarray*}
(The last step is justified because, being a unit, $u^{-1}$ is congruent to an element of $\FF_{q^n}^\times$ modulo $\Pi$.)   This establishes the claim for Eq. \eqref{defY1}.   The argument for Eq. \eqref{defY2} is similar.
\end{proof}

\begin{lem} \label{stabilizesorigin} The action of $\mathcal{J}$ on the identity element  of $\overline{\mathcal{Y}}\isom \mathbf{U}^{\text{perf}}$ is as described by Thm. \ref{existenceofaffinoid}.
\end{lem}

\begin{proof} The reduction of the CM point $x\in \mathcal{Y}$ corresponds to the identity element of $\mathbf{U}^{\text{perf}}$.

Recall the coordinates $Y_1,\dots,Y_n$, defined by $Y_i=\varpi^{-m/2}\beta s^{-i}\cdot X$ for $i=1,\dots,n-1$ and $Y_n=\varpi^{-m}\beta\cdot (X-x)$.  We verify the claim for those elements of $\mathcal{J}$ belonging to $1+P_{1,m}$;  the calculation for the other generators of $\mathcal{J}$ is simpler.  Let $1+\overline{a}_1\tau+\dots+\overline{a}_n\tau^n$ be an element of $U$, and let $g=1+\varpi^{m/2}a_1s+\dots+\varpi^{m/2}a_{n-1}s^{n-1}+\varpi^ma_n$ be a lift of it to $1+P_{1,m}\subset GL_n(K)$.  For $i=1,\dots,n-1$ we have (using Lemma \ref{betavector})
\begin{eqnarray*}
Y_i(x^g) &=& \varpi^{-m/2}\beta s^{-i}\cdot x^g \\
&=& \varpi^{-m/2}\beta s^{-i} g \cdot x \\
&=& \varpi^{-m/2}\pr_1(s^{-i}g)(\beta\cdot x) \\
&=& a_iy,
\end{eqnarray*}
and similarly $Y_n(x^g)=a_ny$.  From this we see that $Z_i(x^g)=a_i$ for $i=1,\dots,n$.  This shows that the reduction of $x^g$ in $\overline{Y}$ is the point $1+a_1\tau+\dots+a_n\tau^n$ as required.
\end{proof}

Let $\rho\from\mathcal{J}\to \Aut \mathbb{A}^{n,\text{perf}}_{\overline{\mathbb{F}}_q}$ be the action of $\mathcal{J}$ on the reduction of the affinoid $\mathcal{Y}\isom \mathbb{A}^{n,\text{perf}}_{\overline{\mathbb{F}}_q}$.
We can now show that $\rho$ agrees with the action described in Thm. \ref{existenceofaffinoid}.  That theorem describes an action $\rho'\from\mathcal{J}\to\Aut \mathbb{A}^{n,\text{perf}}_{\overline{\mathbb{F}}_q}$.  The preceding lemma shows that for all $g\in \mathcal{J}$, $\rho'(g)\rho(g)^{-1}$ fixes the origin in $\mathbb{A}^{n,\text{perf}}_{\overline{\mathbb{F}}_q}$ and all of its translates, which is to say it fixes the subset $\mathbb{A}^n(\mathbb{F}_{q^n})\subset\mathbb{A}^n(\mathbb{F}_{q^n})$.  We also know that since $\rho$ and $\rho'$ are both $\overline{\mathbb{F}}_q$-semilinear in the same way, so that $\rho'(g)\rho(g)^{-1}$ is actually $\overline{\mathbb{F}}_q$-linear.   Finally, we know that $g$ descends to an automorphism of the Lubin-Tate tower at some finite level, which means that $\rho(g)$ descends to an automorphism of $\mathbb{A}^n_{\overline{\FF}_q}$.  These considerations show that $\rho'(g)=\rho(g)$.

\subsection{The special affinoid:  Case of $m\geq 1$ odd}
Once again, define elements $A_0,\dots,A_{n-1}\in\tilde{H}(R)$ by the system of equations
\[
X_i=
\alpha_i A_0 + \alpha_i^q A_1 + \dots + \alpha_i^{q^{n-1}}A_{n-1}, \]
($i=1,\dots,n$).  Now define elements $Y_1,\dots,Y_n\in \tilde{H}(R)$ by setting
\begin{eqnarray*}
A_0 &=& x+\varpi^m Y_n\\
A_1 &=& \varpi^{(m-1)/2} \Pi Y_1\\
&\vdots& \\
A_{n-1} &=& \varpi^{(m-1)/2} \Pi^{n-1} Y_{n-1}.
\end{eqnarray*}

Now consider the elements $\log_H\varpi^iA_j$.  These lie in $S^+$, and in that ring we have the congruences
\begin{eqnarray*}
\log_H A_0 &\equiv& \varpi^m \xi(Z_n^{q^n}-Z_n) \\
\log_H\Pi A_0 &\equiv& \varpi^{(m-1)/2} \xi^q \\
&\vdots&\\
\log_H\Pi^{n-1} A_0 &\equiv &\varpi^{(m-1)/2} \xi^{q^{n-1}}
\end{eqnarray*}
and, for $i=1,\dots,n-1$,
\begin{eqnarray*}
\log_H A_i &\equiv& \varpi^{(m-1)/2} \xi^{q^i}Z_i^{q^i} \\
\log_H \Pi A_i&\equiv & \varpi^{(m-1)/2} \xi^{q^{i+1}}Z_i^{q^{i+1}} \\
&\vdots&\\
\log_H \Pi^{n-i} A_i &\equiv & \varpi^{(m+1)/2} \xi (Z_i-Z_i^{q^n})\\
&\vdots& \\
\log_H \Pi^{n-1}A_i &\equiv & \varpi^{(m+1)/2} \xi^{q^{i-1}}Z_i^{q^{i-1}}.
\end{eqnarray*}
%
The determinant of $(\qlog(X_i))$ equals
{\footnotesize
\[
\tau
\det
\begin{pmatrix}
\varpi^{m}(Z_n^{q^n}-Z_n) & \varpi^{\frac{m+1}{2}}(Z_{n-1}^{q^n}-Z_{n-1}) &\cdots
& \varpi^{\frac{m+1}{2}}(Z_{2}^{q^n}-Z_{2}) & \varpi^{\frac{m+1}{2}}(Z_1^{q^n}-Z_1) \\
\varpi^{\frac{m-1}{2}}Z_1^q & 1 & \cdots & \varpi^{\frac{m+1}{2}}Z_{3}^q & \varpi^{\frac{m+1}{2}}Z_{2}^q \\
\varpi^{\frac{m-1}{2}}Z_{2}^{q^2} & \varpi^{\frac{m-1}{2}}Z_1^{q^2} & \cdots & \varpi^{\frac{m+1}{2}}Z_{4}^{q^2} &
\varpi^{\frac{m+1}{2}}Z_{3}^{q^2} \\
\vdots & &\ddots & & \vdots \\
\varpi^{\frac{m-1}{2}}Z_{n-1}^{q^{n-1}} & \varpi^{\frac{m-1}{2}}Z_{n-2}^{q^{n-1}}
& \cdots & \varpi^{\frac{m-1}{2}}Z_1^{q^{n-1}} & 1
\end{pmatrix}
\]
}
plus smaller terms.  Up to smaller terms (and up to sign), this latter determinant equals $\varpi^m\tau $ times
\begin{equation}
\label{bigdet3}
\det
\begin{pmatrix}
Z_n^{q^n}-Z_n & Z_{n-1}^{q^n}-Z_{n-1} & \cdots & Z_2^{q^n}-Z_2 &
Z_1^{q^n} - Z_1 \\
Z_1^q & 1 & \cdots & 0 & 0 \\
Z_2^{q^2} & Z_1^{q^2} & \cdots & 0 & 0 \\
\vdots & &  \ddots &  & \vdots\\
Z_{n-1}^{q^{n-1}} & Z_{n-2}^{q^{n-1}} & \cdots & Z_1^{q^{n-1}} & 1
\end{pmatrix}
\end{equation}
if $m=1$, and
\begin{equation}
\label{bigdet4}
\det
\begin{pmatrix}
Z_n^{q^n}-Z_n & Z_{n-1}^{q^n}-Z_{n-1} & \cdots & Z_2^{q^n}-Z_2 &
Z_1^{q^n} - Z_1 \\
Z_1^q & 1 & \cdots & 0 & 0 \\
Z_2^{q^2} & 0 & \cdots & 0 & 0 \\
\vdots & &  \ddots &  & \vdots\\
Z_{n-1}^{q^{n-1}} & 0 & \cdots & 0 & 1
\end{pmatrix}
\end{equation}
if $m\geq 3$.  Let $f(z_1,\dots,z_n)\in\bF_{q^n}[z_1,\dots,z_n]$ be the polynomial defined in Eqs. \eqref{bigdet3} and \eqref{bigdet4}.  We have the congruence
\begin{equation}
\label{detqlog}
\det(\qlog(X_1),\dots,\qlog(X_n))\equiv \varpi^m\tau f(Z_1,\dots,Z_n)
\end{equation}
modulo smaller terms in $S^+$.  Once again, $f=0$ cuts out the variety $X$ (see Eqs. \eqref{hyp3} and \eqref{hyp4}.)  We now proceed exactly as in the case of $m$ even to complete the proof of Thm. \ref{existenceofaffinoid}.


\part{Deligne-Lusztig theory for certain unipotent groups}

\section{Formulation of the results}\label{s:DL-formulation}

\subsection{Overview} This part can be read essentially independently of the rest of the article. In it we formulate and prove a more precise version of Theorem B stated in the Introduction. We use the methods developed in \cite[\S2]{DLtheory}. A special case of Theorem B is also proved in the preprint \cite{SpecialCasesCorrespondences}, to which we refer the reader who would first like to see how our approach works in a simpler setting. However, for the purpose of proving Theorem A, the full strength of Theorem B is needed, and the arguments that appear in the current part of the article do not rely on \emph{op.~cit.}

\subsection{Notation} Throughout this part of the article, we fix prime numbers $p\neq\ell$, a power $q$ of $p$ and an integer $n\geq 2$. We will freely use the formalism of $\ell$-adic cohomology with compact supports and the standard notation and terminology of that theory. The only nonstandard notation we employ is as follows.

\begin{rems}\label{rems:ell-adic-cohomology}
\begin{enumerate}[(1)]
\item If $X$ is a scheme of finite type over $\bF_q$ and $\cF$ is a (constructible) $\ell$-adic sheaf, we write $H^i_c(X,\cF)$ in place of $R^i\pr_!(\cF)$, where $\pr:X\to\Spec(\bF_q)$ is the structure morphism. With this convention, $H^i_c(X,\cF)$ is an $\ell$-adic sheaf on $\Spec(\bF_q)$, i.e., a continuous finite dimensional representation of $\Gal(\overline{\bF}_q/\bF_q)$ over $\ql$. The underlying vector space of $H^i_c(X,\cF)$ is equal to the compactly supported cohomology $H^i_c(X\tens_{\bF_q}\bfq,\cF)$, and the action of the canonical generator of $\Gal(\bfq/\bF_q)$ on $H^i_c(X,\cF)$ will be denoted by $\Fr_q$.
 \sbr
\item The above conventions apply in particular to the case where $\cF=\ql$ is the constant $\ql$-local system of rank $1$ on $X$.
 \sbr
\item Our normalization of $\Fr_q$ is such that the Tate twist $\ql(1)$ on $\Spec(\bF_q)$ corresponds to a $1$-dimensional vector space over $\ql$ on which $\Fr_q$ acts as $q^{-1}$. So, for example, $H^2_c(\bA^1,\ql)=\ql(-1)$ and $H^i_c(\bA^1,\ql)=0$ for $i\neq 2$, where $\bA^1$ is the affine line over $\bF_q$.
\end{enumerate}
\end{rems}

\subsection{Additive characters of $\bF_{q^n}$}\label{ss:additive-characters} Given a character $\psi:\bF_{q^n}\to\qls$, there is a unique integer $1\leq m\leq n$ (which divides $n$) such that $\psi$ factors through the trace map $\Tr_{\bF_{q^n}/\bF_{q^m}}:\bF_{q^n}\to\bF_{q^m}$ and does not factor through the trace map $\bF_{q^n}\to\bF_{q^k}$ for any $1\leq k<m$. We call $q^m$ the \emph{conductor} of $\psi$. Since $\Tr_{\bF_{q^n}/\bF_{q^m}}$ is surjective, we can write $\psi=\psi_1\circ\Tr_{\bF_{q^n}/\bF_{q^m}}$ for a \emph{unique} character $\psi_1:\bF_{q^m}\to\qls$.

\subsection{Definitions}\label{ss:definitions-U-n-q} In \S\ref{theunipotentgroup} we introduced a unipotent group $\U$ over $\fqn$, a hyperplane $Y\subset\U$ and a smooth hypersurface $X=L_{q^n}^{-1}(Y)\subset\U$, where $L_{q^n}:\U\to\U$ is the Lang morphism $g\mapsto\Fr_{q^n}(g)\cdot g^{-1}$. The definition of $\U$ depends on whether $m=1$ or $m\geq 2$, where $m$ is the positive integer appearing in the formulation of Theorem A from the Introduction. For the sake of brevity, we will treat both cases simultaneously. Since we will need to vary $n$ and $q$ in what follows, we will modify the notation $\U$ to make the dependence on $n$ and $q$ more explicit.

\mbr

We first introduce a (noncommutative) ring object $\cR$ in the category of affine $\fqn$-schemes defined as follows. If $B$ is a commutative $\fqn$-algebra, then $\cR(B)$ is the ring consisting of all formal expressions $a_0+a_1\cdot e_1+\dotsc+a_n\cdot e_n$, which are added in the obvious way and multiplied according to the following rules.

\subsubsection{Case 1}\label{sss:Case1} This case corresponds to the case where $m=1$ in Theorem A:
\begin{itemize}
\item $e_i\cdot a=a^{q^i}\cdot e_i$ for all $1\leq i\leq n$ and all $a\in B$;
 \sbr
\item for all $i,j\geq 1$,
\[
e_i\cdot e_j =
\begin{cases}
e_{i+j} & \text{if } i+j\leq n, \\
0 & \text{otherwise}.
\end{cases}
\]
\end{itemize}

\subsubsection{Case 2}\label{sss:Case2} This case corresponds to the case where $m\geq 2$ in Theorem A:
\begin{itemize}
\item $e_i\cdot a=a^{q^i}\cdot e_i$ for all $1\leq i\leq n$ and all $a\in B$;
 \sbr
\item for all $i,j\geq 1$,
\[
e_i\cdot e_j =
\begin{cases}
e_n & \text{if } i+j=n, \\
0 & \text{otherwise}.
\end{cases}
\]
\end{itemize}

\begin{rem}
In the remainder of this part of the article, the letter $m$ will be used as an auxiliary index, independent of its meaning in Theorem A.
\end{rem}

In both cases the multiplicative group $\cR^\times\subset\cR$ is given by $a_0\neq 0$, and we let $U^{n,q}\subset\cR^\times$ denote the subgroup defined by $a_0=1$. Then $U^{n,q}$ is the unipotent group that was denoted by $\U$ in \S\ref{s:affinoid-LT-tower}. We write $Y\subset U^{n,q}$ for the subvariety defined by $a_n=0$ and we put $X=L_{q^n}^{-1}(Y)$. The finite group $U^{n,q}(\fqn)$ acts on $X$ by right translation, so we obtain a representation of $U^{n,q}(\fqn)$ on $H^i_c(X,\ql)$ (cf.~Remarks \ref{rems:ell-adic-cohomology}) for each $i\in\bZ$, which commutes with the action of $\Fr_{q^n}$.

\begin{rem}\label{r:rings-are-obviously-equal}
By construction, the ring $\cR(\fqn)$ of $\fqn$-valued points of $\cR$ can be identified with the quotient ring $S=\mathcal{L}/\gP$ considered in Lemma \ref{Sring}.
\end{rem}

\begin{rem}\label{r:central-characters-new}
If $Z\subset U^{n,q}$ consists of expressions of the form $1+a_n e_n$, then $Z$ is the center of $U^{n,q}$ and $Z(\bF_{q^n})$ is the center of $U^{n,q}(\bF_{q^n})$. We have $Z\cong\bG_a$, and we often tacitly identify the two groups. In particular, every irreducible representation of $U^{n,q}(\bF_{q^n})$ over $\ql$ has a central character $\bF_{q^n}\to\qls$.
\end{rem}

\begin{rem}\label{r:reduced-norm-new}
We recall from \S\ref{ss:reduced-norm-first-definition} that there is a natural group homomorphism $U^{n,q}(\fqn)\to\bF_q$ (there it was denoted simply by $N$). In this part we will denote it by $\Nm^{n,q}$ and refer to it as the \emph{reduced norm map}. An alternative approach to defining $\Nm^{n,q}$, which is independent of Part 1, can be found in \S\ref{s:reduced-norm}. In particular, Proposition \ref{p:reduced-norm-key} shows that $\Nm^{n,q}$ depends only on whether $m=1$ or $m\geq 2$ in Theorem A, which is not obvious from the original definition.

\mbr

This map plays the following role in the study of representations of $U^{n,q}(\bF_{q^n})$. The restriction of $\Nm^{n,q}$ to $Z(\bF_{q^n})=\bF_{q^n}$ is equal to the trace map $\Tr_{\bF_{q^n}/\bF_q}$. In particular, given a character $\psi:Z(\bF_{q^n})\to\qls$ with conductor $q$, we obtain a preferred extension of $\psi$ to a character of $U^{n,q}(\bF_{q^n})$. Namely, if $\psi=\psi_1\circ\Tr_{\bF_{q^n}/\bF_q}$, where $\psi_1:\bF_q\to\qls$, then $\psi_1\circ\Nm^{n,q}:U^{n,q}(\bF_{q^n})\to\qls$ extends $\psi$.
\end{rem}

\begin{rem}\label{r:inclusion-new}
Suppose that $n=m\cdot n_1$, where $m,n_1\in\bN$, and put $q_1=q^m$, so that $q_1^{n_1}=q^n$. We can consider the unipotent group $U^{n_1,q_1}$ over $\bF_{q^n}$. To avoid confusion, let us temporarily denote its elements by $1+b_1 e'_1+\dotsb + b_{n_1} e'_{n_1}$. We can naturally embed $U^{n_1,q_1}$ as a subgroup of $U^{n,q}$ via the map
\[
1+b_1 e'_1+b_2 e'_2+\dotsb+b_{n_1}e'_{n_1} \longmapsto 1 + b_1 e_m + b_2 e_{2m} + \dotsb + b_{n_1} e_{n}.
\]
From now on we identify $U^{n_1,q_1}$ with its image under this embedding. In particular, we view $U^{n_1,q_1}(\bF_{q^n})$ as the subgroup of $U^{n,q}(\bF_{q^n})$ consisting of all elements of the form $1+\sum_{m\mid j}a_j e_j$, where each $a_j\in\bF_{q^n}$.
\end{rem}

\subsection{A more precise version of Theorem B}

\begin{thm}\label{t:cohomology-of-X}
Fix an arbitrary character $\psi:\bF_{q^n}\rar{}\qls$.
 \sbr
\begin{enumerate}[$($a$)$]
\item There is a unique $($up to isomorphism$)$ irreducible representation $\rho_\psi$ of $U^{n,q}(\bF_{q^n})$ that has central character $\psi$ and occurs in \[H^\bullet_c(X,\ql):=\bigoplus_{i\in\bZ} H^i_c(X,\ql).\] Moreover, the multiplicity of $\rho_\psi$ in $H^\bullet_c(X,\ql)$ as a representation of $U^{n,q}(\bF_{q^n})$ is equal to $1$.
 \sbr
\item Let $\psi$ have conductor $q^m$, so that $n=mn_1$ for some $n_1\in\bN$. Then $\rho_\psi$ occurs in $H^{n+n_1-2}_c(X,\ql)$, and $\Fr_{q^n}$ acts on it via the scalar $(-1)^{n-n_1}\cdot q^{n(n+n_1-2)/2}$.
 \sbr
\item The representation $\rho_\psi$ can be constructed as follows. Write $\psi=\psi_1\circ\Tr_{\bF_{q^n}/\bF_{q_1}}$ for a unique character $\psi_1:\bF_{q_1}\rar{}\qls$, where $q_1=q^m$ as in Remark \ref{r:inclusion-new}. Put
    \[
    H_m = \Bigl\{ 1 + \sum_{\substack{j\leq n/2 \\ m\mid j}} a_j e_j + \sum_{n/2<j\leq n} a_j e_j \Bigr\} \subset U^{n,q},
    \]
    a connected subgroup. The projection $\nu_m:H_m\rar{}U^{n_1,q_1}$ obtained by discarding all summands $a_j e_j$ with $m\nmid j$ $($cf.~Remark \ref{r:inclusion-new}$)$ is a group homomorphism, and $\widetilde{\psi}:=\psi_1\circ\Nm^{n_1,q_1}\circ\nu_m$ is a character of $H_m(\bF_{q^n})$ that extends $\psi:Z(\bF_{q^n})\rar{}\qls$ $($see Remark \ref{r:reduced-norm-new}$)$. With this notation:
    \begin{itemize}
    \item if $m$ is odd \underline{or} $n_1$ is even, then $\rho_\psi\cong\Ind_{H_m(\bF_{q^n})}^{U^{n,q}(\bF_{q^n})}(\widetilde{\psi})$;
    \item if $m$ is even \underline{and} $n_1$ is odd, then $\Ind_{H_m(\bF_{q^n})}^{U^{n,q}(\bF_{q^n})}(\widetilde{\psi})$ is isomorphic to a direct sum of $q^{n/2}$ copies of $\rho_\psi$. Moreover, in this case, if $\Ga_m\subset U^{n,q}(\bF_{q^n})$ is the subgroup consisting of all elements of the form $h+a_{n/2}e_{n/2}$, where $h\in H_m(\bF_{q^n})$ and $a_{n/2}\in\bF_{q^{n/2}}$, then $\widetilde{\psi}$ can be extended to a character of $\Ga_m$, and if $\chi:\Ga_m\rar{}\qls$ is any such extension, then $\rho_\psi\cong\Ind_{\Ga_m}^{U^{n,q}(\bF_{q^n})}(\chi)$.
    \end{itemize}
\end{enumerate}
\end{thm}

\begin{rem}\label{r:connected-components-new}
By construction, $X$ is a finite \'etale cover of $Y\cong\bA^{n-1}$, so all connected components of $X\tens_{\fqn}\bfq$ are irreducible and smooth of dimension $n-1$. Theorem \ref{t:cohomology-of-X} implies that the top compactly supported cohomology $H^{2n-2}_c(X,\ql)$ has dimension $q$ (indeed, as a representation of $U^{n,q}(\fqn)$ it is the direct sum of $1$-dimensional representations of the form $\psi_1\circ\Nm^{n,q}$, where $\psi_1$ ranges over all characters $\bF_q\to\qls$), and $\Fr_{q^n}$ acts on it via the scalar $q^{n(n-1)/2}$. Hence $X$ has $q$ connected components, which are geometrically irreducible and smooth of dimension $n-1$.

\mbr

Let us give an explicit description of these components. In \S\ref{s:reduced-norm} below we introduce a morphism $N^{n,q}:U^{n,q}\rar{}\bG_a$, which extends the reduced norm map $\Nm^{n,q}$ in the sense that $N^{n,q}:U^{n,q}(\fqn)\rar{}\fqn$ has image in $\bF_q$ and is equal to $\Nm^{n,q}$. By Proposition \ref{p:reduced-norm-key},
$X$ can be described as the subvariety of $U^{n,q}\tens_{\bF_q}\fqn$ defined by the equation $N^{n,q}(g)^q=N^{n,q}(g)$. Hence the connected components of $X$ are precisely the subvarieties given by $N^{n,q}(g)=c$ as $c$ ranges over the points of $\bF_q\subset\bG_a$.
\end{rem}

\section{Properties of the reduced norm map}\label{s:reduced-norm}

\subsection{Summary}\label{ss:properties-reduced-norm-new} In this section we will extend the reduced norm map mentioned in Remark \ref{r:reduced-norm-new} to a morphism of $\fqn$-varieties $N^{n,q}:U^{n,q}\to\bG_a$ (it was denoted simply by $N$ in \S\ref{ss:reduced-norm-first-definition}) and establish some properties of $N^{n,q}$. The main result is

\begin{prop}\label{p:reduced-norm-key}
\begin{enumerate}[$($a$)$]
\item There is a unique morphism $N^{n,q}:U^{n,q}\to\bG_a$ of $\fqn$-schemes such that $N^{n,q}(1)=0$ and $\pr_n(L_{q^n}(g))=N^{n,q}(g)^q-N^{n,q}(g)$ for all $g\in U^{n,q}$, where $\pr_n:U^{n,q}\to\bG_a$ denotes the projection onto the last coordinate $e_n$.
 \sbr
\item If $g\in U^{n,q}$ and $h\in U^{n,q}(\fqn)$, then $N^{n,q}(gh)=N^{n,q}(g)+N^{n,q}(h)$.
 \sbr
\item We have $N^{n,q}(g)=\Nm^{n,q}(g)$ for all $g\in U^{n,q}(\fqn)$.
\end{enumerate}
\end{prop}

\begin{example}\label{ex:reduced-norm-n=2}
If $n=2$, the two cases considered in \S\ref{sss:Case1} and \S\ref{sss:Case2} become the same. For $g=1+a_1e_1+a_2e_2\in U^{2,q}$, we have $\Fr_{q^2}(g)=1+a_1^{q^2}e_1+a_2^{q^2}e_2$ and $g^{-1}=1-a_1e_1+(a_1^{1+q}-a_2)e_2$, so that $\pr_2(L_{q^2}(g)) = a_2^{q^2}+a_1^{1+q}-a_2-a_1^{q^2+q}$ and hence $N^{2,q}(g)=a_2+a_2^q-a_1^{1+q}$. For higher $n$, it is possible to give a formula for $N^{n,q}(g)$ as the determinant of a certain matrix whose entries are given explicitly in terms of the coefficients in the expansion $g=1+a_1e_1+\dotsc+a_ne_n$, but we do not find this formula to be useful and prefer to work exclusively in terms of the axiomatic characterization given in Proposition \ref{p:reduced-norm-key}.
\end{example}

\begin{cor}\label{c:reduced-norm-and-trace-new}
If $g\in U^{n,q}(\overline{\bF}_{q^n})$ is such that $\pr_n(L_{q^n}(g))\in \bF_{q^n}$, then
\[
\Tr_{\bF_{q^n}/\bF_q}\bigl(\pr_n(L_{q^n}(g))\bigr) = N^{n,q}(\Fr_{q^n}(g))-N^{n,q}(g).
\]
\end{cor}

\begin{proof}
By the proposition, $\pr_n(L_{q^n}(g))=N^{n,q}(g)^q-N^{n,q}(g)$, whence
\[
\Tr_{\bF_{q^n}/\bF_q}\bigl(\pr_n(L_{q^n}(g))\bigr) = \sum_{i=0}^{n-1} \pr_n(L_{q^n}(g))^{q^i} = N^{n,q}(g)^{q^n}-N^{n,q}(g).
\]
But $N^{n,q}(g)^{q^n}=\Fr_{q^n}(N^{n,q}(g))=N^{n,q}(\Fr_{q^n}(g))$, completing the proof.
\end{proof}

\subsection{Proof of Proposition \ref{p:reduced-norm-key}(a)} We begin by proving the uniqueness of $N^{n,q}$. Suppose that $N_1:U^{n,q}\to\bG_a$ is another morphism with the same properties as $N^{n,q}$. Then $\bigl(N^{n,q}(g)-N_1(g)\bigr)^q = N^{n,q}(g)-N_1(g)$ for all $g\in U^{n,q}$, which means that the image of $N^{n,q}-N_1:U^{n,q}\to\bG_a$ is contained in the discrete subset $\bF_q\subset\bG_a$. Since $U^{n,q}$ is connected, $N^{n,q}-N_1$ is constant. Since $N^{n,q}(1)=N_1(1)$, we have $N^{n,q}\equiv N_1$.

\mbr

To prove the existence of $N^{n,q}$ we use

\begin{lem}\label{l:normal-form}
Every element of $U^{n,q}$ can be written uniquely as
\[
1+a_1e_1+a_2e_2+\dotsb+a_ne_n = (1-b_1e_1)\cdot(1-b_2e_2)\cdot\dotsc\cdot(1-b_ne_n).
\]
The maps relating each of the $n$-tuples $(a_i)$ and $(b_j)$ to the other one are polynomial functions with coefficients in $\bF_p$.
\end{lem}

\begin{proof}
This is straightforward: first observe that $b_1$ must necessarily equal $-a_1$. Then multiply both sides of the identity above by $(1+a_1e_1)^{-1}$ on the left, and observe that the left hand side takes the form $1+a_2'e_2+\dotsb+a_n'e_n$, where the $a_j'$ are certain polynomial functions of the $a_i$. Proceed by induction.
\end{proof}

To complete the proof of Proposition \ref{p:reduced-norm-key}(a), we consider the following situation. Assume that we are given an element of $U^{n,q}$ of the form
\[
g=(1-b_k e_k)\cdot(1-b_{k+1}e_{k+1})\cdot\dotsc\cdot(1-b_ne_n),
\]
where $1\leq k\leq n$. We would like to show that there exists a polynomial map $F_k$ (depending only on $k$) such that $\pr_n(L_{q^n}(g))=F_k(b_k,\dotsc,b_n)^q-F_k(b_k,\dotsc,b_n)$.

\mbr

To this end, we use descending induction on $k$. When $k=n$, we have $g=1-b_ne_n$, so $L_{q^n}(g)=1+(b_n-b_n^{q^n})e_n$, and we can take $F_n(b_n)=-(1+b_n^q+b_n^{q^2}+\dotsb+b_n^{q^{n-1}})$.

\mbr

Now suppose that $1\leq k<n$ is arbitrary. We have
\[
L_{q^n}(g) = (1-b_k^{q^n}e_k)\cdot\dotsc\cdot(1-b_n^{q^n}e_n)\cdot (1-b_ne_n)^{-1}\cdot\dotsc\cdot(1-b_ke_k)^{-1},
\]
which can be rewritten as
\[
L_{q^n}(g) = (1-b_k^{q^n}e_k) \cdot \left( 1+\sum_{i=k+1}^n c_ie_i \right) \cdot \left( 1 + (b_ke_k)+(b_ke_k)^2+\dotsb \right).
\]
Here each $c_i$ is some polynomial function of the variables $b_{k+1},\dotsc,b_n$. Further, by induction, we may assume that
\[
c_n=F_{k+1}(b_{k+1},\dotsc,b_n)^q-F_{k+1}(b_{k+1},\dotsc,b_n)
\]
for some polynomial function $F_{k+1}$.

\mbr

Expanding out the product above and collecting only the terms that involve $e_n$, we obtain the following expression:
\begin{eqnarray*}
c_ne_n &+& \left[ (c_{n-k}e_{n-k})\cdot(b_ke_k) + (c_{n-2k}e_{n-2k})\cdot(b_ke_k)^2 + \dotsb \right] \\ &-& (b_k^{q^n}e_k)\cdot \left[ (c_{n-k}e_{n-k}) + (c_{n-2k}e_{n-2k})\cdot(b_ke_k) + \dotsb \right].
\end{eqnarray*}
Thanks to our induction assumption, the term $c_ne_n$ can be ignored for the purpose of the present proof. The remaining terms can be regrouped as follows:
\[
\sum_{i\geq 1} \left[ (c_{n-ik}e_{n-ik})\cdot(b_ke_k)^i - (b_k^{q^n}e_k)\cdot(c_{n-ik}e_{n-ik})\cdot(b_ke_k)^{i-1} \right].
\]
It remains to observe that if we are in the case of \S\ref{sss:Case2}, then the terms with $i\geq 2$ in the last sum are all $0$, and the sum becomes
\[
\bigl(c_{n-k} e_{n-k}\bigr) \cdot \bigl( b_k e_k \bigr) - \bigl( b_k^{q^n} e_k \bigr) \cdot \bigl( c_{n-k} e_{n-k} \bigr) = \Bigl[ \bigl(c_{n-k} b_k^{q^{n-k}}\bigr) - \bigl(c_{n-k} b_k^{q^{n-k}}\bigr)^{q^k} \Bigr] \cdot e_n.
\]
On the other hand, if we are in the case of \S\ref{sss:Case1}, then
\[
(c_{n-ik}e_{n-ik})\cdot(b_ke_k)^i - (b_k^{q^n}e_k)\cdot(c_{n-ik}e_{n-ik})\cdot(b_ke_k)^{i-1} = \]
\[
 = \left( c_{n-ik}\cdot b_k^{q^{n-ik}(1+q^k+\dotsb+q^{ki-k})} - c_{n-ik}^{q^k}\cdot b_k^{{q^n}+q^{n-ik+k}(1+q^k+\dotsb+q^{ki-2k})} \right) \cdot e_n
\]
\[
= \left( c_{n-ik}\cdot b_k^{q^{n-ik}+q^{n-ik+k}+\dotsb+q^{n-k}} - c_{n-ik}^{q^k}\cdot b_k^{q^{n-ik+k}+q^{n-ik+2k}+\dotsb+q^n} \right) \cdot e_n,
\]
and since
\[
c_{n-ik}\cdot b_k^{q^{n-ik}+q^{n-ik+k}+\dotsb+q^{n-k}} - c_{n-ik}^{q^k}\cdot b_k^{q^{n-ik+k}+q^{n-ik+2k}+\dotsb+q^n} = A-A^{q^k}
\]
where $A=c_{n-ik}\cdot b_k^{q^{n-ik}+q^{n-ik+k}+\dotsb+q^{n-k}}$,
the induction step is complete.

\mbr

Finally, define $N^{n,q}:U^{n,q}\rar{}\bG_a$ by the formula
\[
N^{n,q}\bigl((1-b_1e_1)\cdot(1-b_2e_2)\cdot\dotsc\cdot(1-b_ne_n)\bigr)
=F_1(b_1,\dotsc,b_n).
\]
It is clear that $N^{n,q}$ has the two properties stated in Proposition \ref{p:reduced-norm-key}(a).

\subsection{Proof of Proposition \ref{p:reduced-norm-key}(b)} Fix $h\in U^{n,q}(\fqn)$. The Lang map $L_{q^n}:U^{n,q}\to U^{n,q}$ has the property that $L_{q^n}(gh)=L_{q^n}(g)$ for all $g\in U^{n,q}$, whence
\[
N^{n,q}(gh)^q - N^{n,q}(gh) = N^{n,q}(g)^q-N^{n,q}(g)
\]
by the definition of $N^{n,q}$. By the same argument as in the proof of the uniqueness assertion of Proposition \ref{p:reduced-norm-key}(a), the morphism $U^{n,q}\to\bG_a$ given by $g\mapsto N^{n,q}(gh)-N^{n,q}(g)$ is constant. Its value at $g=1$ equals $N^{n,q}(h)$, proving Proposition \ref{p:reduced-norm-key}(b).

\subsection{Proof of Proposition \ref{p:reduced-norm-key}(c)}\label{ss:proof-proposition-key-c} We use an observation due to V.~Drinfeld. Recall that $U^{n,q}$ is a normal subgroup of the multiplicative group $\cR^\times$ of the ring scheme $\cR$ introduced in \S\ref{ss:definitions-U-n-q}. Now $\cR^\times$ also contains the multiplicative group $\bG_m$ as the subgroup defined by the equations $a_1=a_2=\dotsc=a_n=0$, so we obtain a conjugation action of $\bG_m$ on $U^{n,q}$: for $g\in U^{n,q}$, the action map is given by
\begin{equation}\label{e:conjugation-action-multiplicative-group}
\bG_m\ni\la : g = 1 + \sum_{j=1}^n a_j e_j \longmapsto \la g\la^{-1} = 1 + \sum_{j=1}^n \la^{1-q^j}a_j  e_j
\end{equation}

\begin{lem}\label{l:Drinfeld}
The reduced norm map $\Nm^{n,q}:U^{n,q}(\bF_{q^n})\rar{}\bF_q$ is the unique group homomorphism which is invariant under the action of $\bF_{q^n}^\times$ on $U^{n,q}(\bF_{q^n})$ coming from \eqref{e:conjugation-action-multiplicative-group} and restricts to $\Tr_{\bF_{q^n}/\bF_q}$ on the center $Z(\bF_{q^n})=\bF_{q^n}$ of $U^{n,q}(\bF_{q^n})$.
\end{lem}

\begin{proof}
The fact that $\Nm^{n,q}$ has all of the stated properties follows easily from its original definition (cf.~Remarks \ref{r:reduced-norm-new}, \ref{r:rings-are-obviously-equal} and \S\ref{ss:reduced-norm-first-definition}). To check the uniqueness claim, let $H\subset U^{n,q}(\bF_{q^n})$ be the subgroup generated by all elements of the form $g^{-1}\cdot (\la g\la^{-1})$ with $g\in U^{n,q}(\bF_{q^n})$ and $\la\in\bF_{q^n}^\times$. It suffices to show that $U^{n,q}(\bF_{q^n})=H\cdot Z(\bF_{q^n})$.

\mbr

Assume that this is not the case, and let $g=1+\sum_{j=k}^n a_j e_j \in U^{n,q}(\bF_{q^n})$ be an element that does not belong to $H\cdot Z(\bF_{q^n})$, where $k\geq 1$ is as large as possible. In particular, $k<n$. Hence there exists $\la\in\bF_{q^n}^\times$ with $\la^{1-q^k}\neq 1$. Put $b=\frac{a_k}{\la^{1-q^k}-1}$ and $g_1=1+b e_k$. Then $g_1^{-1}\cdot(\la g_1\la^{-1})=1+a_k e_k+O( e_{k+1})$, where $O(e_{k+1})$ denotes an unspecified expression of the form $\sum_{j\geq k+1}a'_j e_j$. Therefore $g=g_1^{-1}\cdot(\la g_1\la^{-1})\cdot g'$ for some $g'\in U^{n,q}(\bF_{q^n})$ such that $g'=1+O( e_{k+1})$. The maximality of $k$ implies that $g'\in H\cdot Z(\bF_{q^n})$, which is a contradiction.
\end{proof}

To see that Lemma \ref{l:Drinfeld} implies Proposition \ref{p:reduced-norm-key}(c), we argue as follows. As a special case of Proposition \ref{p:reduced-norm-key}(b), we see that $N^{n,q}:U^{n,q}(\bF_{q^n})\to\bF_q$ is a group homomorphism. Hence it suffices to check that $N^{n,q}$ is invariant under the action \eqref{e:conjugation-action-multiplicative-group} and that $N^{n,q}(1+a e_n)=\Tr_{\bF_{q^n}/\bF_q}(a)$ for all $a\in\bF_{q^n}$.

\mbr

Choose any $1\leq j\leq n$, pick $x\in\bF_{q^n}$, and consider $g=1-x e_j\in U^{n,q}(\bF_{q^n})$. We will check that $N^{n,q}(\la g\la^{-1})=N^{n,q}(g)$ for any $\la\in\bF_{q^n}^\times$, and, in addition, if $j=n$, then $N^{n,q}(g)=-\Tr_{\bF_{q^n}/\bF_q}(x)$. Since \eqref{e:conjugation-action-multiplicative-group} is a group action, and since $N^{n,q}:U^{n,q}(\bF_{q^n})\rar{}\bF_q$ is a homomorphism, it will follow from Lemma \ref{l:normal-form} that $N^{n,q}:U^{n,q}(\bF_{q^n})\rar{}\bF_q$ is invariant under the $\bF_{q^n}^\times$-action coming from \eqref{e:conjugation-action-multiplicative-group} and the proof will be complete. As before, we consider two cases.

\subsubsection{Case 1} Assume that we are in the case of \S\ref{sss:Case1} and consider the restriction of $N^{n,q}$ to the subvariety of $U^{n,q}$ consisting of all points of the form $1-b e_j$ (this subvariety is isomorphic to $\bA^1$). We calculate it explicitly as follows. We have
\[
L_{q^n}(1-b e_j) = (1-b^{q^n} e_j)\cdot\left[ 1 + (b e_j) + (b e_j)^2 + \dotsb \right],
\]
so if $j\nmid n$, we get $\pr_n\bigl(L_{q^n}(1-be_j)\bigr)=0$ and $N^{n,q}(1-b e_j)=0$, while if $j\mid n$, we get
\[
\pr_n(L_{q^n}(1-b e_j)) = b^{1+q^j+q^{2j}+\dotsb+q^{n-j}} - b^{q^j+q^{2j}+\dotsb+q^n},
\]
whence\footnote{Here we are using the fact that $\bA^1$ is connected to ensure that the expression we wrote down coincides with $N^{n,q}(1-b e_j)$ for all $b\in\bA^1$ (cf. the proof of uniqueness in Proposition \ref{p:reduced-norm-key}(a)).}
\[
N^{n,q}(1-b e_j)=-\vp(b)-\vp(b)^q-\vp(b)^{q^2}-\dotsb-\vp(b)^{q^{j-1}},
\]
where $\vp(b):=b^{1+q^j+\dotsb+q^{n-j}}$. In particular, if $j=n$, we obtain $\vp(x)=x$ and $N^{n,q}(1-x e_n)=-\Tr_{\bF_{q^n}/\bF_q}(x)$. In addition, if $j$ is arbitrary, then given $\la\in\bF_{q^n}^\times$, we have $\la(1-b e_j)\la^{-1}=1-\la^{1-q^j}b e_j$. So if $j\nmid n$, we get $N^{n,q}\bigl(\la(1-b e_j)\la^{-1}\bigr)=0=N^{n,q}(1-b e_j)$. If $j\mid n$, then with the notation above, $\vp(\la^{1-q^j}b)=\la^{1-q^n}\vp(b)=\vp(b)$, so we again have $N^{n,q}\bigl(\la(1-b e_j)\la^{-1}\bigr)=N^{n,q}(1-b e_j)$, completing the proof.

\subsubsection{Case 2} Now assume instead that we are in the case of \S\ref{sss:Case2}. Then we can repeat the same calculations as in the previous case, the sole difference being that the condition $j\mid n$ must be replaced with the following one: either $n=j$ or $n=2j$.

\section{Proof of Theorem B}\label{s:proof-Thm-B}

\subsection{Outline of the argument}\label{ss:outline-new} We first describe the strategy we will use to prove Theorem \ref{t:cohomology-of-X} (which is a stronger version of Theorem B from the introduction). We fix a character $\psi:\fqn\to\qls$ and let $q^m$ be its conductor. We also write $n_1=n/m$ and $q_1=q^m$, and let $\psi_1:\bF_{q_1}\to\qls$ be the character such that $\psi=\psi_1\circ\Tr_{\fqn/\bF_{q_1}}$.

\mbr

Recall that $\pr_n:U^{n,q}\to\bG_a$ denotes the projection onto the last factor:
\[
\pr_n(1 + a_1 e_1 + \dotsb + a_n e_n) = a_n.
\]
If $W\subset U^{n,q}$ is a subvariety, we also write $\pr_n$ for the restriction of $\pr_n$ to $W$.

\subsubsection*{Step 1} We first obtain some information about the irreducible representations of the group $U^{n,q}(\bF_{q^n})$. To this end, along with the closed connected subgroup
\[
H_m = \Bigl\{ 1 + \sum_{\substack{j\leq n/2 \\ m\mid j}} a_j e_j + \sum_{n/2<j\leq n} a_j e_j \Bigr\} \subset U^{n,q}
\]
defined in Theorem \ref{t:cohomology-of-X}(c), we introduce two more:
\[
H^+_m = \Bigl\{ 1 + \sum_{\substack{j<n/2 \\ m\mid j}} a_j  e_j + \sum_{n/2\leq j\leq n} a_j  e_j \Bigr\} \subset U^{n,q}
\]
and
\[
H^-_m = \Bigl\{ 1 + \sum_{\substack{n/2<j<n \\ m\nmid j}} a_j  e_j + a_n  e_n \Bigr\} \subset U^{n,q}.
\]
By construction, $H^-_m\subset H_m\subset H^+_m$. The subgroup $H^-_m$ will play a role in the other steps of the proof as well.

\begin{rems}
\begin{enumerate}[(1)]
\item We have $H^+_m=H_m$ unless $m$ is even and $n_1$ is odd, in which case $H_m$ is a normal subgroup of $H^+_m$ of codimension $1$.
 \sbr
\item If $m=n$, then $H^-_m=H_m$.
\end{enumerate}
\end{rems}

The following lemma is proved in \S\ref{ss:proof-l:representation-contains-new}.

\begin{lem}\label{l:representation-contains-new}
If $\rho$ is an irreducible representation of $U^{n,q}(\bF_{q^n})$ with central character $\psi$, the restriction of $\rho$ to $H^-_m(\bF_{q^n})$ contains the character $\psi\circ\pr_n:H^-_m(\bF_{q^n})\to\qls$.
\end{lem}

Now consider the character\footnote{The fact that the projection map $\nu_m:H_m\to U^{n_1,q_1}$ is a group homomorphism is verified by a direct calculation, and since the restriction of $\Nm^{n_1,q_1}:U^{n_1,q_1}(\bF_{q^n})\rar{}\bF_{q_1}$ to $Z(\bF_{q^n})$ is equal to $\Tr_{\bF_{q^n}/\bF_{q_1}}:\bF_{q^n}\rar{}\bF_{q_1}$, we see that $\widetilde{\psi}$ is indeed a character that extends $\psi:Z(\bF_{q^n})\rar{}\qls$.} $\widetilde{\psi}:=\psi_1\circ\Nm^{n_1,q_1}\circ\nu_m : H_m(\bF_{q^n})\to\qls$.

\begin{rem}\label{r:restriction-psi-tilde-new}
Recall that $\nu_m:H_m(\bF_{q^n})\rar{}U^{n_1,q_1}(\bF_{q^n})$ is the map that discards all summands $a_j e_j$ with $m\nmid j$ $($cf.~Remark \ref{r:inclusion-new}$)$. Hence $\widetilde{\psi}\bigl\lvert_{H^-_m(\bF_{q^n})}=\psi\circ\pr_n$.
\end{rem}

\begin{prop}\label{p:construction-rho-psi-new}
\begin{enumerate}[$($a$)$]
\item Suppose that $m$ is odd or $n_1$ is even. Then \[\rho_\psi:=\Ind_{H_m(\bF_{q^n})}^{U^{n,q}(\bF_{q^n})}(\widetilde{\psi})\] is an irreducible representation of $U^{n,q}(\bF_{q^n})$.
 \sbr
\item Suppose that $m$ is even and $n_1$ is odd\footnote{Equivalently, $n$ is even and $m$ does not divide $n/2$.}. Let $\Ga_m\subset U^{n,q}(\bF_{q^n})$ be the subgroup defined in Theorem \ref{t:cohomology-of-X}(c); in other words,
    \[
    \Ga_m = \Bigl\{ \ga=1+\sum_{j=1}^n a_j  e_j \Bigl\lvert \ga\in H^+_m(\bF_{q^n}) \text{ and } a_{n/2}\in\bF_{q^{n/2}} \Bigr\}.
    \]
    Then $\widetilde{\psi}$ can be extended to a character of $\Ga_m$, and if $\chi:\Ga_m\rar{}\qls$ is any such extension, then $\rho_\psi:=\Ind_{\Ga_m}^{U^{n,q}(\bF_{q^n})}(\chi)$ is an irreducible representation of $U^{n,q}(\bF_{q^n})$, which is independent of the choice of $\chi$. Furthermore, $\Ind_{H_m(\bF_{q^n})}^{U^{n,q}(\bF_{q^n})}(\widetilde{\psi})$ is isomorphic to a direct sum of $q^{n/2}$ copies of $\rho_\psi$.
\end{enumerate}
 \mbr
\noindent In both cases, the restriction of $\rho_\psi$ to $H^-_m(\bF_{q^n})$ contains $\psi\circ\pr_n$; in particular, $\rho_\psi$ has central character $\psi$.
\end{prop}

This proposition is proved in \S\ref{ss:proof-p:construction-rho-psi-new}.

\subsubsection*{Step 2} We consider $H^\bullet_c(X,\ql)=\bigoplus_{i\in\bZ} H^i_c(X,\ql)$ as a finite dimensional graded vector space over $\ql$ equipped with commuting actions of $U^{n,q}(\bF_{q^n})$ and $\Fr_{q^n}$. In particular, given any representation (not necessarily irreducible) $\xi$ of $U^{n,q}(\bF_{q^n})$, we obtain a graded vector space $\Hom_{U^{n,q}(\bF_{q^n})}\bigl(\xi,H^\bullet_c(X,\ql)\bigr)$ with an action of $\Fr_{q^n}$.

\mbr

Now consider the representation $\xi_\psi=\Ind_{H^-_m(\bF_{q^n})}^{U^{n,q}(\bF_{q^n})}(\psi\circ\pr_n)$. In view of Lemma \ref{l:representation-contains-new}, $\xi_\psi$ is isomorphic to a direct sum of all irreducible representations of $U^{n,q}(\bF_{q^n})$ that have central character $\psi$, taken with certain multiplicities.

\begin{prop}\label{p:how-xi-psi-appears-new}
$\Hom_{U^{n,q}(\bF_{q^n})}\bigl(\xi_\psi,H^\bullet_c(X,\ql)\bigr)$ is concentrated in degree $n+n_1-2$. It has dimension $1$ if $m$ is odd or $n_1$ is even, and it has dimension $q^{n/2}$ if $m$ is even and $n_1$ is odd. $\Fr_{q^n}$ acts on it via the scalar $(-1)^{n-n_1}\cdot q^{n(n+n_1-2)/2}$.
\end{prop}

This proposition is proved in \S\ref{ss:proof-p:how-xi-psi-appears-new}.

\subsubsection*{Step 3} The last ingredient is the following result, proved in \S\ref{ss:proof-p:rho-psi-appears-new}.
\begin{prop}\label{p:rho-psi-appears-new}
If $\rho_\psi$ is the representation of $U^{n,q}(\bF_{q^n})$ constructed in Proposition \ref{p:construction-rho-psi-new}, then $\Hom_{U^{n,q}(\bF_{q^n})}\bigl( \rho_\psi ,H^\bullet_c(X,\ql)\bigr)\neq 0$.
\end{prop}

\subsubsection*{The finale} Let us show that combining the three steps above, we obtain a proof of Theorem \ref{t:cohomology-of-X}. Write $\rho_\psi$ for the irreducible representation of $U^{n,q}(\bF_{q^n})$ constructed in Proposition \ref{p:construction-rho-psi-new}. Then $H^\bullet_c(X,\ql)$ contains $\rho_\psi$ as a direct summand by Proposition \ref{p:rho-psi-appears-new}. Introduce the following multiplicities:
\[
d_1 = \dim\Hom_{U^{n,q}(\bF_{q^n})}\bigl(\rho_\psi,H^\bullet_c(X,\ql)\bigr)\geq 1,
\]
\[
d_2 = \dim\Hom_{U^{n,q}(\bF_{q^n})}\bigl(\rho_\psi,\xi_\psi\bigr),
\]
\[
d_3 = \dim\Hom_{U^{n,q}(\bF_{q^n})}\bigl(\xi_\psi,H^\bullet_c(X,\ql)\bigr).
\]
Then $d_2$ is at least the multiplicity of $\rho_\psi$ in $\Ind_{H_m(\bF_{q^n})}^{U^{n,q}(\bF_{q^n})}(\widetilde{\psi})$. Furthermore, it is clear that $d_3\geq d_1\cdot d_2$, and equality holds if and only if $\rho_\psi$ is the \emph{unique} irreducible representation of $U^{n,q}(\bF_{q^n})$ that appears both in $H^\bullet_c(X,\ql)$ and in $\xi_\psi$.

\mbr

We now claim that $d_2\geq d_3$. Indeed, combining Proposition \ref{p:construction-rho-psi-new} with Proposition \ref{p:how-xi-psi-appears-new}, we see that \[d_3=\dim\Hom_{U^{n,q}(\bF_{q^n})}\Bigl(\rho_\psi,\Ind_{H_m(\bF_{q^n})}^{U^{n,q}(\bF_{q^n})}(\widetilde{\psi})\Bigr).\] The last assertion of Remark \ref{r:restriction-psi-tilde-new} implies that $\Ind_{H_m(\bF_{q^n})}^{U^{n,q}(\bF_{q^n})}(\widetilde{\psi})$ is a direct summand of $\xi_\psi$, whence $d_2\geq d_3$.

\mbr

Comparing the assertions of the last two paragraphs, we find that $d_1=1$ and $d_1\cdot d_2=d_2=d_3$. In view of the first and third assertions of Proposition \ref{p:how-xi-psi-appears-new}, we see that all parts of Theorem \ref{t:cohomology-of-X} follow.

\subsection{Proof of Lemma \ref{l:representation-contains-new}}\label{ss:proof-l:representation-contains-new} Fix an irreducible representation $\rho$ of $U^{n,q}(\bF_{q^n})$ with central character $\psi$. We must prove that the restriction of $\rho$ to $H^-_m(\bF_{q^n})$ contains the $1$-dimensional representation $\psi\circ\pr_n:H^-_m(\bF_{q^n})\rar{}\qls$. For each integer $k$ we write $U^{\geq k}\subset U^{n,q}$ for the subgroup consisting of elements of the form $1+\sum_{j=k}^n a_j e_j$. We will show, using descending induction on $k$, that for any $k$ the restriction of $\rho$ to $H^-_m(\bF_{q^n})\cap U^{\geq k}(\bF_{q^n})$ contains the character $\psi\circ\pr_n$. This will imply the lemma.

\mbr

If $k=n$, there is nothing to prove. So we assume that $k<n$ and that the assertion in the previous paragraph holds for $k+1$ in place of $k$. Further, we may assume that $k>n/2$ and $m\nmid k$, since otherwise $H^-_m(\bF_{q^n})\cap U^{\geq k}(\bF_{q^n})=H^-_m(\bF_{q^n})\cap U^{\geq k+1}(\bF_{q^n})$.

\mbr

By the induction hypothesis, the restriction of $\rho$ to $H^-_m(\bF_{q^n})\cap U^{\geq k+1}(\bF_{q^n})$ contains $\psi\circ\pr_n$. This implies that the restriction of $\rho$ to $H^-_m(\bF_{q^n})\cap U^{\geq k}(\bF_{q^n})$ contains some character $\chi:H^-_m(\bF_{q^n})\cap U^{\geq k}(\bF_{q^n})\rar{}\qls$ such that
\[
\chi\bigl\lvert_{H^-_m(\bF_{q^n})\cap U^{\geq k+1}(\bF_{q^n})}=\psi\circ\pr_n.
\]
The subgroup $U^{\geq n-k}(\bF_{q^n})\subset U^{n,q}(\bF_{q^n})$ normalizes $H^-_m(\bF_{q^n})\cap U^{\geq k}(\bF_{q^n})$ and centralizes $H^-_m(\bF_{q^n})\cap U^{\geq k+1}(\bF_{q^n})$. It will be enough to find an element $g\in U^{\geq n-k}(\bF_{q^n})$ that conjugates $\chi$ into the character $\psi\circ\pr_n$ on $H^-_m(\bF_{q^n})\cap U^{\geq k}(\bF_{q^n})$.

\mbr

To this end, observe that by construction, we can write
\[
\chi \Bigl( 1 + \sum_{\substack{k\leq j<n \\ m\nmid j}} a_j  e_j + a_n  e_n \Bigr) = \chi_1(a_k)\cdot\psi(a_n)
\]
for some character $\chi_1:\bF_{q^n}\to\qls$. Let $g=1+a_{n-k} e_{n-k}\in U^{\geq n-k}(\bF_{q^n})$, where $a_{n-k}\in\bF_{q^n}$ will be chosen later. Then a direct calculation shows that
\[
\chi \biggl( g \cdot \Bigl( 1 + \sum_{\substack{k\leq j<n \\ m\nmid j}} a_j  e_j + a_n  e_n \Bigr) \cdot g^{-1} \biggr) = \chi_1(a_k)\cdot\psi\bigl(a_n+a_{n-k} a_k^{q^{n-k}} - a_k a_{n-k}^{q^k}\bigr).
\]
It remains to check that $a_{n-k}\in\bF_{q^n}$ can be chosen so that
\begin{equation}\label{e:need}
\chi_1(x)=\psi\bigl(a_{n-k}^{q^k}x-a_{n-k}x^{q^{n-k}}\bigr)
\end{equation}
for all $x\in\bF_{q^n}$. If we fix a nontrivial character $\psi_0:\bF_p\to\qls$, we can find (unique) $y,b\in\bF_{q^n}$ such that $\chi_1(x)=\psi_0\bigl(\Tr_{\bF_{q^n}/\bF_p}(yx)\bigr)$ and $\psi(x)=\psi_0\bigl(\Tr_{\bF_{q^n}/\bF_p}(bx)\bigr)$ for all $x\in\bF_{q^n}$. Then
\begin{eqnarray*}
\psi\bigl(a_{n-k}^{q^k}x-a_{n-k}x^{q^{n-k}}\bigr) &=& \psi_0\Bigl(\Tr_{\bF_{q^n}/\bF_p}\bigl( b a_{n-k}^{q^k}x - b a_{n-k}x^{q^{n-k}} \bigr)\Bigr) \\
&=& \psi_0\Bigl(\Tr_{\bF_{q^n}/\bF_p}\bigl( a_{n-k}^{q^{-(n-k)}}\cdot x \cdot (b-b^{q^k}) \bigr)\Bigr),
\end{eqnarray*}
where we used the identities $a_{n-k}^{q^k}=a_{n-k}^{q^{-(n-k)}}$ and $b^{q^k}=b^{q^{-(n-k)}}$ (which hold because $a_{n-k},b\in\bF_{q^n}$) together with the fact that $\Tr_{\bF_{q^n}/\bF_p}(z^{q^{n-k}})=\Tr_{\bF_{q^n}/\bF_p}(z)$ for all $z\in\bF_{q^n}$. Since $m\nmid k$ and $\psi$ has conductor $q^m$ by assumption, we have $b^{q^k}\neq b$. So if we choose $a_{n-k}=\bigl( y / (b-b^{q^k}) \bigr)^{q^{n-k}}$, then \eqref{e:need} is satisfied. \qed

\subsection{Auxiliary lemmas on finite group representations} The next two lemmas (which are rather standard) will be used in the proof of Proposition \ref{p:construction-rho-psi-new}.

\begin{lem}\label{l:irreducibility-induced}
Let $\Ga$ be a finite group, $N\subset\Ga$ a normal subgroup and $\chi:N\to\qls$ a character. Write $\Ga^\chi\subset\Ga$ for the stabilizer of $\chi$ with respect to the conjugation action of $\Ga$. If $\rho$ is any irreducible representation of $\Ga^\chi$ whose restriction to $N$ is isomorphic to a direct sum of copies of $\chi$, then $\Ind_{\Ga^\chi}^\Ga\rho$ is irreducible.
\end{lem}

The proof is an easy exercise in applying Mackey's irreducibility criterion.

\begin{lem}[See Prop.~B.4 in \cite{intro} and its proof]\label{l:heisenberg-representations}
Let $H$ be a finite group, $N\subset H$ a normal subgroup and $\chi:N\to\qls$ a character that is invariant under $H$-conjugation. Assume also that $N$ contains the commutator subgroup $[H,H]$.

\begin{enumerate}[$($a$)$]
\item The map $H\times H\rar{}\qls$ given by $(h_1,h_2)\mapsto\chi(h_1h_2h_1^{-1}h_2^{-1})$ descends to a bimultiplicative map $B_\chi:(H/N)\times(H/N)\rar{}\qls$.
 \sbr
\item Let $K=\bigl\{x\in H/N \st B_\chi(x,y)=1\ \forall\,y\in H/N \bigr\}$ denote the kernel of $B_\chi$ and write $K'\subset H$ for the preimage of $K$ in $H$. Then $\chi$ extends to a character of $K'$, and given any such extension $\chi':K'\rar{}\qls$, the induced representation $\Ind_{K'}^H\chi'$ is a direct sum of copies of an irreducible representation $\rho_{\chi'}$ of $H$.
 \sbr
\item Let $L\subset H/N$ be maximal among all subgroups of $H/N$ with the property that $B_\chi\bigl\lvert_{L\times L}\equiv 1$, write $L'\subset H$ for the preimage of $L$ in $H$ (so that $K'\subset L'$), and let $\chi':K'\rar{}\qls$ be as in part (b). Then $\chi'$ extends to a character of $L'$, and given any such extension $\widetilde{\chi}':L'\rar{}\qls$, we have $\rho_{\chi'}\cong\Ind_{L'}^H\widetilde{\chi}'$.
\end{enumerate}
\end{lem}


\subsection{Proof of Proposition \ref{p:construction-rho-psi-new}}\label{ss:proof-p:construction-rho-psi-new} Write $U^{>n/2}\subset U^{n,q}$ for the subgroup consisting of elements of the form $1+\sum_{n/2<j\leq n}a_j  e_j$. This is a normal abelian subgroup of $U^{n,q}$, which is contained in $H_m$. We first establish the following

\begin{lem}\label{l:stabilizer-new}
The normalizer in $U^{n,q}(\bF_{q^n})$ of the character
\[
\widetilde{\psi}\bigl\lvert_{U^{>n/2}(\bF_{q^n})} : U^{>n/2}(\bF_{q^n}) \rar{} \qls
\]
is equal to $H^+_m(\bF_{q^n})$.
\end{lem}

\begin{proof}
First let us check that $H^+_m(\bF_{q^n})$ does normalize $\widetilde{\psi}\bigl\lvert_{U^{>n/2}}(\bF_{q^n})$. If $m$ is odd or $n_1$ is even, then $H^+_m=H_m$, so there is nothing to do. Suppose that $m$ is even and $n_1$ is odd. It is enough to show that any element of the form $g=1+a_{n/2} e_{n/2}\in U^{n,q}(\bF_{q^n})$ normalizes $\widetilde{\psi}\bigl\lvert_{U^{>n/2}}(\bF_{q^n})$. But in fact, $g$ centralizes $U^{>n/2}(\bF_{q^n})$.

\mbr

Now, to obtain a contradiction, assume that there exists an element $g\in U^{n,q}(\bF_{q^n})$ such that $g\not\in H^+_m(\bF_{q^n})$ and $g$ normalizes $\widetilde{\psi}\bigl\lvert_{U^{>n/2}}(\bF_{q^n})$. Write $g=1+\sum_{j=1}^n a_j  e_j$ and let $k$ be the smallest integer such that $a_k\neq 0$ and $m\nmid k$. Then $k<n/2$. Multiplying $g$ by a suitable element of $U^{n_1,q_1}(\bF_{q^n})$ on the right, we may assume that $a_j=0$ for all $1\leq j<k$. Next consider the subgroup of $U^{>n/2}(\bF_{q^n})$ consisting of all elements of the form $1+c_{n-k} e_{n-k}+c_n e_n$. By assumption, we have
\begin{eqnarray*}
\psi(c_n) &=& \widetilde{\psi}(1+c_{n-k} e_{n-k}+c_n e_n) \\
&=& \widetilde{\psi}\bigl(g\cdot(1+c_{n-k} e_{n-k}+c_n e_n)\cdot g^{-1}\bigr) \\
&=& \widetilde{\psi}\bigl( 1 + c_{n-k}  e_{n-k} + (c_n+a_k c_{n-k}^{q^k} - c_{n-k} a_k^{q^{n-k}})  e_n \bigr) \\
&=& \psi \bigl( c_n+a_k c_{n-k}^{q^k} - c_{n-k} a_k^{q^{n-k}} \bigr) = \psi(a_n)\cdot \psi \bigl( a_k c_{n-k}^{q^k} - c_{n-k} a_k^{q^{n-k}} \bigr).
\end{eqnarray*}
We see that $\psi \bigl( a_k c_{n-k}^{q^k} - c_{n-k} a_k^{q^{n-k}} \bigr)=1$ for all $c_{n-k}\in\bF_{q^n}$. We claim that this is a contradiction. Indeed, as in \S\ref{ss:proof-l:representation-contains-new}, choose $b\in\bF_{q^n}$ with $\psi(x)=\psi_0\bigl(\Tr_{\bF_{q^n}/\bF_p}(bx)\bigr)$ for all $x\in\bF_{q^n}$. Then the computation from \S\ref{ss:proof-l:representation-contains-new} shows that
\[
\psi \bigl( a_k c_{n-k}^{q^k} - c_{n-k} a_k^{q^{n-k}} \bigr) = \psi_0 \Bigl(\Tr_{\bF_{q^n}/\bF_p}\bigl( c_{n-k}^{q^{-(n-k)}}\cdot a_k \cdot (b-b^{q^k}) \bigr)\Bigr),
\]
and since $b^{q^k}\neq b$ and $a_k\neq 0$, the right hand side cannot be $1$ for all $c_{n-k}\in\bF_{q^n}$.
\end{proof}

We proceed with the proof of Proposition \ref{p:construction-rho-psi-new}. If $m$ is odd or $n_1$ is even, then by Lemma \ref{l:stabilizer-new}, the normalizer in $U^{n,q}(\bF_{q^n})$ of $\widetilde{\psi}\bigl\lvert_{U^{>n/2}(\bF_{q^n})}$ is equal to $H_m(\bF_{q^n})$, so applying Lemma \ref{l:irreducibility-induced} to $\Ga=U^{n,q}(\bF_{q^n})$, $N=U^{>n/2}(\bF_{q^n})$ and $\chi=\widetilde{\psi}\bigl\lvert_{U^{>n/2}(\bF_{q^n})}$ implies that $\Ind_{H_m(\bF_{q^n})}^{U^{n,q}(\bF_{q^n})}(\widetilde{\psi})$ is irreducible, proving part (a) of the proposition.

\mbr

Next assume that $m$ is even and $n_1$ is odd. Let us apply Lemma \ref{l:heisenberg-representations} to the group $H=H^+_m(\bF_{q^n})$, the normal subgroup $N=H_m(\bF_{q^n})$ of $H$ and the character $\chi=\widetilde{\psi}$. By Lemma \ref{l:stabilizer-new}, $\widetilde{\psi}$ is invariant under $H^+_m(\bF_{q^n})$-conjugation. The quotient $H^+_m(\bF_{q^n})/H_m(\bF_{q^n})$ can be naturally identified with the additive group of $\bF_{q^n}$. To calculate the induced ``commutator pairing''
\[
B_{\widetilde{\psi}} : \bigl(H^+_m(\bF_{q^n})/H_m(\bF_{q^n})\bigr) \times \bigl(H^+_m(\bF_{q^n})/H_m(\bF_{q^n})\bigr) \rar{} \qls,
\]
we observe that if $g=1+x e_{n/2}$ and $h=1+y e_{n/2}$ with $x,y\in\bF_{q^n}$, then
\[
ghg^{-1}h^{-1} = 1 + \bigl( x\cdot y^{q^{n/2}} - y\cdot x^{q^{n/2}} \bigr) \cdot  e_n,
\]
whence $B_{\widetilde{\psi}}$ can be identified with the pairing
\begin{equation}\label{e:pairing-1}
\bF_{q^n} \times \bF_{q^n} \rar{} \qls, \qquad (x,y) \longmapsto \psi\bigl( x\cdot y^{q^{n/2}} - y\cdot x^{q^{n/2}} \bigr).
\end{equation}

\begin{lem}\label{l:lagrangian-1}
The pairing \eqref{e:pairing-1} is nondegenerate, and the additive subgroup $\bF_{q^{n/2}}\subset\bF_{q^n}$ is maximal isotropic with respect to it.
\end{lem}

\begin{proof}
It is clear that $\bF_{q^{n/2}}$ is isotropic with respect to \eqref{e:pairing-1}. If we show that \eqref{e:pairing-1} is nondegenerate, then the maximality will follow from the fact that $\#{\bF_{q^{n/2}}}=\sqrt{\#{\bF_{q^n}}}$. To this end, as in \S\ref{ss:proof-l:representation-contains-new}, choose $b\in\bF_{q^n}$ such that $\psi(x)=\psi_0\bigl(\Tr_{\bF_{q^n}/\bF_p}(bx)\bigr)$ for all $x\in\bF_{q^n}$. Assume that $y\in\bF_{q^n}$ is such that $\psi\bigl( x\cdot y^{q^{n/2}} - y\cdot x^{q^{n/2}} \bigr)=1$ for all $x\in\bF_{q^n}$. Then $\psi_0\Bigl( \Tr_{\bF_{q^n}/\bF_p} \bigl( b\cdot x\cdot y^{q^{n/2}} - b^{q^{n/2}}\cdot y^{q^{n/2}}\cdot x \bigr) \Bigr)=1$ for all $x\in\bF_{q^n}$, where we used the identities $b^{q^{-n/2}}=b^{q^{n/2}}$, $y^{q^{-n/2}}=y^{q^{n/2}}$ and the fact that $\Tr_{\bF_{q^n}/\bF_p}(z^{q^{n/2}})=\Tr_{\bF_{q^n}/\bF_p}(z)$ for all $z\in\bF_{q^n}$. This forces $(b-b^{q^{n/2}})\cdot y^{q^{n/2}}=0$. Since $\psi$ has conductor $q^m$ and $m$ does not divide $n/2$ by assumption, we have $b\neq b^{q^{n/2}}$, whence $y=0$, as needed.
\end{proof}

Now we complete the proof of Proposition \ref{p:construction-rho-psi-new}(b). Note that the subgroup $\Ga_m$ equals the preimage of $\bF_{q^{n/2}}\subset\bF_{q^n}=H^+_m(\bF_{q^n})/H_m(\bF_{q^n})$ in $H^+_m(\bF_{q^n})$. By Lemmas \ref{l:heisenberg-representations} and \ref{l:lagrangian-1}, $\Ind_{H_m(\bF_{q^n})}^{H^+_m(\bF_{q^n})}(\widetilde{\psi})$ is a direct sum of $q^{n/2}$ copies of a single irreducible representation $\rho$ of $H^+_m(\bF_{q^n})$. Moreover, $\widetilde{\psi}$ can be extended to a character of $\Ga_m$, and if $\chi$ is any such extension, then $\rho\cong\Ind_{\Ga_m}^{H^+_m(\bF_{q^n})}(\chi)$. We see that
\[
\rho_\psi := \Ind_{\Ga_m}^{U^{n,q}(\bF_{q^n})}(\chi) \cong \Ind_{H^+_m(\bF_{q^n})}^{U^{n,q}(\bF_{q^n})}(\rho)
\]
is independent of the choice of $\chi$, and
\[
\Ind_{H_m(\bF_{q^n})}^{U^{n,q}(\bF_{q^n})}(\widetilde{\psi})\cong \Ind_{H^+_m(\bF_{q^n})}^{U^{n,q}(\bF_{q^n})} \Ind_{H_m(\bF_{q^n})}^{H^+_m(\bF_{q^n})}(\widetilde{\psi})
\]
is isomorphic to a direct sum of $q^{n/2}$ copies of $\rho_\psi$, completing the proof.

\subsection{Proof of Proposition \ref{p:how-xi-psi-appears-new}}\label{ss:proof-p:how-xi-psi-appears-new} As explained in Remark \ref{r:inclusion-new}, we identify $U^{n_1,q_1}$ with the subgroup of $U^{n,q}$ consisting of all elements of the form $1+\sum_{m\mid j}a_j e_j$.

\subsubsection{Auxiliary notation} Let $I'$ denote the set of integers $j$ such that $n/2<j<n$ and $m\nmid j$. Put $I=I'\cup\{n\}$ and $J=\{1,2,\dotsc,n\}\setminus I$. Then we can write
\[
H^-_m = \Bigl\{ 1 + \sum_{i\in I} a_i  e_i \Bigr\} \subset U^{n,q},
\]
and we can identify $U^{n,q}/H^-_m$ with an affine space $\bA^d$ of dimension
\[
d:=\#{J} = \begin{cases}
\frac{n+n_1}{2}-1 & \text{if } m \text{ is odd or } n_1 \text{ is even}, \\
\frac{n+n_1+1}{2}-1 & \text{if } m \text{ is even and } n_1 \text{ is odd}.
\end{cases}
\]
We will denote the coordinates of this affine space by $(a_j)_{j\in J}$.

\subsubsection{A reformulation of Proposition \ref{p:how-xi-psi-appears-new}}\label{sss:reformulation-1} The morphism
\[
s : \bA^d \rar{} U^{n,q}, \qquad (a_j)_{j\in J} \longmapsto 1+\sum_{j\in J} a_j e_j
\]
is a section of the quotient map $U^{n,q}\rar{}U^{n,q}/H^-_m$. Define
\[
F:\bA^d\times H^-_m\rar{}U^{n,q} \quad\text{via}\quad (x,h)\mapsto \Fr_{q^n}(s(x))h s(x)^{-1},
\]
write $\cL_\psi$ for the Artin-Scherier local system on $\bG_a$ defined by the character $\psi$, and let $\widetilde{\pr}_n:\bA^d\times H^-_m\rar{}\bG_a$ be the composition of the second projection $\bA^d\times H^-_m\rar{}H^-_m$ with $\pr_n:H^-_m\rar{}\bG_a$. By \cite[Prop.~2.3]{DLtheory}, we have
\begin{equation}\label{e:2-1}
\Hom_{U^{n,q}(\bF_{q^n})}\bigl(\xi_\psi,H^\bullet_c(X,\ql)\bigr) \cong H^\bullet_c \bigl( F^{-1}(Y), \widetilde{\pr}_n^*\cL_\psi\bigl\lvert_{F^{-1}(Y)} \bigr)
\end{equation}
as graded vector spaces with an action of $\Fr_{q^n}$.

\bigbreak

\bigbreak

Now note that if $h=1+\sum_{i\in I}a_i e_i\in H^-_m$ and we put $h^\circ:=1+\sum_{i\in I\setminus\{n\}}a_i e_i$, then for any $g_1,g_2\in U^{n,q}$, we have $\pr_n(g_1hg_2)=a_n+\pr_n(g_1h^\circ g_2)$. In particular, for any $x\in\bA^d$, we have $\pr_n\bigl(F(x,h)\bigr)=a_n+\pr_n\bigl(F(x,h^\circ)\bigr)$. This implies that the map
\[
\bA^d\times H^-_m \rar{} \bA^d\times (H^-_m\cap Y), \qquad (x,h)\longmapsto(x,h^\circ)
\]
(which is the projection onto the first $n-1$ coordinates) yields an isomorphism between $F^{-1}(Y)\subset\bA^d\times H^-_m$ and $\bA^d\times (H^-_m\cap Y)$. Under this isomorphism, the local system $\widetilde{\pr}_n^*\cL_\psi\bigl\lvert_{F^{-1}(Y)}$ corresponds to the local system $\al^*(\cL_\psi)$ on $\bA^d\times(H^-_m\cap Y)$, where $\al:\bA^d\times(H^-_m\cap Y)\rar{}\bG_a$ is given by $\al(x,h)=-\pr_n\bigl(\Fr_{q^n}(s(x))h s(x)^{-1}\bigr)$. So in view of \eqref{e:2-1}, Proposition \ref{p:how-xi-psi-appears-new} is equivalent to the following assertions:
\begin{equation}\label{e:assertion1-1}
H^r_c \bigl( \bA^d\times(H^-_m\cap Y), \al^*(\cL_\psi) \bigr) = 0 \qquad\text{if } r\neq n+n_1-2;
\end{equation}
\begin{equation}\label{e:assertion2-1}
\begin{split}
\dim H^{n+n_1-2}_c \bigl( \bA^d\times(H^-_m\cap Y), \al^*(\cL_\psi) \bigr) = \\ =
\begin{cases}
1 & \text{if } m \text{ is odd or } n_1 \text{ is even}, \\
q^{n/2} & \text{if } m \text{ is even and } n_1 \text{ is odd};
\end{cases}
\end{split}
\end{equation}
\begin{equation}\label{e:assertion3-1}
\begin{split}
\Fr_{q^n} \text{ acts on } \dim H^{n+n_1-2}_c \bigl( \bA^d\times(H^-_m\cap Y), \al^*(\cL_\psi) \bigr) \\ \text{ as multiplication by the scalar } (-1)^{n-n_1}\cdot q^{n(n+n_1-2)/2}.
\end{split}
\end{equation}

\subsubsection{Additional notation} There are $n-d-1$ integers $j$ such that $n>j>n/2$ and $m\nmid j$. Let us label them as follows: $j_1>j_2>\dotsb>j_{n-d-1}$. Thus $I=\{n\}\cup I^\circ$, where $I^\circ=\{j_1,j_2,\dotsc,j_{n-d-1}\}$, and $J=\{m,2m,\dotsc,(n_1-1)m\}\cup J^\circ$, where
\[
J^\circ = \begin{cases}
\{n-j_1,\dotsc,n-j_{n-d-1}\} & \text{if } m \text{ is odd or } n_1 \text{ is even}, \\
\{n-j_1,\dotsc,n-j_{n-d-1},n/2\} & \text{if } m \text{ is even and } n_1 \text{ is odd}.
\end{cases}
\]
From now on we will identify $\bA^d\times(H^-_m\cap Y)$ with the affine space $\bA^{n-1}$, whose coordinates will be denoted by $(a_j)_{j\in J\cup I^\circ}$.

\subsubsection{An inductive setup}\label{sss:inductive-setup-1} For each $1\leq k\leq n-d$, put $I_k^\circ=\{j_k,j_{k+1},\dotsc,j_{n-d-1}\}$ and
$J_k=\{m,2m,\dotsc,(n_1-1)m\}\cup J_k^\circ$, where
\[
J^\circ_k = \begin{cases}
\{n-j_k,\dotsc,n-j_{n-d-1}\} & \text{if } m \text{ is odd or } n_1 \text{ is even}, \\
\{n-j_k,\dotsc,n-j_{n-d-1},n/2\} & \text{if } m \text{ is even and } n_1 \text{ is odd}.
\end{cases}
\]
In particular, $I_1^\circ=I^\circ$, $J^\circ_1=J^\circ$, $I^\circ_{n-d}=\varnothing$, $J^\circ_{n-d}=\varnothing$ if $m$ is odd or $n_1$ is even, and $J^\circ_{n-d}=\{n/2\}$ if $m$ is even and $n_1$ is odd. Observe also that for each $1\leq k\leq n-d-1$, the set $I^\circ_{k+1}$ (respectively, $J_{k+1}$) is obtained from $I^\circ_k$ (respectively, $J_k$) by removing $j_k$ (respectively, $n-j_k$).

\mbr

We write $\bA^{n-2k+1}$ for the $(n-2k+1)$-dimensional affine space, whose coordinates will be denoted by $(a_j)_{j\in J_k\cup I_k^\circ}$. If $1\leq k\leq n-d-1$, we write $p_k:\bA^{n-2k+1}\to\bA^{n-2k-1}$ for the projection obtained by discarding $a_{j_k}$ and $a_{n-j_k}$, and $\iota_k:\bA^{n-2k-1}\into\bA^{n-2k+1}$ for the natural ``zero section'' of $p_k$. We put $\al_k=\al\circ\iota_1\circ\dotsc\circ\iota_{k-1}:\bA^{n-2k+1}\to\bG_a$, where $\al:\bA^{n-1}=\bA^d\times(H^-_m\cap Y)\rar{}\bG_a$ is the morphism introduced in \S\ref{sss:reformulation-1}. In particular, we have $\al_1=\al$.

\subsubsection{The key lemma} The next result allows us to exploit the inductive setup formulated in \S\ref{sss:inductive-setup-1}. Its proof is based on \cite[Prop.~2.10]{DLtheory}.

\begin{lem}\label{l:key-induction-new}
For each $1\leq k\leq n-d-1$, we have
\[
H^\bullet_c(\bA^{n-2k+1},\al_k^*\cL_\psi) \cong H^\bullet_c(\bA^{n-2k-1},\al_{k+1}^*\cL_\psi)[-2](-1)
\]
as graded vector spaces with an action of $\Fr_{q^n}$.
\end{lem}

\begin{proof}
Recall that
\[
\al_k \bigl( (a_j)_{j\in J_k\cup I^\circ_k} \bigr)
=
-\pr_n \biggl(
\Bigl( 1+\sum_{j\in J_k}a_j^{q^n} e_j \Bigr) \cdot \Bigl( 1+\sum_{i\in I_k^\circ} a_i  e_i \Bigr) \cdot \Bigl( 1+\sum_{j\in J_k}a_j e_j \Bigr)^{-1}
\biggr).
\]
The right hand side can be written as a certain sum of monomials in the variables $(a_j)_{j\in J_k\cup I^\circ_k}$. By our choice of the ordering $j_1>j_2>\dotsb>j_{n-d-1}>n/2$, only two of these monomials involve the variable $a_{j_k}$, namely, $-a_{n-j_k}^{q^n}\cdot a_{j_k}^{q^{n-j_k}}$ and $a_{j_k}\cdot a_{n-j_k}^{q^{j_k}}$. This implies that we can write
\[
\begin{split}
\al_k \bigl( (a_j)_{j\in J_k\cup I^\circ_k} \bigr) = & a_{j_k}\cdot a_{n-j_k}^{q^{j_k}} - a_{j_k}^{q^{n-j_k}}\cdot a_{n-j_k}^{q^n} \\
& + \widetilde{\al}_{k+1} \bigl( (a_j)_{j\in J_{k+1}\cup I^\circ_{k+1}} \bigr) \\
& + a_{n-j_k} \cdot \be_k\bigl( (a_j)_{j\in J_k\cup I^\circ_{k+1}} \bigr)
\end{split}
\]
for some polynomials $\widetilde{\al}_{k+1}:\bA^{n-2k-1}\to\bG_a$ and $\be_k:\bA^{n-2k}\to\bG_a$. Substituting $a_{j_k}=a_{n-j_k}=0$ into the last identity shows that $\widetilde{\al}_{k+1}=\al_{k+1}$.

\mbr

We now apply \cite[Prop.~2.10]{DLtheory} in the following setting. $\psi$, $q$ and $n$ have the same meaning as in \emph{loc.~cit.} We let $S_2$ be the affine space $\bA^{n-2k}$ with coordinates $(a_j)_{j\in J_k\cup I^\circ_{k+1}}$ and identify $S=S_2\times\bA^1$ with the affine space $\bA^{n-2k+1}$, where the additional coordinate is labeled $a_{j_k}$ and corresponds to the coordinate $y$ in \emph{loc.~cit.} The morphism $f:S_2\to\bG_a$ from \emph{loc.~cit.} is the projection onto the coordinate $a_{n-j_k}$, so that the subscheme $S_3\subset S_2$ introduced in \emph{loc.~cit.} is identified with the affine space $\bA^{n-2k-1}$ with coordinates $(a_j)_{j\in J_{k+1}\cup I^\circ_{k+1}}$.
Since $m\nmid j_k$, the assertion of Lemma \ref{l:key-induction-new} follows at once from \cite[Prop.~2.10]{DLtheory}.
\end{proof}

\subsubsection{Conclusion of the proof of Proposition \ref{p:how-xi-psi-appears-new}} We return to the task of proving \eqref{e:assertion1-1}--\eqref{e:assertion3-1}. Applying Lemma \ref{l:key-induction-new} successively for $k=1,2,\dotsc,n-d-1$, we obtain
\begin{equation}\label{e:6-1}
\begin{split}
H^\bullet_c \bigl( \bA^d\times(H^-_m\cap Y), \al^*\cL_\psi \bigr) = H^\bullet_c \bigl( \bA^{n-1}, \al^*\cL_\psi \bigr) \\
\cong H^\bullet_c \bigl( \bA^{n-2(n-d)+1}, \al_{n-d}^*\cL_\psi \bigr)[-2(n-d-1)](-(n-d-1)).
\end{split}
\end{equation}
Let us now calculate $\al_{n-d}^*(\cL_\psi)$. Recall that the coordinates on the affine space $\bA^{n-2(n-d)+1}$ are labeled $(a_j)_{j\in J_{n-d}}$, where
\[
J_{n-d} = \begin{cases}
\{m,2m,\dotsc,(n_1-1)m\} & \text{if } m \text{ is odd or } n_1 \text{ is even}, \\
\{m,2m,\dotsc,(n_1-1)m,n/2\} & \text{if } m \text{ is even and } n_1 \text{ is odd},
\end{cases}
\]
and
\[
\al_{n-d} \bigl( (a_j)_{j\in J_{n-d}} \bigr) = -\pr_n \biggl( \Bigl( 1 + \sum_{j\in J_{n-d}} a_j^{q^n}  e_j \Bigr) \cdot \Bigl( 1 + \sum_{j\in J_{n-d}} a_j  e_j \Bigr)^{-1} \biggr).
\]
So if $m$ is odd or $n_1$ is even, we can naturally identify $\bA^{n-2(n-d)+1}$ with
\[
(U^{n_1,q_1}\cap Y)\subset U^{n_1,q_1}\subset U^{n,q},
\]
and $\al_{n-d}$ with the map
\[
U^{n_1,q_1}\cap Y \rar{} \bG_a, \qquad g\longmapsto -\pr_n(L_{q^n}(g)).
\]
Recall that $q^n=q_1^{n_1}$.
Applying Proposition \ref{p:reduced-norm-key} with $(n,q)$ replaced by $(n_1,q_1)$, we find that
$-\pr_n(L_{q^n}(g)) = N^{n_1,q_1}(g) - N^{n_1,q_1}(g)^{q_1}$
for all $g\in U^{n_1,q_1}$. Since $q_1=q^m$ is the conductor of $\psi$, the pullback of $\cL_\psi$ by the map $x\mapsto x-x^{q_1}$ is trivial. Therefore $\al_{n-d}^*(\cL_\psi)$ is the trivial local system on $\bA^{n-2(n-d)+1}$, whence
\[
H^\bullet_c \bigl( \bA^{n-2(n-d)+1}, \al_{n-d}^*\cL_\psi \bigr) \cong \ql[-2(n-2(n-d)+1)](-(n-2(n-d)+1)).
\]
Combining this with \eqref{e:6-1} yields
\[
H^\bullet_c \bigl( \bA^d\times(H^-_m\cap Y), \al^*\cL_\psi \bigr) \cong \ql[-2d](-d).
\]
Recalling that $d=(n+n_1-2)/2$ when $m$ is odd or $n_1$ is even, we obtain all the desired assertions \eqref{e:assertion1-1}--\eqref{e:assertion3-1} in this case.

\mbr

Suppose next that $m$ is even and $n_1$ is odd. Then we can naturally identify $\bA^{n-2(n-d)+1}$ with $(U^{n_1,q_1}\cap Y)\times\bG_a$, where the first factor $U^{n_1,q_1}\cap Y$ corresponds to the coordinates $(a_j)_{j=m,2m,\dotsc,(n_1-1)m}$ and the second factor $\bG_a$ corresponds to the coordinate $a_{n/2}$. Under this identification, $\al_{n-d}$ corresponds to the map
\[
(g,x) \longmapsto -\pr_n(L_{q^n}(g)) + x^{q^n+q^{n/2}} - x^{1+q^{n/2}}.
\]
As in the previous case, the pullback of $\cL_\psi$ by the map $(g,x)\mapsto -\pr_n(L_{q^n}(g))$ is trivial. Since the local system $\cL_\psi$ is multiplicative, we see that $H^\bullet_c \bigl( \bA^{n-2(n-d)+1}, \al_{n-d}^*\cL_\psi \bigr)$ is isomorphic to
\[ H^\bullet_c\bigl(\bG_a,f^*(\cL_\psi)\bigr)[-2(n-2(n-d))](-(n-2(n-d))),
\]
where $f:\bG_a\to\bG_a$ is given by $x\mapsto x^{q^n+q^{n/2}} - x^{1+q^{n/2}}$. So by \eqref{e:6-1},
\begin{equation}\label{e:7-1}
H^\bullet_c \bigl( \bA^d\times(H^-_m\cap Y), \al^*\cL_\psi \bigr) \cong H^\bullet_c\bigl(\bG_a,f^*(\cL_\psi)\bigr)[-2(d-1)](-(d-1))
\end{equation}
Now we can factor $f$ as $f=f_1\circ f_2$, where $f_1(x)=x^{q^{n/2}}-x$ and $f_2(x)=x^{1+q^{n/2}}$. Thus $f^*(\cL_\psi)\cong f_2^*f_1^*(\cL_\psi)$. Since $f_1$ is a homomorphism, $f_1^*(\cL_\psi)\cong\cL_{\psi\circ f_1}$ is the multiplicative local system on $\bG_a$ corresponding to the character $\psi\circ f_1:\bF_{q^n}\rar{}\qls$. Now since $\psi$ has conductor $q^m$ and $m\nmid(n/2)$ by assumption, $\psi\circ f_1$ is nontrivial. On the other hand, $(\psi\circ f_1)\bigl\lvert_{\bF_{q^{n/2}}}$ is necessarily trivial. Applying Proposition \ref{p:pullback-q+1-power} below to $q^{n/2}$ in place of $q$ and $\psi\circ f_1$ in place of $\psi$, we see that
\begin{equation}\label{e:8-1}
\dim H^r_c(\bG_a,f_2^*\cL_{\psi\circ f_1}) = \begin{cases}
q^{n/2} & \text{if } r=1, \\
0 & \text{otherwise},
\end{cases}
\end{equation}
and
\begin{equation}\label{e:9-1}
\Fr_{q^n} \text{ acts on } H^1_c(\bG_a,f_2^*\cL_{\psi\circ f_1}) \text{ via } -q^{n/2}
\end{equation}
Recalling that $d=(n+n_1-1)/2$ in the situation under consideration, we see that all the desired assertions \eqref{e:assertion1-1}--\eqref{e:assertion3-1} follow in this case from \eqref{e:7-1}, \eqref{e:8-1} and \eqref{e:9-1}.

\subsection{An auxiliary cohomology calculation} The next result was used in \S\ref{ss:proof-p:how-xi-psi-appears-new}.

\begin{prop}\label{p:pullback-q+1-power}
Let $\psi:\bF_{q^2}\rar{}\qls$ be a nontrivial character such that $\psi$ is trivial on $\bF_q\subset\bF_{q^2}$, and let $\cL_\psi$ be the corresponding Artin-Schreier local system on $\bG_a$ over $\bF_{q^2}$. Write $f:\bG_a\rar{}\bG_a$ for the map $x\mapsto x^{q+1}$. Then
\[
\dim H^i_c(\bG_a,f^*\cL_\psi) = \begin{cases}
q & \text{if } i=1, \\
0 & \text{otherwise}.
\end{cases}
\]
Moreover, the Frobenius $\Fr_{q^2}$ acts on $H^1_c(\bG_a,f^*\cL_\psi)$ via multiplication by $-q$.
\end{prop}

\begin{proof}
Since $x^{q+1}\in\bF_q$ for all $x\in\bF_{q^2}$, we have the identity
\begin{equation}\label{e:*}
\sum_{x\in\bF_{q^2}} \psi(x^{q+1}) = q^2,
\end{equation}
which is consistent with the assertion of the proposition in view of the Grothendieck-Lefschetz trace formula. We will deduce the proposition from \eqref{e:*}. To this end,
observe that if $\pr:\bG_a\to\Spec\bF_{q^2}$ denotes the structure morphism, then
\begin{equation}\label{e:**}
R\pr_!(f^*\cL_\psi)\cong R\pr_!(f_!f^*\cL_\psi) \cong R\pr_!(\cL_\psi\tens f_!\ql),
\end{equation}
where $\ql$ is the constant rank $1$ local system on $\bG_a$ and in the second isomorphism we used the projection formula.

\begin{lem}\label{l:sublemma}
One has
\[
f_!\ql\cong\ql\oplus \bigoplus_{s=1}^q j_!(\sM_s)
\]
for certain nontrivial multiplicative local systems $\sM_1,\sM_2,\dotsc,\sM_q$ on $\bG_m$ over $\bF_{q^2}$, where $j:\bG_m\into\bG_a$ denotes the inclusion map.
\end{lem}

Before proving this lemma, let us explain why it implies the assertion of the proposition. By the lemma and \eqref{e:**}, we have
\begin{equation}\label{e:***}
R\pr_!(f^*\cL_\psi) \cong R\pr_!(\cL_\psi)\oplus\bigoplus_{s=1}^q R\pr_!(\cL_\psi\tens j_!\sM_s).
\end{equation}
Now $R\pr_!(\cL_\psi)=0$ because $\psi$ is nontrivial. By \cite[Prop.~4.2 in \emph{Sommes Trig.}]{sga4.5},
\[
\dim H^r_c(\bG_a,\cL_\psi\tens j_!\sM_s) = \begin{cases}
1 & \text{if } r=1, \\
0 & \text{otherwise},
\end{cases}
\]
and if $\la_s$ is the scalar by which $\Fr_{q^2}$ acts on $H^1_c(\bG_a,\cL_\psi\tens j_!\sM_s)$, then $\abs{\la_s}=q$. Applying the Grothendieck-Lefschetz trace formula to \eqref{e:***} yields
\[
\sum_{x\in\bF_{q^2}} \psi(x^{q+1}) = -(\la_1+\dotsb+\la_q).
\]
Comparing this with \eqref{e:*} and using the fact that $\abs{\la_s}=q$ for all $s$, we see that $\la_s=-q$ for all $s$, which, in view of \eqref{e:***}, yields the proposition.
\end{proof}

\begin{proof}[Proof of Lemma \ref{l:sublemma}]
Since $f^{-1}(0)=\{0\}$, the stalk of $f_!(\ql)$ at $0$ is $1$-dimensional. Since $f$ is a finite morphism, we have $f_!(\ql)=f_*(\ql)$, and by adjunction, there is a natural map $\ql\rar{}f_*(\ql)$ (coming from the natural isomorphism $f^*(\ql)\rar{\simeq}\ql$), which induces an isomorphism on the stalks over $0$.

\mbr

Next let us calculate the restriction of $f_!(\ql)$ to $\bG_m\subset\bG_a$. To this end, consider the restriction of $f$ to a morphism $\bG_m\rar{}\bG_m$. Since $f'(x)=(q+1)x^q=x^q$, the map $f:\bG_m\rar{}\bG_m$ is \'etale, and in fact, it can be identified with the quotient of $\bG_m$ by the finite discrete\footnote{Observe that all the $(q+1)$-st roots of unity already lie in $\bF_{q^2}$.} subgroup $\mu_{q+1}(\bF_{q^2})\subset\bG_m$ of $(q+1)$-st roots of unity. This means that the restriction $f_!(\ql)\bigl\lvert_{\bG_m}$ decomposes as a direct sum $\ql\oplus\bigoplus_{s=1}^q\sM_s$, where the $\sM_s$ are the local systems coming from the nontrivial characters of $\mu_{q+1}(\bF_{q^2})$.

\mbr

In particular, for each $s$, we have a map $\sM_s\into f_!(\ql)\bigl\lvert_{\bG_m}$, which by adjunction induces a map $j_!\sM_s\rar{}f_!(\ql)$. Finally, combining these maps with the map $\ql\rar{}f_!(\ql)$ constructed in the first paragraph of the proof, we obtain a map
\[
\ql\oplus \bigoplus_{s=1}^q j_!(\sM_s) \rar{} f_!(\ql).
\]
By looking at the stalks, one sees that this map is an isomorphism.
\end{proof}

\subsection{The Lang torsor} This subsection contains background material for the proof of Proposition \ref{p:rho-psi-appears-new}. Let $G$ be a connected algebraic group over $\bF_q$. The Lang isogeny $L_q:G\to G$, given by $L_q(g)=\Fr_q(g)\cdot g^{-1}$, identifies $G$ with the quotient of $G$ by the right multiplication action of the finite discrete group $G(\bF_q)$. We will view $G$ as a right $G(\bF_q)$-torsor over itself by means of $L_q$, which we will call the ``Lang torsor.''

\mbr

If $\rho$ is a representation of $G(\bF_q)$ over $\ql$, we denote by $\cE_\rho$ the $\ql$-local system associated to the Lang torsor by means of $\rho$. For the definition of $\cE_\rho$, see \cite[\S\S1.2 and 1.22 in \emph{Sommes Trig.}]{sga4.5}; note that $\cE_\rho$ is denoted by $\cF(\rho)$ in \emph{loc.~cit.}

\begin{prop}
Let $\widehat{G(\bF_q)}$ be a set of representatives of the isomorphism classes of all irreducible representations of $G(\bF_q)$ over $\ql$. Then
\begin{equation}\label{e:lang-pushforward-decomposition}
L_{q!}(\ql) \cong \bigoplus_{\rho\in\widehat{G(\bF_q)}} \rho\tens\cE_\rho
\end{equation}
as local systems with an action of $G(\bF_q)$, where $\ql$ is the constant sheaf on $G$ and the action of $G(\bF_q)$ on the pushforward $L_{q!}(\ql)=L_{q*}(\ql)$ comes from the right multiplication action of $G(\bF_q)$ on $G$.
\end{prop}

We remark that in formula \eqref{e:lang-pushforward-decomposition}, the action of $G(\bF_q)$ on each of the summands $\rho\tens\cE_\rho$ comes only from the action of $G(\bF_q)$ on $\rho$.

\begin{cor}\label{c:hom-rep-into-cohomology}
Let $Y\subset G$ be an $\bF_q$-subvariety and put $X=L_q^{-1}(Y)$. Then
\begin{equation}\label{e:isotypic-component-cohomology}
\Hom_{G(\bF_q)}\bigl(\rho,H^i_c(X,\ql)\bigr) \cong H^i_c\bigl(Y,\cE_\rho\bigl\lvert_Y\bigr)
\end{equation}
as vector spaces with an action of $\Fr_q$, for any $i\in\bZ$ and any representation $\rho$ of $G(\bF_q)$ over $\ql$, where the action of $G(\bF_q)$ on $H^i_c(X,\ql)$ comes from the right multiplication action of $G(\bF_q)$ on $X$.
\end{cor}

\begin{proof}
We have $H^i_c(X,\ql)\cong H^i_c\bigl(Y,L_{q!}(\ql)\bigl\lvert_Y\bigr)$ by the proper base change theorem. Both sides of \eqref{e:isotypic-component-cohomology} are additive with respect to $\rho$, so it suffices to prove it when $\rho$ is irreducible. In that case \eqref{e:isotypic-component-cohomology} follows from \eqref{e:lang-pushforward-decomposition}.
\end{proof}

If $X$ is a scheme of finite type over $\bF_q$ and $\cF$ is a constructible $\ell$-adic sheaf (for example, a $\ql$-local system) on $X$, we denote the corresponding trace of Frobenius function by $t_{\cF}:X(\bF_q)\to\ql$. The next result is \cite[\S1.23 in \emph{Sommes Trig.}]{sga4.5}.

\begin{prop}\label{p:trace-function-of-E-rho}
Given $\ga\in G(\bF_q)$, choose any $g\in G(\overline{\bF}_q)$ with $\ga=L_q(g)$. Then $g^{-1}\cdot\Fr_q(g)\in G(\bF_q)$ and $t_{\cE_\rho}(\ga)=\tr(\rho(g^{-1}\cdot\Fr_q(g)))$.
\end{prop}

We will also need

\begin{prop}\label{p:trace-function-induced-representation}
Let $H\subset G$ be a connected algebraic subgroup and assume that the quotient map $G\to G/H$ has a section $s:G/H\to G$ defined over $\bF_q$. Let $\eta$ be a representation of $H(\bF_q)$ over $\ql$, put $\rho=\Ind_{H(\bF_q)}^{G(\bF_q)}\eta$, and write $\cE_\eta$ $($respectively, $\cE_\rho${}$)$ for the $\ql$-local system on $H$ $($respectively, on $G${}$)$ coming from the Lang isogeny for $H$ $($respectively, for $G${}$)$ via $\eta$ $($respectively, via $\rho${}$)$, as above. Then
\begin{equation}\label{e:compatibility-of-inductions}
t_{\cE_\rho} = \operatorname{ind}_{H(\bF_q)}^{G(\bF_q)} t_{\cE_\eta},
\end{equation}
where $\operatorname{ind}_{H(\bF_q)}^{G(\bF_q)}$ denotes the induction map from conjugation-invariant functions on $H(\bF_q)$ to conjugation-invariant functions on $G(\bF_q)$.
\end{prop}

\begin{rem}
In general, $t_{\cE_\rho}$ is \emph{not} equal to the character of the representation $\rho$, so formula \eqref{e:compatibility-of-inductions} is not evident. However, by \cite[Lem.~1.24 in \emph{Sommes trig.}]{sga4.5}, $t_{\cE_\eta}$ is a conjugation-invariant function on $H(\bF_q)$, so formula \eqref{e:compatibility-of-inductions} makes sense.
\end{rem}

\begin{proof}
Write $\pr_2:(G/H)\times H \to H$ for the second projection and define $F:(G/H)\times H\to G$ by
$F(x,h)=\Fr_q(s(x))\cdot h\cdot s(x)^{-1}$.
Then $\cE_\rho \cong F_!(\pr_2^*\cE_\eta)$ by the argument used in \cite[\S6.2]{DLtheory}. The Grothendieck-Lefschetz trace formula yields
\[
t_{\cE_\rho}(g) = \sum_{\substack{(x,h)\in (G/H)(\bF_q)\times H(\bF_q) \\ F(x,h)=g}} t_{\cE_\eta}(h) \qquad\forall\,g\in G(\bF_q).
\]
Now if $x\in (G/H)(\bF_q)$, then $F(x,h)=s(x)hs(x)^{-1}$ for all $h$. Moreover, since $G$ and $H$ are both connected, we obtain $G(\bF_q)/H(\bF_q)\rar{\simeq}(G/H)(\bF_q)$, so as $x$ ranges over $(G/H)(\bF_q)$, we see that $s(x)$ ranges over a set of representatives of the left cosets of $H(\bF_q)$ in $G(\bF_q)$. Recalling the definition of the map $\operatorname{ind}_{H(\bF_q)}^{G(\bF_q)}$, we obtain \eqref{e:compatibility-of-inductions}.
\end{proof}

\subsection{Proof of Proposition \ref{p:rho-psi-appears-new}}\label{ss:proof-p:rho-psi-appears-new} Recall that $\widetilde{\psi}:H_m(\bF_{q^n})\to\qls$ is the character defined by $\widetilde{\psi}=\psi_1\circ\Nm^{n_1,q_1}\circ\nu_m$ (see \S\ref{ss:outline-new}). In view of the last statements of parts (a) and (b) of Proposition \ref{p:construction-rho-psi-new}, the assertion of Proposition \ref{p:rho-psi-appears-new} is equivalent to:
\[
\Hom_{U^{n,q}(\bF_{q^n})}\Bigl( \Ind_{H_m(\bF_{q^n})}^{U^{n,q}(\bF_{q^n})}(\widetilde{\psi}),H^\bullet_c(X,\ql)\Bigr) \neq 0.
\]

\mbr

Let $\cE_{\widetilde{\psi}}$ be the local system on $H_m$ coming from the Lang isogeny $L_{q^n}:H_m\to H_m$ via $\widetilde{\psi}$. We simply write $\cE$ for the local system on $U^{n,q}$ coming from the Lang isogeny $L_{q^n}:U^{n,q}\rar{}U^{n,q}$ via $\Ind_{H_m(\bF_{q^n})}^{U^{n,q}(\bF_{q^n})}(\widetilde{\psi})$. By Corollary \ref{c:hom-rep-into-cohomology}, we have
\[
\Hom_{U^{n,q}(\bF_{q^n})}\Bigl( \Ind_{H_m(\bF_{q^n})}^{U^{n,q}(\bF_{q^n})}(\widetilde{\psi}),H^\bullet_c(X,\ql)\Bigr) \cong H^\bullet_c\bigl(Y,\cE\bigl\lvert_{Y}\bigr).
\]
So in order to show that the left hand side is nonzero, it suffices (by the Grothendieck-Lefschetz trace formula) to check that
\begin{equation}\label{e:need-nonvanishing-1}
\sum_{y\in Y(\bF_{q^n})} t_{\cE}(y) \neq 0.
\end{equation}
Now by Proposition \ref{p:trace-function-induced-representation},
\begin{equation}\label{e:formula-for-t-E-1}
t_{\cE} = \operatorname{ind}_{H_m(\bF_{q^n})}^{U^{n,q}(\bF_{q^n})}(t_{\cE_{\widetilde{\psi}}}).
\end{equation}

To compute the right hand side of the last identity, we use

\begin{lem}\label{l:trace-function-of-E-psi-tilde-1}
We have\footnote{We point out that in general, $\widetilde{\psi}\neq\psi\circ\pr_n$ on $H_m(\bF_{q^n})$ and $\psi\circ\pr_n:H_m(\bF_{q^n})\to\qls$ is not a group homomorphism.} $t_{\cE_{\widetilde{\psi}}}=\psi\circ\pr_n:H_m(\bF_{q^n})\rar{}\qls$.
\end{lem}

\begin{proof}
The projection $\nu_m:H_m\rar{}U^{n_1,q_1}$ obtained by discarding all summands $a_j e_j$ with $m\nmid j$ is an algebraic group homomorphism (cf.~Remark \ref{r:inclusion-new}), whose kernel is equal to $H^-_m\cap Y$. Moreover, if we view $U^{n_1,q_1}$ as a subgroup of $U^{n,q}$ as explained earlier, then $U^{n_1,q_1}\subset H^-_m$ and $\nu_m$ restricts to the identity on $U^{n_1,q_1}$. So we obtain a semidirect product decomposition $H_m=U^{n_1,q_1}\ltimes(H^-_m\cap Y)$.

\mbr

Now to calculate the function $t_{\cE_{\widetilde{\psi}}}:H_m(\bF_{q^n})\rar{}\qls$ we use Proposition \ref{p:trace-function-of-E-rho}. Fix $\ga\in H_m(\bF_{q^n})$ and choose $g\in H_m(\overline{\bF}_{q^n})$ with $L_{q^n}(g)=\ga$. Write $g=g_1\cdot h$ for uniquely determined $g_1\in U^{n_1,q_1}(\overline{\bF}_{q^n})$ and $h\in (H^-_m\cap Y)(\overline{\bF}_{q^n})$. Then
\[
\nu_m(g^{-1}\cdot\Fr_{q^n}(g)) = \nu_m(h^{-1}\cdot g_1^{-1}\cdot \Fr_{q^n}(g_1)\cdot\Fr_{q^n}(h)) = g_1^{-1}\cdot\Fr_{q^n}(g_1)
\]
because $\nu_m$ is an algebraic group homomorphism, so by Proposition \ref{p:trace-function-of-E-rho},
\[
t_{\cE_{\widetilde{\psi}}}(\ga) = \widetilde{\psi}(g^{-1}\cdot\Fr_{q^n}(g)) = \psi_1(\Nm^{n_1,q_1}(g_1^{-1}\cdot\Fr_{q^n}(g_1))).
\]
Now $g_1^{-1}\cdot\Fr_{q^n}(g_1)\in U^{n_1,q_1}(\bF_{q^n})$ and $\Fr_{q^n}(g_1)=g_1\cdot \bigl(g_1^{-1}\cdot\Fr_{q^n}(g_1)\bigr)$. Applying Proposition \ref{p:reduced-norm-key} to the reduced norm morphism $N^{n_1,q_1}:U^{n_1,q_1}\to\bG_a$, we obtain
\begin{eqnarray*}
\Nm^{n_1,q_1}(g_1^{-1}\cdot\Fr_{q^n}(g_1)) &=& N^{n_1,q_1}(\Fr_{q^n}(g_1))-N^{n_1,q_1}(g_1) \\ &=& \Tr_{\bF_{q^n}/\bF_{q_1}}(\pr_n(L_{q^n}(g_1))),
\end{eqnarray*}
where in the last step we used Corollary \ref{c:reduced-norm-and-trace-new}. Here we note that $L_{q^n}(g_1)\in U^{n_1,q_1}(\bF_{q^n})$ because $L_{q^n}(g_1)=\nu_m(\ga)$. Now $\pr_n(L_{q^n}(g_1))=\pr_n(\ga)$, so we finally obtain
\[
t_{\cE_{\widetilde{\psi}}}(\ga) = \psi_1 \bigl( \Tr_{\bF_{q^n}/\bF_{q_1}}(\pr_n(L_{q^n}(g_1))) \bigr) = \psi(\pr_n(\ga)),
\]
which completes the proof of Lemma \ref{l:trace-function-of-E-psi-tilde-1}.
\end{proof}

Let us now verify \eqref{e:need-nonvanishing-1}. By \eqref{e:formula-for-t-E-1} and Lemma \ref{l:trace-function-of-E-psi-tilde-1}, \[t_{\cE}=\operatorname{ind}_{H_m(\bF_{q^n})}^{U^{n,q}(\bF_{q^n})}(\psi\circ\pr_n),\] so if $\{g_i\}_{i\in I}$ are representatives of all the left cosets of $H_m(\bF_{q^n})$ in $U^{n,q}(\bF_{q^n})$, then
\begin{equation}\label{e:auxiliary-1}
\sum_{y\in Y(\bF_{q^n})} t_{\cE}(y) = \sum_{i\in I} \sum_{y\in Y(\bF_{q^n})\cap(g_i\cdot H_m(\bF_{q^n})\cdot g_i^{-1})} \psi(\pr_n(g_i^{-1}yg_i)).
\end{equation}
We will show that the right hand side is a strictly positive integer. To this end, consider a new group operation on $U^{n,q}$, which we denote by $\boxplus$ and define by
\[
\Bigl( 1 + \sum_{j=1}^n a_j  e_j \Bigr) \boxplus \Bigl( 1 + \sum_{j=1}^n b_j  e_j \Bigr) = 1 + \sum_{j=1}^n (a_j+b_j) e_j.
\]
For any $g\in U^{n,q}$, the map $x\mapsto gxg^{-1}$ is a homomorphism with respect to $\boxplus$ (where we denote the old group operation on $U^{n,q}$ multiplicatively, as usual), as is the map $\pr_n:U^{n,q}\rar{}\bG_a$. Hence for each $i\in I$, the subset
\[
Y(\bF_{q^n})\cap(g_i\cdot H_m(\bF_{q^n})\cdot g_i^{-1}) \subset U^{n,q}(\bF_{q^n})
\]
is a subgroup with respect to $\boxplus$, and the map $y\mapsto\psi(\pr_n(g_i^{-1}yg_i))$ is a character of this subgroup. Therefore the $i$-th summand,
\[
\sum_{y\in Y(\bF_{q^n})\cap(g_i\cdot H_m(\bF_{q^n})\cdot g_i^{-1})} \psi(\pr_n(g_i^{-1}yg_i)),
\]
is equal to either $0$ or the order of $Y(\bF_{q^n})\cap(g_i\cdot H_m(\bF_{q^n})\cdot g_i^{-1})$, which is a positive integer. Finally note that there is an $i_0\in I$ for which $g_{i_0}\in H_m(\bF_{q^n})$. Then the corresponding character $y\mapsto\psi(\pr_n(g_{i_0}^{-1}yg_{i_0}))$ of $Y(\bF_{q^n})\cap(g_{i_0}\cdot H_m(\bF_{q^n})\cdot g_{i_0}^{-1})=Y(\bF_{q^n})\cap H_m(\bF_{q^n})$ is equal to $y\mapsto\psi(\pr_n(y))\equiv\psi(0)=1$ because $Y=\pr_n^{-1}(0)$. So the summand corresponding to $i=i_0$ in \eqref{e:auxiliary-1} is positive, which yields \eqref{e:need-nonvanishing-1}.

\part{Geometric realization of the correspondences}

\section{Notation and constructions}\label{s:formulation-Thm-C}

\subsection{Assumptions and terminology} In this section we state the last main result of the article, Theorem C, which provides a bridge between Theorems A and B from the Introduction. For convenience, we begin by briefly reviewing the notation introduced earlier in the article, which will be used in the current part.

\mbr

We work with a local non-Archimedean field $K$ with residue field $\bF_q$ and integers $m,n\geq 1$. We fix a uniformizer $\varpi\in\cO_K$, an unramified extension $L\supset K$ of degree $n$ and a central division algebra $D$ over $K$ with invariant $1/n$. We also fix a $K$-algebra embedding $L\subset D$. There exists a uniformizer $\Pi\in\cO_D$ such that $\Pi^n=\varpi$ and $\Pi a\Pi^{-1}=\vp(a)$ for all $a\in L$, where $\vp\in\Gal(L/K)$ is the (arithmetic) Frobenius.

\mbr

We assume that the matrix algebra $M_n(K)$ is identified with the algebra $\End_K(L)$ of $K$-vector space endomorphisms of $L$ by means of choosing a basis of $\cO_L$ as an $\cO_K$-module. This determines a $K$-algebra embedding $L\subset M_n(K)$. We write $G=GL_n(K)$ and identify it with the group of $K$-vector space automorphisms of $L$ whenever convenient. In particular, we view $\Gal(L/K)$ as a subgroup of $G$.

\mbr

The groups $L^\times,G,D^\times$ have natural filtrations by principal congruence subgroups, which we denote by $U^r_L$, $U^r_G$ and $U^r_D$, respectively, where $r\geq 1$. We have
$U^r_L = 1+\fp_L^r = 1 + \varpi^r\cO_L$, where $\fp_L\subset\cO_L$ denotes the maximal ideal, $U^r_G = 1 + \varpi^r M_n(\cO_K)$ and $U^r_D = 1 + \Pi^r \cO_D$, where $\cO_D\subset D$ is the unique maximal $\cO_K$-order.

\mbr

We will say that a character $\te:L^\times\to\qls$ is \emph{primitive of conductor $r\geq 2$} if $\te\bigl\lvert_{U^r_L}\equiv 1$ and $\te\bigl\lvert_{U^{r-1}_L}$ has trivial stabilizer in $\Gal(L/K)$.

\subsection{Auxiliary groups} In \S\ref{theunipotentgroup} we introduced a unipotent group\footnote{In Part 2, $\U$ was denoted by $U^{n,q}$ to make the dependence on $n$ and $q$ explicit; since $n$ and $q$ are fixed throughout the present part, we omit them to simplify the notation.} $\U$ defined over $\fqn$ and a smooth hypersurface $X\subset\U$. As in Part 1, we will write $X_{\bfq}=X\tens_{\fqn}\bfq$. In Theorem \ref{existenceofaffinoid} we described a right action of a certain subgroup $\cJ\subset G\times D^\times\times\cW_K$ on $X_{\bfq}^\perf$. In the next few subsections we reformulate the definition of $\cJ$ and its action on $X_{\bfq}^\perf$ in a slightly different way, which is more suitable for the proof of Theorem C. Consider the subgroups
\[
J^m_G = 1 + \fp_L^m + \fp_L^{\ceil{m/2}} C_1^\circ \subset G, \qquad J^m_D = 1 + \fp_L^m + \fp_L^{\floor{m/2}} C_2^\circ \subset D^\times
\]
(we remark that $\floor{m/2}=\ceil{(m-1)/2}$), where $C_1^\circ=C_1\cap M_n(\cO_K)$, $C_2^\circ=C_2\cap\cO_D$ and $C_1\subset M_n(K)$ (resp. $C_2\subset D$) is the orthogonal complement of $L$ with respect to the trace pairing (resp. the reduced trace pairing).

\mbr

With this notation, $\cL^\times=\De(\cO_L^\times)\cdot(J^m_G\times J^m_D)$, where $\cL$ is the linking order introduced in \S\ref{ss:linking-orders-unramified-case} and $\De:L^\times\into G\times D^\times$ is the diagonal embedding. Now define
\[
\diag_{1,2}:L^\times \into G\times D^\times\times \cW_K, \qquad \al \longmapsto (\al,\al,1),
\]
and let $\cJ_0=\diag_{1,2}(L^\times)\cdot(J^m_G\times J^m_D\times\{1\})$. This is a subgroup of $G\times D^\times\times \cW_K$ because $L^\times$ normalizes $J^m_G$ and $J^m_D$. With the notation of Part 1, $\cJ_0$ is the subgroup of $G\times D^\times\times \cW_K$ generated by $\cL^\times\times\{1\}$ and the (central) subgroup consisting of elements of the form $(\al,\al,1)$, where $\al\in K^\times$.

\subsection{The subgroup $\cJ$}\label{ss:the-subgroup-J} We write $j_1:\cW_K\to G$ for the composition of the natural surjective homomorphism $\cW_K\to\Gal(L/K)$ with the inclusion $\Gal(L/K)\into G$. Whenever convenient, we will also view $j_1$ as a homomorphism $\cW_{L/K}\to G$, where $\cW_{L/K}=\cW_K/\overline{[\cW_L,\cW_L]}$ is the relative Weil group of $L$ over $K$.

\mbr

Recall also that there is a group isomorphism $j_2:\cW_{L/K}\rar{\simeq}N_{D^\times}(L^\times)$ that on the subgroup $\cW_L^{ab}=\cW_L/\overline{[\cW_L,\cW_L]}$ restricts to the inverse of the local class field theory isomorphism $\rec_L:L^\times\rar{\simeq}\cW_L^{ab}$, where $N_{D^\times}(L^\times)$ denotes the normalizer of $L^\times$ in $D^\times$. We can also view $j_2$ as a homomorphism $\cW_K\to D^\times$.

\mbr

Let $W\subset G\times D^\times\times\cW_K$ be the subgroup consisting of all elements of the form $(j_1(w),j_2(w),w)$, where $w\in\cW_K$. This subgroup normalizes both $\diag_{1,2}(L^\times)$ and $J^m_G\times J^m_D\times\{1\}$, whence it also normalizes $\cJ_0$. The product $W\cdot\cJ_0$ is an open subgroup of $G\times D^\times\times\cW_K$, which can be identified with a semidirect product $\cW_K\ltimes\cJ_0$ for the obvious action of $\cW_K$ on $\cJ_0$. It is straightforward to check that $W\cdot\cJ_0=\cJ$, where $\cJ\subset G\times D^\times\times\cW_K$ is the subgroup appearing in Theorem \ref{existenceofaffinoid}.

\subsection{The action of $\cJ$ on $X_{\bfq}^\perf$}\label{ss:the-action-of-J-on-X} Recall from \S\ref{theunipotentgroup} that the variety $X$ can already be defined over $\bF_q\subset\fqn$, so that $X=X_0\tens_{\bF_q}\fqn$ and $X_{\bfq}=X_0\tens_{\bF_q}\bfq$. The group $\U(\fqn)$ acts on $X$ by right multiplication; the resulting representations of $\U(\fqn)$ in $H^i_c(X_{\bfq},\ql)$ were calculated in Part 2. The right action of $\cJ$ on $X_{\bfq}^\perf$, which was constructed in Theorem \ref{existenceofaffinoid}, is determined uniquely by the following four rules. The first one describes the action of $J^m_G\times J^m_D\times\{1\}$, the second one describes the action of $\diag_{1,2}(L^\times)$ and the third and fourth ones describes the action of $W$.

\subsubsection{} Let $ Z$ denote the center of $\U$, as in Part 2; then $ Z(\fqn)$ is equal to the center of $\U(\fqn)$, and $ Z(\fqn)$ can be naturally identified with the additive group of $\fqn$. If $m$ is odd, there is a natural surjective homomorphism $J^m_D\to\U(\fqn)$ which induces an isomorphism $J^m_D/J^{m+1}_D\rar{\simeq}\U(\fqn)$ (cf.~the proof of Lemma \ref{Sring}), and the natural projection $J^m_G\to 1+\fp_L^m$ induces a surjective homomorphism $J^m_G\to U^m_L/U^{m+1}_L\cong\fqn\cong Z(\fqn)$. Similarly, if $m$ is even, we have natural surjections $J^m_G\to\U(\fqn)$ and $J^m_D\to Z(\fqn)$.
In both cases, we obtain a surjective homomorphism $J^m_G\times J^m_D \to \U(\fqn)\times  Z(\fqn)$, which determines a right action of the group $J^m_G\times J^m_D\times\{1\}$ on $X_{\bfq}^\perf$ as follows: $\U(\fqn)$ acts by right multiplication and $ Z(\fqn)$ acts via $z:x\mapsto x\cdot z^{-1}$ for any $z\in Z(\fqn)$.

\subsubsection{} If $\al\in K^\times$, then $(\al,\al,1)\in\diag_{1,2}(L^\times)$ acts trivially on $X_{\bfq}^\perf$. If $\al\in\cO_L^\times$, then $(\al,\al,1)$ acts on $X_{\bfq}^\perf$ via $x\mapsto\overline{\al}^{-1}x\overline{\al}$, where $\overline{\al}$ denotes the image of $\al$ in $\cO_L^\times/U_L^1\cong\fqn^\times$ (the conjugation action of $\fqn^\times$ on $X_{\bfq}$ comes from viewing $X_{\bfq}$ as a subvariety of $\cR^\times_{\bfq}$ and identifying $\fqn^\times$ with a subgroup of $\cR^\times(\fqn)$ as in \S\S\ref{ss:definitions-U-n-q}, \ref{ss:proof-proposition-key-c}).

\subsubsection{} If $w\in\overline{[\cW_L,\cW_L]}\subset\cW_K$, the element $(j_1(w),j_2(w),w)=(1,1,w)\in W$ acts trivially on $X_{\bfq}^\perf$, so from now on we form $G\times D^\times\times\cW_{L/K}$, identified with a quotient of $G\times D^\times\times\cW_K$ in the obvious way, and (whenever convenient) use the notation $\cJ_0,W,\cJ$ for the images of these groups in $G\times D^\times\times\cW_{L/K}$ (by a slight abuse of notation). With this in mind, the action of $W\cong\cW_{L/K}$ on $X_{\bfq}^\perf$ is characterized as follows. If $\al\in\cO_L^\times$, then $(1,\al,\rec_L(\al))$ acts on $X_{\bfq}^\perf$ via $x\mapsto\overline{\al}^{-1}x\overline{\al}$ when $m$ is odd, and $(1,\al,\rec_L(\al))$ acts on $X_{\bfq}^\perf$ trivially when $m$ is even.

\subsubsection{}\label{sss:the-action-of-Frobenius} If $\Phi\in\cW_{L/K}$ is the element determined by the equality $j_2(\Phi)=\Pi$, the action of $(j_1(\Phi),j_2(\Phi),\Phi)$ on $X_{\bfq}^\perf=(X_0\tens_{\bF_q}\bfq)^{\perf}$ comes from the automorphism $\id_{X_0}\tens\vp_q^{-1}$ of $X_0\tens_{\bF_q}\bfq$, where $\vp_q:\bfq\to\bfq$ is the Frobenius substitution $a\mapsto a^q$. Since $\cW_{L/K}$ is the semidirect product of the cyclic group generated by $\Phi$ and the subgroup $\rec_L(\cO_L^\times)$, these formulas uniquely determine a right action of $W$ on $X_{\bfq}^\perf$.

\subsection{Formulation of the main result}\label{ss:Theorem-C} The right action of $\cJ$ on $X_{\bfq}^\perf$ yields a (smooth) representation of $\cJ$ in the vector space $H^{n-1}_c(X_{\bfq}^\perf,\ql)$. The natural morphism $X_{\bfq}^\perf\to X_{\bfq}$ induces an isomorphism between $H^{n-1}_c(X_{\bfq},\ql)$ and $H^{n-1}_c(X_{\bfq}^\perf,\ql)$, which we use to tacitly identify the two spaces. In particular, we let $\cH^{n-1}$ denote the resulting representation of $\cJ$ in $H^{n-1}_c(X_{\bfq},\ql)$; this will allow us to write $X_{\bfq}$ instead of $X_{\bfq}^\perf$ in all that follows.

\begin{rems}\label{rems:actions-on-X-over-F-q-bar}
\begin{enumerate}[(1)]
\item Recall that with the conventions of Part 2 (see Remarks \ref{rems:ell-adic-cohomology}), the underlying vector space of $H^{n-1}_c(X,\ql)$ is equal to $H^{n-1}_c(X_{\bfq},\ql)$. Thus Theorem \ref{t:cohomology-of-X} gives us a description of $H^{n-1}_c(X_{\bfq},\ql)$ as a representation of $\U(\fqn)$.
 \sbr
\item If $\Fr_q$ is the endomorphism of $X_{\bfq}=X_0\tens_{\bF_q}\bfq$ induced by the absolute Frobenius on $X_0$, then the actions of $\Fr_q$ and $\id_{X_0}\tens\vp_q^{-1}$ on $H^\bullet_c(X_{\bfq},\ql)$ coincide. Hence if $\Phi\in\cW_{L/K}$ is the element such that $j_2(\Phi)=\Pi$, then by Theorem \ref{t:cohomology-of-X}(b), $(j_1(\Phi),j_2(\Phi),\Phi)^n$ acts on $H^{n-1}_c(X_{\bfq},\ql)$ via the scalar $(-1)^{n-1} q^{n(n-1)/2}$.
\end{enumerate}
\end{rems}

The next result is proved in \S\ref{s:proof-Thm-C}. In \S\ref{s:proof-Thm-A} we use it to prove Theorem A.

\begin{thmc}
Let $\pi,\rho,\sg$ be smooth irreducible representations of $G,D^\times$ and $\cW_K$, respectively. The following two conditions are equivalent:

\begin{enumerate}[$($i$)$]
\item $\Hom_{\cJ}\bigl(\pi\tens\rho\tens\sg\bigl\lvert_{\cJ},\cH^{n-1}\bigr)\neq 0$;
 \sbr
\item $\pi$ corresponds to $\rho^\vee$ under the local Jacquet-Langlands correspondence, and to the twist $\sg\bigl(\frac{n-1}{2}\bigr)$ under the local Langlands correspondence, and there is a primitive character $\te:L^\times\to\qls$ of conductor $m+1$ such that $\sg\cong\sg_\te = \Ind_{L/K}\te$.
\end{enumerate}
Moreover, if these conditions hold, then $\dim\Hom_{\cJ}\bigl(\pi\tens\rho\tens\sg\bigl\lvert_{\cJ},\cH^{n-1}\bigr)=1$.
\end{thmc}

Explicitly, $\sg\bigl(\frac{n-1}{2}\bigr)$ is the twist of $\sg$ by the character of $\cW_K$ coming from
\[
\abs{\cdot}^{(n-1)/2}\circ\rec_K^{-1}: \cW_K^{ab} \rar{} K^\times\rar{}\qls,
\]
where $\abs{\cdot}$ is the normalized absolute value on $K$, so that $\abs{\varpi}=q^{-1}$. The apparent dependence of this twist on the choice of $\sqrt{q}\in\ql$ when $n$ is even is explained as follows: if $\eps:K^\times\rar{}\qls$ is a character whose kernel equals $N_{L/K}(L^\times)$, then for even $n$, we have $\eps^{n/2}(\varpi)=-1$ and $\eps^{n/2}\bigl\lvert_{\cO_K^\times}\equiv 1$, and $\sg_\te$ is invariant under twisting by $\eps\circ\rec_K^{-1}$ (hence also by $\eps^{n/2}\circ\rec_K^{-1}$).

\section{Proof of Theorem C}\label{s:proof-Thm-C}

\subsection{Outline} We first introduce some notation and formulate two auxiliary results that will be used in the proof of Theorem C. The proof is given in the next subsection, and the auxiliary results are proved in the remainder of this section.

\mbr

The main ingredients in the proof of Theorem C are Theorem \ref{t:cohomology-of-X} from Part 2 and some results of Kazhdan and Henniart on the local Langlands and Jacquet-Langlands correspondences \cite{Kazhdan-on-lifting,Henniart-MathNachr1992,Henniart-JLC-I}. The results of Henniart on which we rely were restated in a form suited for our purposes in the article \cite{SpecialCasesCorrespondences}. However, the portion of the latter article on which the current one depends is rather small: most of \emph{op.~cit.} was devoted to background from $p$-adic representation theory and to the proof of a special case of Theorem \ref{t:cohomology-of-X}, namely, \cite[Thm.~2.9]{SpecialCasesCorrespondences}. The full strength of Theorem \ref{t:cohomology-of-X} is needed for Theorem C, and the proof of Theorem \ref{t:cohomology-of-X} that we gave in Part 2 is independent of \emph{op.~cit.}

\subsubsection{} As we remarked earlier, the subgroup
\[
\{1\}\times\{1\}\times\overline{[\cW_L,\cW_L]}\subset\cJ\subset G\times D^\times\times\cW_K
\]
acts trivially on $X_{\bfq}$. Therefore without loss of generality we can replace $\cW_K$ with its quotient $\cW_{L/K}$ throughout the proof of the theorem. As before, by a slight abuse of notation, we also use the letter $\cJ$ for the image of $\cJ$ in $G\times D^\times\times\cW_{L/K}$. The representation $\sg_\te=\Ind_{L/K}\te$ is also trivial on $\overline{[\cW_L,\cW_L]}$, so from now on we will view $\sg_\te$ as a representation of $\cW_{L/K}$. Explicitly, $\sg_\te=\Ind_{\cW_L^{ab}}^{\cW_{L/K}}(\te\circ\rec_L^{-1})$.

\subsubsection{} Recall that the isomorphism $j_2:\cW_{L/K}\rar{\simeq}N_{D^\times}(L^\times)$ used in \S\ref{ss:the-subgroup-J} restricts to $\rec_L^{-1}:\cW_L^{ab}\rar{\simeq}L^\times$. We introduce the homomorphism
\[
\diag_{2,3}:L^\times\rar{}G\times D^\times\times\cW_{L/K}, \qquad \al\mapsto(1,\al,j_2^{-1}(\al))
\]
and consider the normal subgroup of index $n$
\[
\cJ_1:=\diag_{1,2}(L^\times)\cdot (J^m_G\times J^m_D \times\{1\}) \cdot \diag_{2,3}(L^\times) \subset \cJ.
\]
We remark that $\cJ_1$ can also be viewed as a subgroup of $G\times D^\times\times\cW_L^{ab}$. For each character $\psi:\fqn\to\qls$ with trivial $\Gal(\fqn/\bF_q)$-stabilizer, let $\cH^{n-1}[\psi]\subset\cH^{n-1}$ denote the subspace on which $Z(\fqn)\cong\fqn$ acts via $\psi$. (As in Part 2, the action of $Z(\fqn)\subset\U(\fqn)$ on $\cH^{n-1}=H^{n-1}_c(X_{\bfq},\ql)$ comes from the right multiplication action of $\U(\fqn)$ on $X_{\bfq}$.)

\begin{lem}\label{l:new}
\begin{enumerate}[$($a$)$]
\item If $\psi:\fqn\to\qls$ is a character with trivial $\Gal(\fqn/\bF_q)$-stabilizer, then $\cH^{n-1}[\psi]$ is stable under $\cJ_1$, and is irreducible as a representation of $\cJ_1$.
 \sbr
\item We have $\cH^{n-1}=\bigoplus\cH^{n-1}[\psi]$, the sum ranging over all characters $\psi:\fqn\to\qls$ with trivial $\Gal(\fqn/\bF_q)$-stabilizer.
\end{enumerate}
\end{lem}

\begin{proof}
The action of $Z(\fqn)$ on $X_{\bfq}$ commutes with the action of $\cJ_1$, so $\cH^{n-1}$ is stable under $\cJ_1$. By Theorem \ref{t:cohomology-of-X}, the characters of $Z(\fqn)\cong\fqn$ that have \emph{non}trivial $\Gal(\fqn/\bF_q)$-stabilizer do not appear in $\cH^{n-1}$ (they appear in higher cohomological degrees), which yields assertion (b). Theorem \ref{t:cohomology-of-X} also implies that $\cH^{n-1}[\psi]$ is already irreducible as a representation of $\U(\fqn)$. Since $\cJ_1$ contains a subgroup $\cJ_2\subset\cJ_1$ whose action on $X_{\bfq}$ comes from the right multiplication action of $\U(\fqn)$ on $X_{\bfq}$ via a surjective homomorphism $\cJ_2\to\U(\fqn)$, assertion (a) follows. (We have $\cJ_2=J^m_G\times\{1\}\times\{1\}$ if $m$ is even and $\cJ_2=\{1\}\times J^m_D\times\{1\}$ if $m$ is odd.)
\end{proof}

\subsubsection{} The first key step of the proof of Theorem C is
\begin{prop}\label{p:nonzero-homs}
Let $\pi,\rho,\sg$ be smooth irreducible representations of $G,D^\times$, $\cW_{L/K}$.

\begin{enumerate}[$($a$)$]
\item If $\Hom_{\cJ}\bigl(\pi\tens\rho\tens\sg\bigl\lvert_{\cJ},\cH^{n-1}\bigr)\neq 0$, there is a primitive character $\te:L^\times\to\qls$ of conductor $m+1$ such that $\sg\cong\sg_\te=\Ind_{\cW_L^{ab}}^{\cW_{L/K}}(\te\circ\rec_L^{-1})$.
 \sbr
\item If there is a primitive character $\te:L^\times\to\qls$ of conductor $m+1$ such that $\sg\cong\sg_\te$, then there exists a natural vector space isomorphism
\begin{equation}\label{e:isomorphism-of-Homs}
\Hom_{\cJ}\bigl(\pi\tens\rho\tens\sg\bigl\lvert_{\cJ},\cH^{n-1}\bigr) \cong \Hom_{\cJ_1}\bigl(\pi\tens\rho\tens(\te\circ\rec_L^{-1})\bigl\lvert_{\cJ_1},\cH^{n-1}[\psi^{\pm 1}]\bigr),
\end{equation}
where $\psi:\fqn\to\qls$ comes from $\te\bigl\lvert_{U^m_L}$ via the identification $U^m_L/U^{m+1}_L\cong\fqn$ and the sign $\pm 1$ is equal to $(-1)^m$.
\end{enumerate}
\end{prop}

This proposition is proved in \S\ref{ss:proof-p:nonzero-homs}. We see that to complete the proof of Theorem C it suffices to show that the right hand side of \eqref{e:isomorphism-of-Homs} is $1$-dimensional whenever it is nonzero, and to identify those pairs $(\pi,\rho)$ for which it is nonzero (for a given primitive character $\te:L^\times\to\qls$ of conductor $m+1$). We will now introduce some additional notation and formulate two more lemmas from which Theorem C follows.

\subsubsection{} In the remainder of the proof of Theorem C (including \S\ref{ss:proof-Thm-C-completion}) we work with a fixed primitive character $\te:L^\times\to\qls$ of conductor $m+1$ and let $\psi:\fqn\to\qls$ be the character induced by $\te$ as in the last proposition.

\subsubsection{} The subgroups $J^m_G\subset G$ and $J^m_D\subset D^\times$ are each normalized by $L^\times$. We let $L^\times$ act on each of them by conjugation and form the corresponding semidirect products $\Jt^m_G=L^\times\ltimes J^m_G$ and $\Jt^m_D=L^\times\ltimes J^m_D$. Let $\cJt_1=\Jt^m_G\times\Jt^m_D$. We obtain a surjective homomorphism $f:\cJt_1\rar{}\cJ_1$ given by
\[
f(\al,x,\be,y)=\bigl(\al\cdot x,\be\cdot y,\rec_L(\al^{-1}\cdot\be)\bigr), \qquad \al,\be\in L^\times,\ x\in J^m_G,\ y\in J^m_D.
\]
Let $f^*$ denote the functor of pullback via $f$ from representations of $\cJ_1$ to representations of $\cJt_1$. Since $f$ is surjective, this functor is fully faithful. If $R$ is a representation of $L^\times\cdot J^m_G$ or $L^\times\cdot J^m_D$, we will denote by $\widetilde{R}$ its pullback to $\Jt^m_G$ (respectively, $\Jt^m_G$) via the multiplication map $\Jt^m_G\rar{}L^\times\cdot J^m_G$ (respectively, $\Jt^m_D\rar{}L^\times\cdot J^m_D$). If $\nu$ is a character of $L^\times$ and $R'$ is a representation of either $\Jt^m_G$ or $\Jt^m_D$, we will write $\nu\cdot R'$ for the twist of $R'$ by the pullback of $\nu$ via the natural projection onto $L^\times$.

\begin{lem}\label{l:pullback-of-pi-rho}
$f^*\bigl(\pi\tens\rho\tens(\te\circ\rec_L^{-1})\bigr)\
\cong(\te^{-1}\cdot\widetilde{\pi})\tens(\te\cdot\widetilde{\rho})$ as representations of $\cJt_1$.
\end{lem}

This lemma follows at once from the definitions.

\subsubsection{} Let $\eta:K^\times\to\qls$ be the character defined by $\eta(x)=\abs{x}^{(n-1)/2}$, where $\abs{\cdot}$ is the normalized absolute value on $K$. If $n$ is even, to define $\eta$ one needs to choose $\sqrt{q}\in\qls$, though ultimately this choice is irrelevant (see the remark following Theorem C). We also let $\xi:L^\times\to\qls$ be the character determined by $\xi\bigl\lvert_{\cO_L^\times}\equiv 1$ and $\xi(\varpi)=(-1)^{n-1}$. We now need to define representations $R_\psi$ and $R'_{\psi^{-1}}$ of the groups $\Jt^m_G$ and $\Jt^m_D$, respectively; their construction depends on the parity of $m$.

\subsubsection*{The case where $m$ is odd}
Let $R_\psi:\Jt^m_G\rar{}\qls$ be the character defined as the composition
\[
\Jt^m_G = L^\times \ltimes J^m_G \rar{} J^m_G \rar{} 1+\fp^m_L = U^m_L \rar{} U^m_L/U^{m+1}_L\cong\fqn \rar{\psi}\qls,
\]
where the first two arrows are the natural projections (even though they are not group homomorphisms, $R_\psi$ is). On the other hand, we have a natural surjective homomorphism $\Jt^m_D\to\cR^\times(\fqn)$, where $\cR$ is the ring scheme introduced in \S\ref{ss:definitions-U-n-q}. It is defined by combining the surjection $J^m_D\to\U(\fqn)\into\cR^\times(\fqn)$ (cf.~the proof of Lemma \ref{Sring}) and the surjection $L^\times\to\cO_L^\times\to (\cO_L/(\varpi))^\times=\fqn^\times\into\cR^\times(\fqn)$, where the arrow $L^\times\to\cO_L^\times$ is the splitting of $\cO_L^\times\into L^\times$ defined by $\varpi$.

\mbr

The right multiplication action of $\U(\fqn)\subset\cR^\times(\fqn)$ on $X_{\bfq}$ and the conjugation action of $\fqn^\times\subset\cR^\times(\fqn)$ on $X_{\bfq}$ combine to form a right action of $\cR^\times(\fqn)$ on $X_{\bfq}$, which yields an irreducible representation $R'_{\psi^{-1}}$ of $\Jt^m_D$ in the space $\cH^{n-1}[\psi^{-1}]$.

\subsubsection*{The case where $m$ is even} Here the roles of $G$ and $D$ are reversed. We have a surjective homomorphism $\Jt^m_G\to\cR^\times(\fqn)$, which yields a right action of $\Jt^m_G$ on $X_{\bfq}$ and an irreducible representation $R_{\psi}$ of $\Jt^m_G$ in the space $\cH^{n-1}[\psi]$. On the other hand, $R'_{\psi^{-1}}$ is the $1$-dimensional representation of $\Jt^m_D$ defined as the composition
\[
\Jt^m_D = L^\times \ltimes J^m_D \rar{} J^m_D \rar{} 1+\fp^m_L = U^m_L \rar{} U^m_L/U^{m+1}_L\cong\fqn \rar{\psi^{-1}}\qls.
\]

\mbr

In each of the two cases considered above, we have the following

\begin{lem}\label{l:pullback-of-H}
$f^*\cH^{n-1}[\psi^{\pm 1}]\cong\bigl((\eta\circ N_{L/K})\cdot\xi\cdot R_\psi\bigr)\tens\bigl((\eta\circ N_{L/K})^{-1}\cdot\xi^{-1}\cdot R'_{\psi^{-1}}\bigr)$.
\end{lem}

As usual, the sign is given by $\pm 1=(-1)^m$. The lemma is proved in \S\ref{ss:proof-l:pullback-of-H}.

\subsection{Completion of the proof of Theorem C}\label{ss:proof-Thm-C-completion} Combining Lemmas \ref{l:pullback-of-pi-rho} and \ref{l:pullback-of-H}, we see that the right hand side of \eqref{e:isomorphism-of-Homs} is nonzero if and only if
\begin{equation}\label{e:nonzero-hom-1}
\Hom_{\Jt^m_G} \bigl( \te^{-1}\cdot\widetilde{\pi} , (\eta\circ N_{L/K})\cdot\xi\cdot R_\psi \bigr) \neq 0
\end{equation}
and
\begin{equation}\label{e:nonzero-hom-2}
\Hom_{\Jt^m_D} \bigl( \te\cdot\widetilde{\rho} , (\eta\circ N_{L/K})^{-1}\cdot\xi^{-1}\cdot R'_{\psi^{-1}} \bigr) \neq 0.
\end{equation}
The representation $(\eta\circ N_{L/K})\cdot\xi\cdot\te\cdot R_\psi$ of $\Jt^m_G$ is trivial on the kernel of the multiplication map $\Jt^m_G\rar{}L^\times\cdot J^m_G$ and hence descends to a representation of $L^\times\cdot J^m_G$ that we will denote by $R(\te)$. Similarly, the representation $(\eta\circ N_{L/K})^{-1}\cdot\xi^{-1}\cdot\te^{-1}\cdot R'_{\psi^{-1}}$ of $\Jt^m_D$ is trivial on the kernel of the multiplication map $\Jt^m_D\rar{}L^\times\cdot J^m_D$ and hence descends to a representation of $L^\times\cdot J^m_D$ that we will denote by $R'(\te^{-1})$. We see that the \eqref{e:nonzero-hom-1} is equivalent to the nonvanishing of $\Hom_{L^\times\cdot J^m_G}(\pi\bigl\lvert_{L^\times\cdot J^m_G},R(\te))$ and \eqref{e:nonzero-hom-2} is equivalent to the nonvanishing of $\Hom_{L^\times\cdot J^m_D}(\rho\bigl\lvert_{L^\times\cdot J^m_D},R'(\te^{-1}))$.

\mbr

Now let $LLC(\sg)$ denote the representation of $G$ corresponding to $\sg$ under the local Langlands correspondence and let $JLC(\sg)$ denote the representation of $D^\times$ corresponding to $LLC(\sg)$ under the Jacquet-Langlands correspondence. The results of \cite[\S2]{SpecialCasesCorrespondences} show that $\Ind_{L^\times\cdot J^m_G}^G(R(\te))\cong (\eta\circ\det)\tens LLC(\sg)$ and $\Ind_{L^\times\cdot J^m_D}^{D^\times}(R'(\te^{-1}))\cong (\eta\circ\Nrd)^{-1}\tens JLC(\sg)^\vee$, where $\Nrd:D^\times\rar{}K^\times$ is the reduced norm and $\det:G\rar{}K^\times$ is the usual determinant. Thus properties \eqref{e:nonzero-hom-1}--\eqref{e:nonzero-hom-2} hold if and only if $\pi\cong(\eta\circ\det)\tens LLC(\sg)$ and $\rho\cong(\eta\circ\Nrd)^{-1}\tens JLC(\sg)^\vee$, in which case both $\Hom$ spaces are $1$-dimensional. This finishes the proof.

\subsection{Proof of Proposition \ref{p:nonzero-homs} (a)}\label{ss:proof-p:nonzero-homs} Assume that $\Hom_{\cJ}\bigl(\pi\tens\rho\tens\sg\bigl\lvert_{\cJ},\cH^{n-1}\bigr)\neq 0$ and choose any character $\te:L^\times\to\qls$ such that $\te\circ\rec_L^{-1}$ is a quotient of $\sg\bigl\lvert_{\cW_L^{ab}}$.

\subsubsection{Case 1}\label{sss:proof-a-odd-case} Suppose that $m$ is odd. If $\al\in U^m_L$, then the element
\[
(\al^{-1},1,\rec_L(\al)) = (\al^{-1},\al^{-1},1)\cdot (1,\al,\rec_L(\al)) \in\cJ_1
\]
acts trivially on $X_{\bfq}$, while the element $(\al^{-1},1,1)$ acts on $X_{\bfq}$ as multiplication by the element of $Z(\fqn)$ coming from $\al$ via the identification $U^m_L/U^{m+1}_L\cong\fqn\cong Z(\fqn)$. Hence $(1,1,\rec_L(\al))$ acts on $X_{\bfq}$ as the inverse of the latter.

\mbr

The assumption that $\Hom_{\cJ}\bigl(\pi\tens\rho\tens\sg\bigl\lvert_{\cJ},\cH^{n-1}\bigr)\neq 0$ and the definition of $\te$ already imply that $\te\bigl\lvert_{U^{m+1}_L}\equiv 1$ and that $\te\bigl\lvert_{U^m_L}$ comes from a character $\psi:\fqn\to\qls$ with trivial $\Gal(\fqn/\bF_q)$-stabilizer. \emph{A fortiori}, $\te$ has trivial $\Gal(L/K)$-stabilizer, so $\sg_\te$ is irreducible. By construction, $\Hom_{\cW_{L/K}}(\sg,\sg_\te)\neq 0$, so $\sg\cong\sg_\te$.

\subsubsection{Case 2}\label{sss:proof-a-even-case} Now let $m$ be even. Then we use the factorization
\[
(1,1,\rec_L(\al)) = (1,\al^{-1},1) \cdot (1,\al,\rec_L(\al))
\]
to conclude that if $\al\in U^m_L$, then the element $(1,1,\rec_L(\al))$ acts on $X_{\bfq}$ as multiplication by the element of $Z(\fqn)$ coming from $\al$ via the identification $U^m_L/U^{m+1}_L\cong\fqn\cong Z(\fqn)$. The rest of the argument is the same as in the case where $m$ is odd.

\subsection{Proof of Proposition \ref{p:nonzero-homs}(b)} Fix $\te$ and $\psi$ as in the statement of the proposition and suppose that $\sg=\sg_\te$. Let us construct the isomorphism \eqref{e:isomorphism-of-Homs}. As an intermediate step, define $\cH^{n-1}_\sg=\bigoplus\cH^{n-1}[\ga\cdot\psi^{\pm 1}]$, where $\ga$ ranges over $\Gal(\fqn/\bF_q)$ and $\ga\cdot\psi^{\pm 1}(x)=\psi(\ga(x))^{\pm 1}$, the sign being $\pm 1=(-1)^m$, as in the proposition.

\begin{lem}\label{l:induce-from-J-1-to-J}
$\cH^{n-1}_\sg\cong\Ind_{\cJ_1}^{\cJ}(\cH^{n-1}[\psi^{\pm 1}])$ as representations of $\cJ$.
\end{lem}

This lemma is proved in \S\ref{ss:proof-l:induce-from-J-1-to-J}.
Frobenius reciprocity yields a natural isomorphism
\begin{equation}\label{e:new-isomorphism}
\Hom_{\cJ}\bigl(\pi\tens\rho\tens\sg\bigl\lvert_{\cJ},\cH^{n-1}_\sg\bigr) \cong \Hom_{\cJ_1}\bigl(\pi\tens\rho\tens\sg\bigl\lvert_{\cJ_1},\cH^{n-1}[\psi^{\pm 1}]\bigr).
\end{equation}
The argument of \S\ref{sss:proof-a-odd-case} (resp. \S\ref{sss:proof-a-even-case}) when $m$ is odd (resp. even) shows that the left hand side of \eqref{e:new-isomorphism} is equal to the left hand side of \eqref{e:isomorphism-of-Homs}. The right hand side of \eqref{e:new-isomorphism} can be naturally identified with the right hand side of \eqref{e:isomorphism-of-Homs} because $\sg\bigl\lvert_{\cW_L^{ab}}$ is isomorphic to the direct sum of $\Gal(L/K)$-conjugates of $\te\circ\rec_L^{-1}$, which have pairwise distinct restrictions to $U^m_L$.
\qed

\subsection{Proof of Lemma \ref{l:induce-from-J-1-to-J}}\label{ss:proof-l:induce-from-J-1-to-J} Let $\Phi\in\cW_{L/K}$ be the element such that $j_2(\Phi)=\Pi\in N_{D^\times}(L^\times)$ and write $\Phi_0=(j_1(\Phi),j_2(\Phi),\Phi)\in W$ (cf.~\S\ref{sss:the-action-of-Frobenius}). The quotient group $\cJ/\cJ_1$ is cyclic of order $n$ and is generated by the image of $\Phi_0$.

\mbr

Let us view $U^m_L$ as a subgroup of $\cJ_1$ via the embedding $U^m_L\into J^m_G$ (resp. $U^m_L\into J^m_D$) when $m$ is odd (resp. even). In each case, the induced action of $U^m_L$ on $X_{\bfq}$ is the inverse of the action induced by the multiplication action of $Z(\fqn)$ via the identification $U^m_L/U^{m+1}_L\cong Z(\fqn)$. Moreover, the conjugation action of $\Phi_0$ on $U^m_L$ coincides with the action of the Frobenius $\vp\in\Gal(L/K)$. This proves the lemma.

\subsection{Proof of Lemma \ref{l:pullback-of-H}}\label{ss:proof-l:pullback-of-H} Note that $(\eta\circ N_{L/K})\cdot\xi$ is the unramified character of $L^\times$ that takes value $(-1)^{n-1}\cdot q^{-n(n-1)/2}$ on $\varpi$. On the other hand, by Remark \ref{rems:actions-on-X-over-F-q-bar}(2), the action of $\id_{X_0}\tens\vp_q^{-n}$ on $\cH^{n-1}$ induced by its action on $X_{\bfq}=X_0\tens_{\bF_q}\bfq$ (see \S\ref{sss:the-action-of-Frobenius}) is given by the scalar $(-1)^{n-1}\cdot q^{n(n-1)/2}$. To obtain Lemma \ref{l:pullback-of-H} it remains to combine the following observations.

\subsubsection{} The element
\[
f\bigl((\varpi,1),(1,1)\bigr) = (\varpi,1,\rec_L(\varpi^{-1})) = (\varpi,\varpi,1)\cdot(1,\varpi^{-1},\rec_L(\varpi^{-1})) \in \cJ
\]
acts on $X_{\bfq}$ via $\id_{X_0}\tens\vp_q^{n}$ because $(\varpi,\varpi,1)$ acts trivially and $(1,\varpi^{-1},\rec_L(\varpi^{-1}))=\Phi_0^{-n}$, where $\Phi_0=(j_1(\Phi),j_2(\Phi),\Phi)\in W$ (cf.~\S\ref{sss:the-action-of-Frobenius}).

\subsubsection{} Similarly, $f\bigl((1,1),(\varpi,1)\bigr) = (1,\varpi,\rec_L(\varpi)) \in \cJ$ acts on $X_{\bfq}$ via $\id_{X_0}\tens\vp_q^{-n}$.

\subsubsection{} Suppose that $\al\in\cO_L^\times$. Then the element
\[
f\bigl( (\al,1),(1,1) \bigr) = (\al,1,\rec_L(\al^{-1})) = (\al,\al,1) \cdot (1,\al^{-1},\rec_L(\al^{-1})) \in \cJ
\]
acts trivially on $X_{\bfq}$ if $m$ is odd, and acts on $X_{\bfq}$ via $x\mapsto\overline{\al}^{-1}x\overline{\al}$, where $\overline{\al}$ denotes the image of $\al$ in $\cO_L^\times/U^1_L\cong\fqn^\times$, if $m$ is even.

\subsubsection{} Similarly, $f\bigl((1,1),(\al,1)\bigr) = (1,\al,\rec_L(\al)) \in \cJ$ acts on $X_{\bfq}$ via $x\mapsto\overline{\al}^{-1}x\overline{\al}$ if $m$ is odd, and acts trivially on $X_{\bfq}$ if $m$ is even.

\section{Proof of Theorem A}\label{s:proof-Thm-A}


Theorem A claimed the existence of an open affinoid subset $V\subset\mathcal{M}_{H,\infty,C}$ for which $\overline{V}$ realizes the Langlands correspondences in its middle cohomology.  We now give the proof of Theorem A.

In Part 1, we constructed an affinoid $\mathcal{Z}\subset \mathcal{M}_{H,\infty,C}$ and described its reduction $\overline{\mathcal{Z}}$ in Thm. \ref{existenceofaffinoid}.  The stabilizer of $\mathcal{Z}$ in $GL_n(K)\times D^\times\times \cW_K$ was a certain open subgroup $\mathcal{J}$.  As for $\overline{\mathcal{Z}}$, it is a union of (uncountably many) copies of $X^{\perf}_{\overline{\FF}_q}$, the perfection of an affine variety $X_{\overline{\FF}_q}$.  Cor. \ref{from-X-to-Z} gave the cohomology of $\overline{\mathcal{Z}}$ in terms of $X_{\bfq}$:
\[ H^\bullet_c(\overline{Z},\overline{\Q}_\ell)
=\bigoplus_\psi H^\bullet_c(X_{\overline{\mathbb{F}}_q},\overline{\Q}_\ell)\otimes (\psi\circ\chi),\]
where $\psi$ runs over characters of $1+\gp_K^m$.  Here $\chi\from GL_n(K)\times D^\times\times \cW_K\to K^\times$ is the determinant map of \S\ref{main-result-part-1}.  Thus representations of $\mathcal{J}$ appearing in $H^{n-1}_c(\overline{\Z},\overline{\Q}_\ell)$ are exactly the twists of representations appearing in $H^{n-1}_c(X_{\overline{\FF}_q},\overline{\Q}_\ell)$ by characters of $1+\gp_K^m$.

The desired affinoid $V$ is the union of translates of $\mathcal{Z}$ by $GL_n(K)\times D^\times\times \cW_K$.  Thus the cohomology of $\overline{V}$ is induced from that of $\overline{\mathcal{Z}}$ from $\mathcal{J}$ to $GL_n(K)\times D^\times\times \cW_K$.  This shows that an irreducible representation $\pi\otimes\rho\otimes\sigma$ of $GL_n(K)\times D^\times\times \cW_K$ appears in $H^{n-1}_c(\overline{V},\overline{\Q}_\ell)$ if and only if its restriction to $\mathcal{J}$ appears in $H^{n-1}(\overline{\mathcal{Z}},\overline{\Q}_\ell)$.  This happens exactly when (up to a twist) $\pi\otimes\rho\otimes\sigma\vert_{\mathcal{J}}$ appears in $H^{n-1}(X_{\overline{\FF}_q},\overline{\Q}_\ell)$.  By Theorem C, this happens if and only if $\pi$ corresponds to $\rho^\vee$ and $\sigma\left(\frac{n-1}{2}\right)$ under the local Langlands correspondence, and (up to a twist) there is a primitive character $\theta\from L^\times\to\overline{\Q}_\ell^\times$ of conductor $m+1$ such that $\sigma=\Ind_{L/K}\theta$.   In this case, $\pi\otimes\rho\otimes\sigma\vert_{\mathcal{J}}$ appears with multiplicity 1 in $H^{n-1}(X\otimes\overline{\FF}_q,\overline{\Q}_\ell)$, so that $\pi\otimes\rho\otimes\sigma$ appears with multiplicity 1 in $H^{n-1}(\overline{V},\overline{\Q}_\ell)$.  This completes the proof of Theorem A.

\bibliographystyle{alpha}
\bibliography{MaximalVarieties}

\begin{thebibliography}{SGA73}

\bibitem[BD06]{intro}
Mitya Boyarchenko and Vladimir Drinfeld.
\newblock A motivated introduction to character sheaves and the orbit method
  for unipotent groups in positive characteristic.
\newblock Preprint, arXiv.org:math.RT/0609769v2, 2006.

\bibitem[Boy12]{DLtheory}
Mitya Boyarchenko.
\newblock {D}eligne-{L}usztig constructions for unipotent and $p$-adic groups.
\newblock Preprint, arXiv:1207.5876, 2012.

\bibitem[BW13]{SpecialCasesCorrespondences}
Mitya Boyarchenko and Jared Weinstein.
\newblock Geometric realization of special cases of local {L}anglands and
  {J}acquet-{L}anglands correspondences.
\newblock Preprint, arXiv:1303.5795, 2013.

\bibitem[Del77]{sga4.5}
P.~Deligne.
\newblock {\em Cohomologie \'etale}.
\newblock Lecture Notes in Mathematics, Vol. 569. Springer-Verlag, Berlin,
  1977.
\newblock S{\'e}minaire de G{\'e}om{\'e}trie Alg{\'e}brique du Bois-Marie SGA
  4$\frac{1}{2}$, Avec la collaboration de J. F. Boutot, A. Grothendieck, L.
  Illusie et J. L. Verdier.

\bibitem[Del80]{De80}
P.~Deligne.
\newblock La conjecture de {W}eil {II}.
\newblock {\em Publ. Math. IHES}, (52):137--252, 1980.

\bibitem[dJ98]{deJongCrystallineDieudonneModuleTheory}
A.~J. de~Jong.
\newblock {C}rystalline {D}ieudonn\'e module theory via formal and rigid
  geometry.
\newblock {\em Inst. Hautes \'Etudes Sci. Publ. Math.}, (87):175, 1998.

\bibitem[DL76]{deligne-lusztig}
P.~Deligne and G.~Lusztig.
\newblock Representations of reductive groups over finite fields.
\newblock {\em Ann. of Math. (2)}, 103(1):103--161, 1976.

\bibitem[FF11]{FarguesFontaine}
Laurent Fargues and Jean-Marc Fontaine.
\newblock Courbes et fibr\'es vectoriels en th\'eorie de {H}odge $p$-adique.
\newblock 2011.

\bibitem[GH94]{GrossHopkins}
B.~H. Gross and M.~J. Hopkins.
\newblock Equivariant vector bundles on the {L}ubin-{T}ate moduli space.
\newblock In {\em Topology and representation theory ({E}vanston, {IL}, 1992)},
  volume 158 of {\em Contemp. Math.}, pages 23--88. Amer. Math. Soc.,
  Providence, RI, 1994.

\bibitem[Hen92]{Henniart-MathNachr1992}
Guy Henniart.
\newblock Correspondance de {L}anglands-{K}azhdan explicite dans le cas non
  ramifi\'e.
\newblock {\em Math. Nachr.}, 158:7--26, 1992.

\bibitem[Hen93]{Henniart-JLC-I}
Guy Henniart.
\newblock Correspondance de {J}acquet-{L}anglands explicite. {I}. {L}e cas
  mod\'er\'e de degr\'e premier.
\newblock In {\em S\'eminaire de {T}h\'eorie des {N}ombres, {P}aris, 1990--91},
  volume 108 of {\em Progr. Math.}, pages 85--114. Birkh\"auser Boston, Boston,
  MA, 1993.

\bibitem[HT01]{HarrisTaylor:LLC}
Michael Harris and Richard Taylor.
\newblock {\em The geometry and cohomology of some simple {S}himura varieties},
  volume 151 of {\em Annals of Mathematics Studies}.
\newblock Princeton University Press, Princeton, NJ, 2001.
\newblock With an appendix by Vladimir G. Berkovich.

\bibitem[Hub94]{HuberAdicSpaces}
R.~Huber.
\newblock A generalization of formal schemes and rigid analytic varieties.
\newblock {\em Math. Z.}, 217(4):513--551, 1994.

\bibitem[Kaz84]{Kazhdan-on-lifting}
David Kazhdan.
\newblock On lifting.
\newblock In {\em Lie group representations, {II} ({C}ollege {P}ark, {M}d.,
  1982/1983)}, volume 1041 of {\em Lecture Notes in Math.}, pages 209--249.
  Springer, Berlin, 1984.

\bibitem[LT65]{LubinTate}
Jonathan Lubin and John Tate.
\newblock Formal complex multiplication in local fields.
\newblock {\em Ann. of Math. (2)}, 81:380--387, 1965.

\bibitem[Sch12]{ScholzePerfectoidSpaces}
Peter Scholze.
\newblock Perfectoid spaces.
\newblock {\em Publ. math. de l'IHÉS}, 116(1):245--313, 2012.

\bibitem[SGA73]{SGA4}
{\em Th\'eorie des topos et cohomologie \'etale des sch\'emas. {T}ome 3}.
\newblock Lecture Notes in Mathematics, Vol. 305. Springer-Verlag, Berlin,
  1973.
\newblock S{\'e}minaire de G{\'e}om{\'e}trie Alg{\'e}brique du Bois-Marie
  1963--1964 (SGA 4), Dirig{\'e} par M. Artin, A. Grothendieck et J. L.
  Verdier. Avec la collaboration de P. Deligne et B. Saint-Donat.

\bibitem[SW12]{ScholzeWeinstein}
Peter Scholze and Jared Weinstein.
\newblock Moduli of $p$-divisible groups.
\newblock 2012.

\bibitem[Tat79]{TateNumberTheoreticBackground}
J.~Tate.
\newblock Number theoretic background.
\newblock In {\em Automorphic forms, representations and {$L$}-functions
  ({P}roc. {S}ympos. {P}ure {M}ath., {O}regon {S}tate {U}niv., {C}orvallis,
  {O}re., 1977), {P}art 2}, Proc. Sympos. Pure Math., XXXIII, pages 3--26.
  Amer. Math. Soc., Providence, R.I., 1979.

\bibitem[Wei10]{WeinsteinGoodReduction}
Jared Weinstein.
\newblock Good reduction of affinoids on the {L}ubin-{T}ate tower.
\newblock {\em Documenta Mathematica}, (15):981--1007, 2010.

\bibitem[Wei12]{WeinsteinSemistableModels}
Jared Weinstein.
\newblock Semistable models for modular curves of arbitrary level.
\newblock 2012.

\bibitem[Yos10]{yoshida}
Teruyoshi Yoshida.
\newblock On non-abelian {L}ubin-{T}ate theory via vanishing cycles.
\newblock {\em Advanced Studies in Pure Mathematics}, (58):361--402, 2010.

\end{thebibliography}

\end{document}